\UseRawInputEncoding
\documentclass[12pt, reqno]{amsart}
\usepackage[margin=1in]{geometry}
\usepackage[mathscr]{eucal}
\usepackage{bm}
\usepackage{mathptmx}
\usepackage{amssymb}
\usepackage{amsthm}
\usepackage{dcolumn}
\usepackage[all]{xy}
\usepackage{enumitem}
\usepackage[hidelinks]{hyperref}

\def \qed {\hfill \vrule height6pt width 6pt depth 0pt}
\def\textmatrix#1&#2\\#3&#4\\{\bigl({#1 \atop #3}\ {#2 \atop #4}\bigr)}
\def\dispmatrix#1&#2\\#3&#4\\{\left({#1 \atop #3}\ {#2 \atop #4}\right)}

\newcommand{\beg}{\begin{equation}}
        \newcommand{\eeg}{\end{equation}}
\newcommand{\ben}{\begin{eqnarray*}}
        \newcommand{\een}{\end{eqnarray*}}

\newtheorem{thm}{Theorem}[section]
\newtheorem{cor}[thm]{Corollary}
\newtheorem{lem}[thm]{Lemma}

\newtheorem{prop}[thm]{Proposition}
\numberwithin{equation}{section} \theoremstyle{definition}
\newtheorem{defn}[thm]{Definition}
\newtheorem{rem}[thm]{Remark}

\newtheorem{eg}[thm]{Example}
\def\textmatrix#1&#2\\#3&#4\\{\bigl({#1 \atop #3}\ {#2 \atop #4}\bigr)}
\def\dispmatrix#1&#2\\#3&#4\\{\left({#1 \atop #3}\ {#2 \atop #4}\right)}
\newcommand{\X}{\mathbb{X}}
\newcommand{\Y}{\mathbb{Y}}
\newcommand{\Z}{\mathbb{Z}}

\newcommand{\N}{\mathbb{N}}
\newcommand{\D}{\mathbb{D}}

\newcommand{\SL}{\mathsf{L}}
\newcommand{\SR}{\mathsf{R}}
\newcommand{\Ann}{\mathsf{Ann}}

\newcommand{\J}{\mathbb{J}}
\newcommand{\C}{\mathbb{C}}
\newcommand{\BL}{\mathbb{L}}
\newcommand{\HS}{\mathcal{H}}
\newcommand{\KS}{\mathcal{K}}

\title[A dilation theoretic approach to Banach spaces]{A dilation theoretic approach to Banach spaces}

\author[Jana, Pal and Roy]{Swapan Jana, Sourav Pal and Saikat Roy}

\address[Swapan Jana]{Mathematics Department, Indian Institute of Technology Bombay,
		Powai, Mumbai - 400076, India.} \email{ swapan.jana@iitb.ac.in , swapan.math2015@gmail.com}

\address[Sourav Pal]{Mathematics Department, Indian Institute of Technology Bombay,
		Powai, Mumbai - 400076, India.} \email{sourav@math.iitb.ac.in , souravmaths@gmail.com}
	
	\address[Saikat Roy]{Mathematics Department, Indian Institute of Technology Bombay, Powai, Mumbai-400076, India.} \email{saikatroy@math.iitb.ac.in , saikatroy.cu@gmail.com}		
	
	
	
\begin{document}

\begin{abstract}

For a complex Banach space $\mathbb X$, we prove that $\mathbb X$ is a Hilbert space if and only if every strict contraction $T$ on $\mathbb X$ dilates to an isometry if and only if for every strict contraction $T$ on $\mathbb X$ the function $A_T: \mathbb X \rightarrow [0, \infty]$ defined by $A_T(x)=(\|x\|^2 -\|Tx\|^2)^{\frac{1}{2}}$ gives a norm on $\mathbb X$. We also find several other necessary and sufficient conditions in this thread such that a Banach sapce becomes a Hilbert space. We construct examples of strict contractions on non-Hilbert Banach spaces that do not dilate to isometries. Then we characterize all strict contractions on a non-Hilbert Banach space that dilate to isometries and find explicit isometric dilation for them. We characterize complemented subspaces of a reflexive, smooth and strictly convex Banach space in terms of duality of subspaces and linearity of Hahn Banach extension operator. We obtain characterizations for the isometries on a reflexive Banach space whose ranges are orthogonally complemented subspaces and thus answer a question posed by Faulkner and Huneycutt. Since bilateral shift on a Banach space is not in general a bounded operator, we define $\sigma$-shift meaningfully which is a Banach space unitary and show that a unilateral shift on a Banach space extends to a $\sigma$-shift. Then we show that a $\sigma$-shift becomes a bilateral shift if the underlying Banach space is a Hilbert space. We show that the spectrum of a $\sigma$-shift acting on a smooth Banach space is the whole unit circle $\mathbb T$. Also, we prove that every Wold isometry on a Banach space extends to a Banach space unitary. We determine the norm attainment set for every Banach sapce operator $T$ such that $T/\|T\|$ dilates to a Banach space isometry. We define a new adjoint $T_*$ for a Banach space operator $T$ and show that a reflexive, smooth and strictly convex Banach space $\mathbb X$ of dimension greater than $2$ is a Hilbert space if and only if $T_*$ is linear for every rank one $T \in \mathcal B(\mathbb X)$. Finally, we find a canonical decomposition for a Banach space contraction $T$ that splits $T$ orthogonally into two parts of which one is a Banach space unitary and the other is a completely non-unitary Banach space contraction.

\end{abstract}

\maketitle

\tableofcontents

\section{Introduction}

\vspace{0.2cm}

\noindent
We begin with a few basic definitions, notations and termonilogies that we shall follow throughout the paper. Unless and otherwise stated, all Banach spaces and Hilbert spaces are over the field of complex numbers $\mathbb{C}$. The collection of all unit vectors in a Banach space $\X$ is denoted by $S_\X$. By a subspace, we always mean a linear subspace. The term `operator' stands for a bounded linear operator and the linear operators that are not bounded will be mentioned separately. An operator $T$ on a Banach space $\X$ is said to be a \textit{contraction} (or, a \textit{strict contraction}) if $\|T\|\leq 1$ (or, $\|T\|< 1$). A \textit{norm-one} operator is an operator whose norm is equal to $1$. The Banach adjoint of $T$ is denoted by $T^\times$. For two normed spaces $\X_1, \X_2$, their \textit{orthogonal $2$-sum} or \textit{$2$-direct sum} is the normed space $\X_1 \oplus_2 \X_2$ that consists of vectors of the form $(x_1,x_2)$ or $x_1 \oplus_2 x_2$, where $x_1 \in \X_1, \, x_2 \in \X_2$ with $\|(x_1,x_2)\|=(\|x_1\|^2 + \| x_2 \|^2)^{\frac{1}{2}}$. Orthogonal direct sum of two Hilbert spaces $\HS_1, \HS_2$ will also be denoted by $\HS_1 \oplus_2 \HS_2$, though the standard practice is $\HS_1 \oplus \HS_2$. If $\X_1,\X_2$ are any two normed spaces over the same filed, then the space $\X_1 \oplus \X_2$ stands for the vector space direct sum of $\X_1, \X_2$. For two Banach spaces $\X$ and $\Y$, an \textit{isometry} is a linear map $V:\, \X \rightarrow \Y$ that satisfies $\|Vx\|=\|x\|$ for all $x\in \X$. A \textit{Banach space unitary} is a surjective isometry. Given any two elements $x$ and $y$ in a Banach space $\X$, $x$ is said to be orthogonal to $y$, written as $x\perp_B y$, if $x$ is orthogonal to $y$ in the sense of Birkhoff-James \cite{Birkhoff, James 1}, i.e.  if $\|x+\lambda y\|\geq \|x\|$ for all scalars $\lambda$. For a Banach space $\X$, the space $\ell_2(\X)$ consists of sequences $(x_n)$ of vectors from $\X$ such that ${\displaystyle \Sigma_{n=1}^{\infty} \ \|x_n\|^2 < \infty}$. The \textit{forward shift operator} $M_z$ on $\ell_2(\X)$ maps a vector $(x_1, x_2, x_3, \dots) $ to $(\mathbf{0}, x_1, x_2, \dots)$ and is a Banach space isometry. Similarly, the \textit{backward shift operator} $\widehat{M}_{z}$ on $\ell_2(\X)$ maps $(x_1, x_2, x_3, \dots) $ to the vector $(x_2, x_3, x_4, \dots)$.

\medskip

Isometric dilation of a Hilbert space contraction is a fundamental concept in operator theory which is defined in the following way.

\begin{defn} \label{def:new-001A}
Suppose $T$ is a contraction on a Hilbert space $\mathcal{H}.$ An isometry $V$ on a Hilbert space $\mathcal{K}$ is said to be an \textit{isometric dilation} of $T$ if there exists an isometry $W:\mathcal{H}\to \mathcal{K}$ such that
\begin{equation} \label{eqn:001}
q(T)=W^* q(V)W \ \ \text{ or equivalently } \ \ q(\widehat{T})= P_{_{W(\HS)}} q(V)|_{W(\HS)}
\end{equation}
for all polynomials $q\in \C[z]$, where $\widehat{T}$ is defined by $\widehat{T}:= \widehat{W} T \widehat{W}^{-1}: W(\HS) \to W(\HS)$ with $\widehat{W}$ being the unitary $\widehat{W}:= W:\HS \to W(\HS)$ and $P_{_{W(\HS)}}$ being the orthogonal projection of $\KS$ onto $W(\HS)$. Such an isometric dilation is called \textit{minimal} if
\[
 \mathcal{K} = \bigvee_{n=0}^\infty V^n W(\mathcal{H}) = \overline{span}\{V^nWh:~h\in \mathcal{H},~n\geq 0\}.
\]
\end{defn}

A celebrated theorem due to Bela Sz. Nagy \cite{Nagy} states that every Hilbert space contraction possesses a minimal isometric dilation. Moreover, a minimal isometric dilation of a contraction is unique up to unitary equivalence, e.g. see \cite{Pisier, Nagy, Nagy Foias}. Thus, every contraction on a Hilbert space can be realized as a compression of an isometry acting on a bigger Hilbert space. Also, the space $\KS$ is isomorphic with the orthogonal direct sum $W(\HS) \oplus_2 \BL$, where $\BL$ is the orthogonal complement of $W(\HS)$ in $\KS$. The first obstruction in the way of defining dilation for a Banach space contraction is that we do not have an adjoint for the isometry $W$. Interestingly, (\ref{eqn:001}) shows that isometric dilation of a Hilbert space contraction can be defined in an alternative and equivalent way by avoiding the Hilbert space adjoint of $W$. If we want to adopt this definition for dilation in Banach space, then the next issue is that a Banach space projection can have norm strictly greater than $1$ which can cause imbalance in norm in either side of (\ref{eqn:001}). However, in Banach space setting if the bigger space $\KS$ is chosen to be isomorphic with $W(\HS) \oplus_2 \BL$ for some Banach space $\BL$, then the projection of $\KS$ onto $W(\HS)$ as in (\ref{eqn:001}) becomes a norm-one projection and the norm-issue is resolved. Thus, to get rid of all such hindrances we define isometric dilation for a Banach space contraction in the following way which generalizes the Hilbert space dilation.

\begin{defn}\label{Definition of Dilation of contractions on Banach space-1}
Suppose $T$ is a contraction acting on a Banach space $\mathbb{X}$. An isometry $V$ on a Banach space $\widetilde{\mathbb{X}}$ is said to be an \textit{isometric dilation} of $T$ if there is an isometry $W:\mathbb{X}\to \widetilde{\mathbb{X}}$ and a closed linear subspace $\BL$ of $\widetilde{\X}$ such that $\widetilde{\mathbb{X}}$ is isomorphic with $ W(\mathbb{X}) \oplus_2 \BL$ and the operator $\widehat{T}:= \widehat{W} T \widehat{W}^{-1}: W(\X) \to W(\X)$ satisfies
\[
q(\widehat{T})= P_{_{W(\mathbb{X})}} q(V)|_{W(\X)}
\]
for all polynomials $q\in \mathbb{C}[z]$, where $\widehat{W}$ is the unitary (i.e. surjective isometry) $\widehat{W}:= W:\mathbb{X}\to W(\mathbb{X})$ and $P_{_{W(\mathbb{X})}}$ is the norm-one projection of $\widetilde{\X}$ onto $W(\mathbb{X})$. Moreover, such an isometric dilation is called \textit{minimal} if
\[
 \widetilde{\X} = \bigvee_{n=0}^\infty V^n W(\X) = \overline{span}\{V^nWx:~x\in \X,~n\geq 0\}.
\]
 
\end{defn}

In Section \ref{sec:NEW-02}, we explicitly justify Definition \ref{Definition of Dilation of contractions on Banach space-1} with proofs as our approach differs from the previous notions of Banach space dilation, e.g., see \cite{Stroescu, AE, AL 1, AL 2, Peller, SJ}. The fact that every Hilbert space contraction dilates to a Hilbert space isometry makes the following three statements equivalent: $T$ is a contraction on a Hilbert space $\HS$; the operator $I-T^*T$ is positive definite; $T$ is an operator (on $\HS$) that admits an isometric dilation. Also, an easy computation shows that if $T$ is a strict contraction, then $I-T^*T$ is positive definite if and only if the map $A_T:\HS \to [0, \infty)$ given by $h \mapsto (\|h\|^2-\|Th\|^2)^{\frac{1}{2}}$ defines a norm on $\HS$. The reason behind dealing with strict contractions only is that if $T$ is a contraction with $\|T\|=1$, then $A_T$ does not define a norm if $T$ attains its norm at a point, i.e. if $\|Tx\|=\|x\|$ for a nonzero vector $x$. The minimal isometric dilation of a strict contraction is more subtle. The unilateral shift $M_z$ on $\ell_2(\HS)$ happens to be the minimal isometric dilation for every strict contraction on a Hilbert space $\HS$, see \cite{Nagy Foias} for details. However, we shall see that in general a Banach space contraction may not dilate to a Banach space isometry. Indeed, if every strict contraction $T$ on a Banach space $\X$ dilates to a Banach space isometry or equivalently if the map $A_T$ defines a norm on $\X$ for every strict contraction $T$, then $\X$ becomes a Hilbert space and vice-versa. This is one of our main results and is stated below.

\begin{thm}\label{thm:main-01}
Let $\mathbb{X}$ be a complex Banach space. Then the following are equivalent.

\smallskip

\begin{enumerate}

\item[(i)] Every strict contraction $T$ on $\X$ dilates to the unilateral shift $M_z$ on $\ell_2(\mathbb{X}).$

\smallskip

\item[(ii)] Every strict contraction $T$ on $\X$ dilates to an isometry.

\smallskip

\item[(iii)]  For every strict contraction $T\in \mathcal{B}(\mathbb{X}),$ the function $A_T: \mathbb{X}\to [0,\infty)$ given by 
\[
  A_T(x) = \left( \| x\|^2 - \|Tx\|^2 \right)^{\frac{1}{2}}, \quad x\in \mathbb{X},
\]
defines a norm on $\mathbb{X}.$

\smallskip

\item[(iv)] $\mathbb{X}$ is a Hilbert space.

\smallskip

\item[(v)] For every strict contraction $T\in \mathcal{B}(\mathbb{X}),$ there exists an isometry $W: \mathbb{X}\to \ell_2(\mathbb{X})$ such that 
\[
 \widehat{M}_{z} W(x)= W(Tx), \quad x\in \mathbb{X},
\]
where $\widehat{M}_{z}$ is the backward shift operator.

\smallskip

\item[(vi)] For every strict contraction $T \in \mathcal B(\X)$ and for every automorphism $\phi_{\lambda}$ of the unit disk $\mathbb D$, the operator $\phi_\lambda(T)$ is a contraction.

\smallskip

\item[(vii)] For every strict contraction $S$ on the dual space $\mathbb{X}^*,$ the function $A_{S}:\mathbb{X}^*\to [0,\infty)$ given by 
\[
   A_S(x^*) = \left( \|x^*\|^2 - \|S(x^*)\|^2 \right)^{\frac{1}{2}},\quad x^*\in \mathbb{X}.
\]
defines a norm on $\mathbb{X}^*.$

\smallskip

\item[(viii)] $\mathbb{X}$ is reflexive and the Banach adjoint $T^\times$ of $T$ dilates to $M_z$ on $\ell_2(\mathbb{X}^*)$ for every strict contraction $T$ on $\mathbb{X}$.

\smallskip

\item[(ix)] The operator $\phi_\alpha(U)= (U-\alpha I)(I-\bar{\alpha}U)^{-1}$ is a contraction for every automorphism $\phi_\alpha$ of the unit disk $ \mathbb{D}$, $U$ being the bilateral shift operator on $\ell_2(\mathbb{Z},\mathbb{X})$ defined by
\[
U((\dots, x_{-2}, x_{-1}, \boxed{x_0}, x_1, x_2, \dots))= (\dots, x_{-2}, \boxed{x_{-1}}, x_0, x_1, x_2, \dots),
\]
where the box on either side indicates the $0$-th position.
\end{enumerate}
\end{thm}

This is Theorem \ref{Characterisation of Hilbert spaces in terms of Dilation on Banach space} in this paper and the definition of Banach space dilation (i.e., Definition \ref{Definition of Dilation of contractions on Banach space-1}) plays a crucial role in the proof of this result. So, we learn from Theorem \ref{thm:main-01} that a strict contraction on a non-Hilbert Banach space (i.e. a Banach space which is not a Hilbert space) may not dilate to an isometry. One naturally asks if we can characterize the class of strict contractions on a non-Hilbert Banach space that possess isometric dilation. The following theorem answers this question.

\begin{thm}\label{thm:main-02} Suppose $\X$ is a complex Banach space. Then a strict contraction $T$ on $\X$ dilates to an isometry if and only if the function $A_T: \mathbb{X}\to [0,\infty)$ given by $
  A_T(x) = \left( \| x\|^2 - \|Tx\|^2 \right)^{\frac{1}{2}}$
defines a norm on $\mathbb{X}.$ Moreover, the minimal isometric dilation space of $T$ is isometrically isomorphic to $\mathbb{X}\oplus_2 \ell_2(\mathbb{X}_0)$, where $\mathbb{X}_0$ is the Banach space $(\mathbb{X}, A_T).$
\end{thm}

This is Theorem \ref{Necessary and sufficient condition for a contraction to be dilated 1} in this paper. Also, in this theorem we explicitly construct a minimal isometric dilation for a strict Banach space contraction. We learn from the literature (e.g. see Chapter 1 of \cite{Nagy Foias}) that a major role is played by the defect operator $D_T=(I-T^*T)^{\frac{1}{2}}$ in the construction of isometric dilation of a Hilbert space contraction $T$. Since we do not have adjoint for a Banach space contraction, naturally we would like to utilize the map $A_T$ in order to avoid the operator pencil $I-T^*T$ while constructing dilation for a Banach space contraction. Interestingly, more is true about the map $A_T$. If it induces a norm, then there is an operator $A$ such that $\|Ax\|=A_T(x)$ for every $x \in \X$ and vice-versa, see Corollary \ref{cor:main-01}. However, in Example \ref{exmp:main-01} we formulate a strict contraction $T$ on a non-Hilbert Banach space such that $A_T$ is not a norm or equivalently in view of Theorem \ref{thm:main-02}, $T$ does not dilate to any Banach space isometry.

\smallskip

An isometry $V$ on a Hilbert space $\HS$ orthogonally splits into two parts of which one is a unitary and the other is a unilateral shift, i.e. $\HS$ admits an orthogonal decomposition $\HS= \HS_0 \oplus_2 \HS_1$ into reducing subspaces $\HS_0, \HS_1$ of $V$ such that $V_0=V|_{\HS_0}$ is a unitary and $V_1=V|
_{\HS_1}$ is isomorphic to a unilateral shift. In the literature (e.g. see \cite{Wold}), this is known as Wold decomposition of a Hilbert space isometry. Also, every unilateral shift on a Hilbert space naturally extends to a bilateral shift with same multiplicity. Hence, every Hilbert space isometry extends to a Hilbert space unitary. Wold decomposition of a Hilbert space isometry was further generalized to canonical decomposition of a Hilbert space contraction. Indeed, every Hilbert space contraction $T$ decomposes into an orthogonal direct sum of a unitary and a completely non-unitary (c.n.u.) contraction, e.g. see \cite{Langer} or Chapter-I of \cite{Nagy Foias}. A c.n.u. contraction on a Hilbert space is a contraction that is not a unitary on any of its nonzero reducing subspaces.

\smallskip

The main hindrance in studying such results in Banach space setting is dealing with orthogonal complement of a closed subspace. The idea of orthogonal complement of a closed subspace in a Hilbert space was generalized to Banach spaces through the notion of $1$-complemented subspace. A closed subspace of a Banach space is said to be \textit{orthogonally complemented} or $1$-\textit{complemented} \cite{Kinnunen, Moslehian, R 1, R 2} if it is the range of a norm-one projection. However, unlike every closed subspace of a Hilbert space, a closed subspace of a Banach space may not be $1$-complemented. In fact, the Banach space isometries that have $1$-complemented ranges are not fully known till date even though there is an extensive research on this particular problem, e.g. see \cite{Ando, CFS, CFG, Ditor, Faulkner Huneycutt, F, LL, Moslehian}. Here we characterize all $1$-complemented subspaces of a reflexive, smooth and strictly convex Banach space and show that the study of $1$-complemented subspaces of a Banach space has connections with important topics like duality of subspaces, linearity of the Hahn-Banach extension operator, which were not known before. This is also a main result of this article.

\begin{thm}\label{Characterization of Right-complemented}
Let $\X$ be a reflexive, smooth and strictly convex Banach space and let $\Y$ be a proper closed subspace of $\X$. Then the following are equivalent.

\begin{itemize}
    \item[(i)] $\Y$ is $1$-complemented in $\X$.
    
    \smallskip
    
    \item[(ii)] The closed linear subspace 
    \[
     \bigcap_{f\in \X^*}\{\ker~f:~M_f\subseteq S_\Y\} 
    \]
    is a vector space complement of $\Y$, where $M_f=\{y\in S_\X:~|f(y)|=\|f\|\}$.
    
    \smallskip
    
    \item[(iii)] $\J_\X(\Y)$ is isometrically isomorphic to $\Y^*$, where $\J_\X(\Y)$ is as in $(\ref{eqn:new-001})$.
    
    \smallskip
    
    \item[(iv)] The Hahn-Banach extension operator $\Psi:\Y^*\to \X^*$ is linear. 
\end{itemize}
\end{thm}

This is Theorem \ref{Characterization of Right-complemented} in this paper. In Section \ref{sec:02}, we shall define left-complemented and right-complemented subspaces in terms of Birkhoff-James orthogonality and show that a $1$-complemented subspace is just a right-complemented subspace and in a similar way a left-complemented subspace is the kernel of a norm-one projection. In Theorem \ref{Characterization of left-complemented}, we provide several characterizations of left-complemented subspaces of a reflexive, smooth and strictly convex Banach space.

\smallskip

In 1966, Ando \cite{Ando} proved that the range of an isometry on the space $\mathcal{L}_p(\Omega, M, \mu)$ for $1 < p < \infty$ is always orthogonally complemented if $\mu$ is a finite measure. Needless to mention that for $1< p < \infty$, the space $\mathcal{L}_p(\Omega, M, \mu)$ is a reflexive, smooth and strictly convex Banach space. In 1973, Ditor \cite{Ditor} constructed a closed linear subspace $\Y$ of the non-reflexive Banach space $(C[0,1], \|\cdot\|_\infty)$ such that $\Y$ is isometrically isomorphic with $(C[0,1], \|\cdot\|_\infty)$ but $\Y$ is not orthogonally complemented in $(C[0,1], \|\cdot\|_\infty)$. In 1978, Faulkner and Huneycutt \cite{Faulkner Huneycutt} obtained a Wold-type decomposition for an isometry having $1$-complemented range and acting on a smooth and reflexive Banach space. The result of Faulkner and Huneycutt was further generalized to reflexive Banach spaces removing the smoothness assumption by Campbell, Faulkner and Sine in \cite{CFS}. In the same paper \cite{Faulkner Huneycutt}, Faulkner and Huneycutt asked if there are Banach space isometries without $1$-complemented ranges (see Question 2 at the end of \cite{Faulkner Huneycutt}). Recently, Pelczar-Barwacz \cite{P-B} constructed a reflexive Banach space $\X$ and a closed linear subspace $\Y$ of $\X$ such that $\Y$ is isometrically isomorphic with $\X$ but is not $1$-complemented in $\X$; in fact, $\Y$ is not even complemented in $\X$. Thus, going back and forth with such $\X$ and $\Y$, one can easily define an isometry on a reflexive Banach space whose range is not $1$-complemented.

\smallskip

In Theorem \ref{Left Inverse of Wold Isometry}, we characterize an isometry on a reflexive Banach space whose range is $1$-complemented (i.e. right-complemented) and hence answer the question raised by Faulkner and Huneycutt in \cite{Faulkner Huneycutt} for such Banach spaces. This is also an important result of this article and is stated below.

\begin{thm}
Let $\mathbb{X}$ be a reflexive Banach space and let $V$ be an isometry on $\X$. Then the following are equivalent:
\item[(i)] $V$ has right-complemented range ;

\item[(ii)] There is an operator $T: \mathbb{X} \to \mathbb{X}$ with $\|T\| =1$ such that $TV = I$.
\end{thm}
Also, under the hypotheses of reflexivity, smoothness and strict convexity, we apply Theorem \ref{Characterization of Right-complemented} and obtain in Theorem \ref{Isometry Range} a few different characterizations for an isometry with $1$-complemented range. 

\smallskip

 Wold decomposition of a Banach space isometry was previously studied in \cite{Faulkner Huneycutt} assuming its range being $1$-complemented. In the same article, the authors introduced the notion of unilateral shift on a Banach space. Bilateral shifts on Banach spaces are not as well-behaved as they are on Hilbert spaces. In fact, a bilateral shift on a Banach space may not be even bounded, e.g. see \cite{CFS, Gellar}. So, in Banach space setting we do not have a meaningful generalization of the fact that a Hilbert space unilateral shift always extends to a Hilbert space bilateral shift. Also, to the best of our knowledge there is no decomposition result for Banach space contractions that is analogous to the canonical decomposition of a Hilbert space contraction.
 
\smallskip

In Section \ref{sec:05}, we define $\sigma$-shift on a Banach space, which is a unitary and is an analogue of bilateral shift on a Hilbert space. In Proposition \ref{prop:011}, we show that a $\sigma$-shift is a bilateral shift if the underlying Banach space is a Hilbert space. In Theorem \ref{Sigma shift extension}, we prove that every unilateral shift on a Banach space can be extended to a $\sigma$-shift and thus generalize the same result in Hilbert space setting. Then, in Theorem \ref{thm:general-extension-01} we show that every Wold isometry on a Banach space extends to a Banach space unitary. Finally, in Proposition \ref{Spectrum of sigma shift operator} we show that the spectrum of a $\sigma$-shift acting on a smooth Banach space is the whole unit circle $\mathbb T$.

\smallskip

Section \ref{sec:07} is all about isometric dilation of a Banach space contraction and characterizations for Hilbert spaces. Here we characterize all strict contractions on Banach spaces that dilate to Banach space isometries and construct explicit dilations. Also, we show that all strict contractions on a Banach space dilate to isometries if and only if the Banach space is a Hilbert space. We have already mentioned the main results of this Section (see Theorems \ref{thm:main-01} \& \ref{thm:main-02}). A special emphasis is given on contractions on $\mathcal L_p$ spaces. Indeed, we present a new proof to the fact that $\mathcal{L}_p$ ($1\leq p\leq \infty$) is a Hilbert space if and only if $p=2$.

\smallskip

In Subsection \ref{Subsection:new-6.1}, we study all norm-one Banach space contractions that dilates to an isometry. In Theorem \ref{thm:new-6.26}, we show that a norm-one contraction $T$ dilates to an isometry acting on a normed (or, semi-normed linear) space if and only if the function $A_T$ as in Theorem \ref{thm:main-02} defines a norm (or, semi-norm) on $\X$. Moreover, we also show that in either case the minimal isometric dilation space cannot be a Banach space.

\smallskip

In Subsection \ref{Subsection:new-6.2}, we define a new adjoint $T_*$ for every operator $T$ acting on a reflexive, smooth and strictly convex Banach space. This definition generalizes the notion of adjoint for Hilbert space operators. With this notion of adjoint, we prove in Theorem \ref{Characterisation of Hilbert space by rank one operators} that a reflexive, smooth and strictly convex Banach space $\X$ with dimension greater than $2$ is a Hilbert space if and only if $T_*$ is linear for every rank one operator $T$ on $\X$.

\smallskip

In Section \ref{Sec:08}, we describe the norm attainment set of all Banach space operators which are scalar multiples of norm-one contractions that possess isometric dilations. This is Theorem \ref{thm:623} in this paper and even more is true about this theorem. Indeed, Theorem \ref{thm:623} shows that the norm attainment set of such an operator (i.e., scalar multiple of a norm-one contraction that admits isometric dilation) is the unit sphere of some subspace of the underlying Banach space. Needless to mention that this theorem is valid for all Hilbert space operators.

\smallskip

 The main result of Section \ref{sec:06} is Theorem \ref{Canonical Decomposition Banach space}, which provides a canonical decomposition for specific Banach space contractions. Then we illustrate this result with suitable examples.
 
\smallskip

Section \ref{sec:02} describes background materials that are necessary for this paper along with some preparatory results. Also, here we develop a few new terminologies that will be followed throughout the paper. 

\smallskip

\section{A brief literature and justification for Definition \ref{Definition of Dilation of contractions on Banach space-1}} \label{sec:NEW-02}

\smallskip

\noindent In this Section, we clarify the validity of Definition \ref{Definition of Dilation of contractions on Banach space-1}, i.e., the definition of isometric dilation of a Banach space contraction and establish with proofs its effectiveness. The two main results, Theorems \ref{thm:main-01} \& \ref{thm:main-02}, and their consequences are dependent on this definition. Dilation of a Banach space contraction has been studied in past in different contexts, sometimes for positive or positively dominated contractions on $\mathcal L_p$ spaces, e.g., \cite{AE, AL 1, AL 2, Peller}, sometimes for general Banach space contractions \cite{SJ, Stroescu} and also for sectorial operators \cite{NLW}. In all these references, the dilation space $\widetilde{\X}$ for a contraction $T$ on a Banach space $\X$ is chosen to be isomorphic with $W(\X) \oplus \BL$, where $\BL$ is another appropriate Banach space and $W:\X \rightarrow \widetilde{\X}$ is the isometric embedding. We prefer to choose $\widetilde{\X}$ to be isomorphic with $W(\X) \oplus_2 \BL$ in Definition \ref{Definition of Dilation of contractions on Banach space-1} and the reasons behind this choice are evident from its consequences as explained below. 

\smallskip

The famous theorem due to Bela Sz. Nagy \cite{Nagy} about Hilbert space dilation along with von Neumann's path-breaking result \cite{JVNII} on Hilbert space contraction versus spectral set made the following statements equivalent.

\begin{thm}\label{thm:401}
For a Hilbert space operator $T$ the following are equivalent:
\begin{enumerate}
\item[(i)] $T$ is a contraction;
\item[(ii)] $T$ dilates to an isometry;
\item[(iii)] $\overline{\D}$ is spectral set of $T$.
\end{enumerate}
\end{thm}
On the other hand, the following celebrated theorem by Foias \cite{CFI} reveals an interesting interplay between Banach space contractions and the notion of spectral set.
\begin{thm} \label{thm:Foias-New}
A complex Banach space $\X$ is a Hilbert space if and only if the closed unit disk $\overline{\D}$ is a spectral set for all $T\in \mathcal{B}(\X)$ with $\|T\| \leq 1$.
\end{thm} 
In view of these two theorems together with a few results from the literature (accumulated below), we place our logics behind the choice of dilation space in Definition \ref{Definition of Dilation of contractions on Banach space-1} through the following points.

\begin{enumerate}
\item The definition of Banach space dilation, i.e., Definition \ref{Definition of Dilation of contractions on Banach space-1} generalizes the definition of Hilbert space dilation (see Definition \ref{def:new-001A}) as is clear from Equation-(\ref{eqn:001}) and is valid for all Banach space contractions. Also, if the dilation space is $W(\X) \oplus_2 \BL$, then the projection map $P_{W(\X)}$ (as in Definition \ref{Definition of Dilation of contractions on Banach space-1}) is a norm-one orthogonal projection. Thus, we do not need to assume the existence of such a norm-one projection unlike the previous definitions, e.g., see \cite{Stroescu, AE, AL 1, AL 2, Peller, SJ}. Also, in these references only unitary dilation of a Banach space contraction is considered, whereas in this article we deal with more general isometric dilation and have constructed such dilation in Theorem \ref{Necessary and sufficient condition for a contraction to be dilated 1}.

\smallskip

\item If a Hilbert space contraction $T\in \mathcal B(\mathcal H)$ dilates to an isometry $V$ on $\mathcal K =W(\HS) \oplus_2 W(\HS)^{\perp}$, then $\|P_{W(\HS)^ \perp} V Wh \| = \left(\|h\|^2 - \|Th\|^2\right)^{\frac{1}{2}}$ holds for all $h \in \HS$ as is shown below in Lemma \ref{lem:V6 new 001}. If a Banach space contraction $T \in \mathcal B(\X)$ dilates to a Banach space isometry $V$ as per Definition \ref{Definition of Dilation of contractions on Banach space-1}, then also $\|P_{W(\X)^ \perp} V Wx \| = \left(\|x\|^2 - \|Tx\|^2\right)^{\frac{1}{2}}$ holds for all $x \in \X$ as is shown in Proposition \ref{prop:new-001A}. Thus, we can say that Definition \ref{Definition of Dilation of contractions on Banach space-1} is an appropriate generalization of Hilbert space dilation. However, we shall prove in Proposition \ref{prop:new-002A} that the same is not true in case of the previously obtained Banach space dilations.

\smallskip

\item It is a straight consequence of Theorem \ref{thm:Foias-New} that all strict contractions on a Banach space $\X$ have the closed unit disk $\overline{\D}$ as a spectral set if and only if $\X$ is a Hilbert space. On the other hand, Theorem \ref{thm:401} confirms the equivalence of isometric dilation of a Hilbert space contraction $T$ and $T$ having $\overline{\D}$ as a spectral set. Thus, it is naturally expected that all strict contractions on a Banach space $\X$ admit isometric dilations if and only if $\X$ is a Hilbert space. So, all Banach space contractions should not dilate to Banach space isometries as $\overline{\D}$ is not a spectral set for all Banach space contractions. However, it was shown in \cite{Stroescu, SJ} that every contraction on a Banach space dilates to a Banach space isometry. This happened because the dilation space was taken to be isomorphic with $W(\X)\oplus \BL$ for some appropriate Banach space $\BL$. Our choice of the dilation space $\widetilde{\X}$, which is isomorphic with $ W(\X) \oplus_2 \BL$ as in Definition \ref{Definition of Dilation of contractions on Banach space-1}, provides the desired equivalence (i.e., all strict contractions on $\X$ dilate to isometries if and only if $\X$ is a Hilbert space) as is shown in Theorem \ref{thm:main-01}.

\smallskip

\item The theory of Hilbert space dilation (e.g., see \cite{Nagy Foias}) tells us that the minimal dilation space of a Hilbert space contraction $T \in \mathcal B(\HS)$ is isomorphic with $\HS \oplus_2 \ell_2(\mathcal D_T)$, which becomes $\HS \oplus_2 \ell_2(\HS)$ when $T$ is a strict contraction. Here $\mathcal D_T$ is the range space of $(I-T^*T)^{\frac{1}{2}}$. Note that such minimal dilation space is unique upto unitary equivalence, e.g., see \cite{Nagy, Nagy Foias, Pisier}. In Theorem \ref{thm:main-02} (which is Theorem \ref{Necessary and sufficient condition for a contraction to be dilated 1}), we have shown that if a strict Banach space contraction $T \in \mathcal B(\X)$ dilates to a Banach space isometry, then the minimal dilation space is isomorphic with $\X \oplus_2 \ell_2(\X_0)$. Here $\X_0$ is the Banach space $(\X, A_T)$, where the norm $A_T$ is given by $A_T(x)=(\|x\|^2 -\|Tx\|^2)^{\frac{1}{2}}$,  $x \in X$. Moreover, when $\X$ is a Hilbert space with norm $\|.\|$, then it is unitarily equivalent with $\X_0$. Also, $\X \oplus_2 \ell_2(\X_0)$ is a Hilbert space if and only if $\X$ is a Hilbert space. Thus, our minimal isometric dilation space for Banach space contractions generalizes the same for Hilbert space contractions. This becomes possible only because our dilation is based on Definition \ref{Definition of Dilation of contractions on Banach space-1}. However, the same is not true for the previously obtained dilations, e.g., see Propositions \ref{lem:V6 new 002} \& \ref{thm:V6 new 004}, where we have proved that $\ell_1(\X; \Z)$ and $\ell_1(\X;\N)$ are minimal unitary dilation space and minimal isometric dilation space respectively, for the dilations obtained in \cite{SJ}.

\smallskip

\item In Theorem \ref{Necessary and sufficient condition for a contraction to be dilated 1},
we explicitly construct a Sch\"{a}ffer-type minimal isometric dilation for a strict Banach space contraction, provided it dilates and the construction goes parallel with the minimal dilation for a Hilbert space contraction. A major role is played by the defect operator $D_T=(I-T^*T)^{\frac{1}{2}}$ in the Sch\"{a}ffer-type construction of Hilbert space dilation, e.g., see Chapter 1 of \cite{Nagy Foias}. Since we do not have adjoint for a Banach space contraction, naturally utilize the map $A_T$ in order to avoid the operator pencil $I-T^*T$ in our Sch\"{a}ffer-type Banach space dilation. So, our minimal isometric dilation for a strict Banach space contraction (when it dilates) is a canonical generalization of the Sch\"{a}ffer-type minimal Hilbert space dilation. Needless to mention, Definition \ref{Definition of Dilation of contractions on Banach space-1} plays the main role in the proof of Theorem \ref{Necessary and sufficient condition for a contraction to be dilated 1}.

\smallskip

\item Theorem \ref{Norm if and only if Hilbert space} is a part of Theorem \ref{thm:main-01} and a direct consequence of Theorem \ref{Norm if and only if Hilbert space} is Theorem \ref{thm:new-003A}, which gives an alternative proof to the well-known fact that $\mathcal{L}_p$ is a Hilbert space if and only if $p = 2$. Since the proof of Theorem \ref{thm:main-01} is based on Definition \ref{Definition of Dilation of contractions on Banach space-1}, it can be concluded that Definition \ref{Definition of Dilation of contractions on Banach space-1} is appropriate and is a perfect fit in the setting of Banach space dilation theory. One can easily verify that Theorem \ref{thm:new-003A} cannot be obtained from the dilation theorems of \cite{Stroescu} and \cite{SJ}.
\end{enumerate}

\smallskip

Now we give proofs to the evidences displayed above in support of Definition \ref{Definition of Dilation of contractions on Banach space-1}. To do this we need to recall a few results from the literature on Banach space dilation theory. As mentioned before, dilation for Banach space contractions has been studied in various contexts starting from Stroescu's unitary dilation \cite{Stroescu}, followed by dilation on $\mathcal L_p$ spaces by Akcoglu, Ekkehard, Sucheston \cite{AE, AL 1, AL 2} and Peller \cite{Peller}. Recently, Fackler and Gl\"{u}k \cite{SJ} made an attempt to tie up all these previously obtained dilations in a single thread by defining and providing unitary dilation for a (general) Banach space contraction. Note that their definition of dilation (as in \cite{SJ}) combines and generalizes the previous definitions which were mainly for contractions on $\mathcal L_p$ spaces.

\begin{defn} [Fackler \& Gl\"{u}k, \cite{SJ}] \label{defn:V6 new 001}
An operator $T$ on a Banach space $\X$ has a dilation to a Banach space $\Y$ if there exist an isometry $J:\X \to \Y$ and a contraction $Q:\Y \to \X$, along with a invertible linear isometry $U:\Y \to \Y$ such that 
\begin{equation*}
T^n = QU^nJ, \quad n\in \mathbb{N} \cup \{0\}.
\end{equation*}
\end{defn}
 
In 1973, Stroescu \cite[Corollary 2]{Stroescu} proved the following dilation theorem for a Banach space contraction.
\begin{thm}\label{thm:V6 new 001}
Let $T$ be a contraction on a Banach space $\X$. Then there exists a Banach space $\widetilde{\X}$ containing $\X$, a norm-one projection $P$ of $\widetilde{\X}$ onto $\X$ and an invertible isometry $U$ on $\widetilde{\X}$ such that the following holds:
\begin{enumerate}
\item[(i)] $PU^n x = T^{|n|} x$, for all $x\in \X$,  $n\in \Z$.

\item[(ii)] $\widetilde{\X}$ is the closed vector space spanned by $\{U^n x: n\in \Z, ~ x\in \X \}$.
\end{enumerate}
\end{thm}

In 1977, Akcoglu and Sucheston introduced positive operators between real $\mathcal{L}_p$ spaces in \cite{AL 2}. A linear operator between real $\mathcal{L}_p$ spaces is called \textit{positive} if it preserves the order or, equivalently if it maps non-negative functions into non-negative functions.  In the same paper, they proved the following dilation theorem for positive contractions on real $\mathcal{L}_p$ spaces, see Theorem 1.1 in \cite{AL 2}.

\begin{thm} \label{thm:V6 new 002}
Let $1\leq p < \infty$ and $T:\mathcal{L}_p(\Omega, M, \mu) \to \mathcal{L}_p(\Omega, M, \mu)$ be a positive contraction. Then there exist another $\mathcal{L}_p(\Omega', M', \mu')$ space and a positive invertible isometry $Q:\mathcal{L}_p(\Omega', M', \mu') \to \mathcal{L}_p(\Omega', M', \mu')$ such that $D T^n = PQ^n D$ for all $n\in \N\cup \{0\}$, where $D:\mathcal{L}_p(\Omega, M, \mu) \to \mathcal{L}_p(\Omega', M', \mu')$ is a positive
isometric embedding and $P:\mathcal{L}_p(\Omega', M', \mu') \to \mathcal{L}_p(\Omega', M', \mu')$ is a positive projection.
\end{thm}

In 1981, Peller introduced positively dominated contractions on $\mathcal{L}_p$ spaces and extend the work by Akcoglu and Sucheston \cite{AL 2} one step ahead. A linear contraction $T\in \mathcal{B}(\mathcal{L}_p(\Omega, M, \mu))$ is said to be \textit{positively dominated} if there is a positive contraction $\widetilde{T}\in \mathcal{B}(\mathcal{L}_p(\Omega, M, \mu))$ such that $|T f | \leq \widetilde{T} |f|$ for every $f\in \mathcal{L}_ p(\Omega, M, \mu)$. In \cite{Peller}, Peller proved that the positively dominated contractions on a $\mathcal{L}_p$ ($1\leq p < \infty$) space are the only contractions that admit unitary dilation on a bigger $\mathcal{L}_p$ space for some fixed $p$. The theorem is stated below.

\begin{thm}\label{thm:V6 new 003}
For $1 < p < \infty$ with $p \neq 2$, the positively dominated operators on $\mathcal{L}_p(\Omega, M, \mu)$ are exactly those which admit unitary dilation on a bigger $\mathcal{L}_p$ space. For $p = 1$, every linear contraction on $\mathcal{L}_1(\Omega, M, \mu)$ has a unitary dilation on a bigger $\mathcal{L}_1$ space.
\end{thm}

Evidently, Peller's work expands the boundary of operators dilated by Akcoglu-Ekkehard \cite{AE} and Akcoglu-Sucheston \cite{AL 2}. However, \cite[Theorem 3]{Peller} provides evidences of contractions on a $\mathcal{L}_p$ space that do not dilate to unitaries on any bigger $\mathcal{L}_p$ space. On the other hand, the work due to Fackler and Gl\"{u}k \cite{SJ} shows that a convex combination of unitaries on a super-reflexive Banach space dilates to a unitary on another space in the same class. In the same paper (see Construction 1.1 in \cite{SJ}), they proved that any contraction on a Banach space $\X$ dilates to a unitary on $\ell_1(\X; \mathbb{Z})$, the space consisting of all absolutely summable bi-sequences (i.e., two-sided sequences) $(x_n)_{n \in \mathbb{Z}}$ on $\mathcal{\X}$. More precisely, the right shift operator $U$ on $\ell_1(\X;\Z)$ dilates a contraction on $\X$ in this case. The construction of Fackler and Gl\"{u}k seems to be inspired by Stroescu's dilation theorem \cite[Corollary 2]{Stroescu}, though Stroescu chose her unitary dilation space to be $\ell_\infty(\X; \mathbb{Z})$.

\smallskip

Below we prove that the unitary dilation of Fackler and Gl\"{u}k is actually a minimal dilation. Let us recall that an isometric dilation $V$, or a unitary dilation $U$ on a Banach space $\Y$ of a contraction $T\in \mathcal B(\X)$ is said to be minimal if
\[
\Y = \overline{span}\{V^n Jx: n\in \N \cup \{0\},  ~x\in \X\} \quad \text{ or } \quad \Y = \overline{span}\{U^n Jx: n\in \Z,  ~x\in \X\},
\]
respectively, where $J:\X \to \Y$ is the associated isometric embedding.

\begin{prop}\label{lem:V6 new 002}
Let $T$ be a contraction on a Banach space $\X$. Then the unitary dilation of $T$ as in \cite[Construction 1.1]{SJ} is a minimal dilation.
\end{prop}

\begin{proof}
The unitary that dilates $T$ as in \cite[Construction 1.1]{SJ} is the right shift operator $U: \ell_1(\X;\Z) \to \ell_1(\X; \Z)$ defined by $U(x_n)_{n\in \Z} = (x_{n-1})_{n\in \Z}$. Also, note that the operators $J, ~ Q$ as in Definition \ref{defn:V6 new 001} are the following:

\begin{align*}
\begin{cases}
& J : \X \to \ell_1(\X;\Z) \\
  & J(x) = (\dotsc, \mathbf{0}, \boxed{x}, \mathbf{0}, \dotsc), \quad x\in \X
\end{cases} 
\quad \text{and} \quad
\begin{cases}
& Q : \ell_1(\X;\Z) \to \X \\
  & Q ((x_n)_{n\in \Z}) = \sum_{n=0}^\infty T^n x_n, \; (x_n)_{n\in \Z} \in \ell_1(\X;\Z),
\end{cases}
\end{align*}
where the box denotes the zeroth component. We now show that $U$ is minimal unitary dilation of $T$. Naturally, $\overline{span}\{U^n Jx: n\in \Z, ~x\in \X\}$ is a closed subspace of $\ell_1(\X;\Z)$. For the reverse inclusion, let $(x_n)_{n\in \Z}$ be arbitrary. Then for each $n\in \N$, the finite sum $\sum_{k=-n}^n U^k Jx_k $ belongs to $\overline{span}\{U^n Jx: n\in \Z, ~x\in \X\}$, and the sequence of bi-sequences $\left(\sum_{k=-n}^n U^k Jx_k \right)_{n=1}^\infty$ converges to $(x_n)_{n\in \Z}$ in $\ell_1(\X;\Z)$. Indeed, we have
\begin{align*}
   \left\|\sum_{k=-n}^n U^k Jx_k - (x_n)_{n\in \Z}\right\| & = \left\|(\dotsc, x_{-n-2}, \underbrace{x_{-n-1}}_{-(n+1)\text{-th}},\mathbf{0}, \dotsc, \mathbf{0},  \underbrace{x_{n+1}}_{(n+1)\text{-th}}, x_{n+2}, \dotsc)\right\| \\
   & = \left(\sum_{k=n+1}^\infty \|x_{-k}\| + \sum_{k=n+1}^\infty \|x_k\| \right) \longrightarrow 0 \quad \text{ as } n\to \infty.
\end{align*}
Consequently, we have $\ell_1(\X;\Z) = \overline{span}\{U^n Jx: n\in \Z, ~ x\in \X\}$. This completes the proof.
\end{proof}

It is merely mentioned that the definition of dilation (see Definition \ref{defn:V6 new 001}) given by Fackler and Gl\"{u}k in \cite{SJ} was about unitary dilation of a Banach space contraction. If one defines isometric dilation of a contraction $T$ on a Banach space $\X$ in the same way as in Definition \ref{defn:V6 new 001}, where the phrase `unitary $U$' is replaced by `isometry $U$', then it can be proved that the forward shift operator $M_z$ on $\ell_1(\X; \N)$ is an isometric dilation of $T$ which further is minimal. The following proposition shows this.

\begin{prop}\label{thm:V6 new 004}
Let $T$ be a contraction on a Banach space $\X$. Then the forward shift operator $M_z:\ell_1(\X;\N) \to \ell_1(\X;\N)$ defined by 
\[
M_z(x_n) = (\mathbf{0}, x_1, x_2, \dotsc), \quad (x_n)\in \ell_1(\X;\N)
\]
is a minimal isometric dilation of $T$.
\end{prop}

\begin{proof}
Consider the linear maps $W: \X \to \ell_1(\X;\N)$ and  $Q: \ell_1(\X;\N) \to \X$ defined by 
\[
W(x) = (x, \mathbf{0}, \mathbf{0}, \dotsc), \quad x\in \X \quad \text{ and } \quad Q((x_n)) = \sum_{n=0}^\infty T^n x_n, \quad (x_n)\in \ell_1(\X;\N).
\]
Then $W$ an isometric embedding of $\X$ into $\ell_1(\X;\N)$, and $Q$ is a contraction satisfying $QW(x) = x$ for all $x\in \X$. Consequently, $P_{_{W(\X)}}:= WQ$ is a norm-one projection of $\ell_1(\X;\N)$ onto $W(\X)$. We also have that 
\begin{equation}\label{eq:V6 new 001}
   Q M_z^k W(x) = T^k x, \quad \text{which implies} \quad P_{_{W(\X)}} M_z^k W(x) = WT^k (x), \quad x\in \X, \quad k\geq 0.
\end{equation}
This shows that $M_z$ is an isometric dilation of $T$. Moreover, we also have that
\begin{equation}\label{eq:V6 new 002}
   \ell_1(\X;\N) = \overline{span}\left\{ M_z^n W(x): n\geq 0, ~ x\in \X \right\}.
\end{equation}
Indeed, $\overline{span}\left\{ M_z^n W(x): n\geq 0, ~ x\in \X \right\}$ is a closed subspace of $\ell_1(\X;\N)$. We now show that every point in $\ell_1(\X;\N)$ is a limit of a sequence from $\overline{span}\left\{ M_z^n W(x): n\geq 0, ~ x\in \X \right\}$. For any $(x_n)_{n=1}^\infty\in \ell_1(\X;\N)$, consider the sequence $\left(\sum_{k=1}^m M_z^{k-1} Wx_k\right)_{m=1}^\infty$ from $\overline{span}\left\{ M_z^n W(x): n\geq 0, ~ x\in \X \right\}$. Needless to mention that $\left(\sum_{k=1}^m M_z^{k-1} Wx_k\right)_{m=1}^\infty$ is a sequence of sequences. Now, we have
\[
   \left\|\sum_{k=1}^m M_z^{k-1} Wx_k - (x_n) \right\| = \left\| (x_1, \dotsc, x_m, \mathbf{0}, \mathbf0, \dotsc) - (x_n)_{n=1}^\infty \right\| = \sum_{k=m+1}^\infty \|x_k\| \rightarrow 0 \quad \text{ as } m \to \infty.
\]
Therefore, $M_z$ on $\ell_1(\X;\N)$ is a minimal isometric dilation of $T$.
\end{proof}

Below we recall an elementary lemma related to Hilbert space dilation which can be found in the literature, e.g., \cite{Nagy Foias}. However, for the convenience of a reader, we present a short proof here.

\begin{lem}\label{lem:V6 new 001}
Let $T$ be a contraction on a Hilbert space $\HS$ and $V\in \mathcal{B}(\KS)$ be an isometric $($or unitary$)$ dilation of $T$. If $W: \HS \to \KS$ is the associated isometric embedding, then for all $h\in \HS$, $\|P_{W(\HS)^ \perp} V Wh \| = \left(\|h\|^2 - \|Th\|^2\right)^{\frac{1}{2}}$ holds, where $P_{W(\HS)^\perp}$ is the orthogonal projection of $\KS$ onto the orthogonal complement of $W(\HS)$ in $\KS$. 
\end{lem}

\begin{proof}
Since $W(\HS)$ is a closed subspace of $\KS$, we have $\KS = W(\HS) \oplus_2 W(\HS)^\perp$. Therefore, for all $h\in \HS$, the expression $VWh = P_{W(\HS)} (VW h) + P_{W(\HS)^\perp} (VWh)$ implies that
\[
  \|h\|^2 = \|VWh\|^2 = \|P_{W(\HS)} VWh \|^2 + \|P_{W(\HS)^\perp} VWh \|^2 = \|Th\|^2 + \| P_{\HS^\perp} VWh \|^2,
\]
where the last equality follows from the fact that  $P_{W(\HS)} VWh = WTh$. This completes the proof.
\end{proof}

Isometric dilation of a Banach space contraction as per Definition \ref{Definition of Dilation of contractions on Banach space-1} generalizes Hilbert space dilation as it also satisfies the conclusion of Lemma \ref{lem:V6 new 001} as shown below.

\begin{prop} \label{prop:new-001A}
If a strict contraction $T$ on a Banach space $\X$ dilates to an isometry $V$ on a Banach space $\widetilde{\X} \supseteq \X$ as per Definition \ref{Definition of Dilation of contractions on Banach space-1} and if $W:\X \to \widetilde{\X}$ is the associate isometric embedding, then for all $x\in \X$, $\|P_{_{W(\X)^ \perp}} V Wx \| = \left(\|x\|^2 - \|Tx\|^2\right)^{\frac{1}{2}}$ holds, where $P_{_{W(\X)^\perp}}$ is the orthogonal projection of $\widetilde{\X}$ onto the orthogonal complement of $W(\X)$ in $\widetilde{\X}$.
\end{prop}

\begin{proof}
It follows from Definition \ref{Definition of Dilation of contractions on Banach space} that there exists a closed linear subspace $\mathbb{L}$ of $\widetilde{\X}$ such that $\widetilde{\X} = W(\X) \oplus_2 \mathbb{L}$. Therefore, for all $x\in \X$, we have
\[
  \|x\|^2 = \|VW(x)\|^2 = \left\|P_{_{W(\X)}}VW(x)\right\|^2 + \left\|\left(I- P_{_{W(\X)}}\right)VW(x)\right\|^2 = \|Tx\|^2 + \left\|P_{_{W(X) ^\perp}}VW(x)\right\|^2,
\]
where the first equality follows from the facts that $V$ and $W$ are isometries, and the last equality follows from the fact that $P_{_{W(\X)}}VW(x) = W(Tx)$. This completes the proof.
\end{proof}

Now, we show that neither the isometric dilation $M_z$ as in Proposition \ref{thm:V6 new 004} nor the unitary dilations $U$ as in Proposition \ref{lem:V6 new 002} and Theorem \ref{thm:V6 new 001} satisfy the conclusion of Lemma \ref{lem:V6 new 001}.

\begin{prop} \label{prop:new-002A}
The dilations as in Proposition \ref{thm:V6 new 004} , Proposition \ref{lem:V6 new 002} and Theorem \ref{thm:V6 new 001} do not satisfy the conclusion of Lemma \ref{lem:V6 new 001}.
\end{prop}

\begin{proof}
First, we note that the only difference between the associated isometric embeddings $W$ (as in Proposition \ref{thm:V6 new 004}) and $J$ (as in Proposition \ref{lem:V6 new 002}) is that $W(x)$ is a sequence, whereas $J(x)$ is a bi-sequence with identical non-zero entries for all $x \in \X$. Therefore, the proofs for the cases Proposition \ref{thm:V6 new 004} and Proposition \ref{lem:V6 new 002} are similar. We give proof only to Proposition \ref{thm:V6 new 004}. From equation \eqref{eq:V6 new 001} in Proposition \ref{thm:V6 new 004}, it follows that for all $x \in \X$, we have
\[
    \|(I-P_{_{W(\X)}}) M_z W(x)\| = \left\|\left(-Tx, x, \mathbf{0}, \mathbf{0}, \dotsc\right)\right\| = \|x\| + \|Tx\| \geq \|x\| \geq \left(\|x\|^2 - \|Tx\|^2\right)^{\frac{1}{2}}.
\]
This shows that $\|P_{{W(\X)}^{\perp}} {\widehat U} Wx \|=\|(I-P_{_{W(\X)}}) M_z W(x)\| > \left(\|x\|^2 - \|Tx\|^2\right)^{\frac{1}{2}}$ for all $x\in \X \setminus \ker(T)$, which violates the conclusion of Lemma \ref{lem:V6 new 001}. Such a choice of vectors is possible since $T$ is a nonzero operator.

\smallskip

We now prove the proposition for the unitary dilation constructed by Stroescu, as in \cite[Corollary 2]{Stroescu} and stated in Theorem \ref{thm:V6 new 001}. Before proceeding, we briefly outline the construction. Consider the Banach space $\Y = \ell_\infty(\X; \Z)$ consisting of all bi-sequences $(x_n)_{n\in \Z}$ in $\X$ such that $\displaystyle\sup_{n\in \Z}\|x_n\| < \infty$. Consider the linear maps $W: \X \to \Y$ and $Q: \Y \to \X$ defined by
\[
 W(x) = (T^{|n|}x)_{n\in \Z}, \quad x\in \X, \quad \text{ and } \quad Q((x_n)_{n\in \Z}) = x_0, \quad (x_n)_{n\in \Z} \in \Y.
\]
Then $W$ is an isometry and $Q$ is a contraction satisfying $Q W(x) = x$ for all $x\in \X$. Therefore, $P_{_{W(\X)}}:= WQ$ is a norm-one projection of $\Y$ onto $W(\X)$. Now, consider the backward shift operator $\widehat U: \Y \to \Y$ defined by
\[
 {\widehat U}(x_n)_{n\in \Z} = (x_{n+1})_{n\in \Z}, \quad (x_n)_{n\in \Z}\in \Y=\ell_\infty(\X; \Z).
\]
Then ${\widehat U}$ is an isometry and we have $Q {\widehat U}^k W = T^k,$ which further implies that
\[
  P_{_{W(\X)}} {\widehat U}^k W(x) = WT^k (x), \quad k \geq 0, \quad x\in \X.
\]
Now, we show that ${\widehat U}$ does not satisfy the conclusion of Lemma \ref{lem:V6 new 001}. For all $x\in \X$, we have that
\[
(I- P_{_{W(\X)}}) {\widehat U} Wx = (z_n)_{n\in \Z},~~\text{where} ~~ z_n = \begin{cases}
							 \mathbf{0} & \text{ if } n\geq 0 \\
							  T^{(|n|-1 )}(x- T^2x) & \text{ if } n \leq -1.
							\end{cases}
\]
Therefore, for all $x\in \X$, we have
\begin{equation}\label{eq:V6 new 003}
\|(I- P_{_{W(\X)}}) {\widehat U} Wx \| = \sup_{n\in \Z} \|z_n\| = \|x- T^2x\| \quad \left[\text{ as } \|T\| \leq 1~\right].
\end{equation}
It is well known that for an arbitrary Banach space operator $T$, we have $\ker(T^2) = \ker(T)$ if and only if $\ker(T) \cap Range(T)=\{\mathbf{0}\}$. Therefore, \eqref{eq:V6 new 003} shows that $\|(I- P_{_{W(\X)}}) {\widehat U} Wx \| > \left(\|x\|^2 - \|Tx\|^2 \right)^{\frac{1}{2}}$ for all $x\in \ker(T^2) \setminus \ker(T)$, if we choose $T$ in such way that $\ker(T)$ becomes a proper subspace of $\ker(T^2)$. For example, consider $\X = (\C^2, \|\cdot\|_1)$ and $T(x,y) = r(x+y, 0)$, where $r\in (0,1)$ is fixed. Then $\|T\| = r < 1$ and $T^2 (0,1) = (r^2, 0)$. Thus, we have
\[
   \left(\|(0,1)\|^2 - \|T(0,1)\|^2 \right)^{\frac{1}{2}} = (1- r^2)^{\frac{1}{2}} < \|(0,1) - T^2(0,1)\| = (1-r^2).
\]
Evidently, the the fact that $\|P_{{W(\X)}^{\perp}} {\widehat U} Wx \|=\|(I- P_{_{W(\X)}}) {\widehat U} Wx \| > \left(\|x\|^2 - \|Tx\|^2 \right)^{\frac{1}{2}}$ contradicts the conclusion of Lemma \ref{lem:V6 new 001}. The proof is complete.
\end{proof}

\section{Background materials and preparatory results} \label{sec:02}

\vspace{0.2cm}

\noindent
In this Section, we recall from the literature (e.g. \cite{Birkhoff, James 1, James 2, Kinnunen, Megginson, R 1, Sain 1, Sain 2} etc.) a few basic concepts as well as introduce a few new terminologies that will be used throughout the paper . We give proofs to a few results that we could not locate in the literature. The mathematical tool that we use very frequently is Birkhoff-James orthogonality.

\medskip

\noindent \textbf{$\bullet$ Birkhoff-James orthogonality:} Given any two elements $x$ and $y$ in a Banach space $\X$, $x$ is said to be orthogonal to $y$ in the sense of Birkhoff-James \cite{Birkhoff, James 1}, written as $x\perp_B y$, if $\|x+\lambda y\|\geq \|x\|$ for all scalars $\lambda$. Similarly, two subspaces $\X_1$ and $\X_2$ of $\X$, $\X_1\perp_B \X_2$, if $x_1\perp_B x_2$ for all $x_1\in \X_1$ and $x_2\in \X_2$. When $\X$ is a Hilbert space, the notion of Birkhoff-James orthogonality is equivalent to the inner product orthogonality. However, unlike to the inner product orthogonality, Birkhoff-James orthogonality may not be symmetric in general \cite{Sain 1, Sain 2}, i.e, $x\perp_B y$ may not necessarily imply $y\perp_B x$. In fact, regarding the symmetry of Birkhoff-James orthogonality, we have the following famous result due to James.

\begin{thm}[\cite{James 2}]\label{Symmetry of Birkhoff-James orthogonality}
Let $\X$ be Banach space with $dim~ \X \geq 3.$ Then $\X$ is a Hilbert space if and only if Birkhoff-James orthogonality is symmetric in $\X.$
\end{thm}
The above result was proved for real Banach spaces using \cite{Kakutani}. In the complex case the result follows from \cite{Bohenblust}.

\medskip

\noindent \textbf{$\bullet$ Support functionals:} Denote the topological dual of $\X$ by $\X^*$. Let $x\in \X$ be nonzero. Then a functional $f\in \X^*$ is said to be a \textit{support functional} at $x$ if $f(x)=\|x\|^2$ and $\|f\|=\|x\|$. The collection of all support functionals at $x$, denoted by $J(x)$, is the following set: 
\begin{equation} \label{eqn:new-021}
 J(x)=\{f_x\in \X^*:~\|f_x\|=\|x\|,~f_x(x)=\|x\|^2\}.
\end{equation}
It follows from Hahn-Banach Theorem that the collection $J(x)$ is nonempty. James \cite{James 1} characterized Birkhoff-James orthogonality in terms of support functionals. Though, the characterization was first proved for real Banach spaces, it was later found to be valid for complex Banach spaces also.

\begin{thm}[James]\cite[Theorem 2.1]{James 1}\label{James Characterization}
 Let $\X$ be a Banach space and $x,y\in \X.$ Then $x\perp_B y$ if and only if there exists $f_x\in J(x)$ such that $f_x(y)=0.$  
\end{thm}

\noindent \textbf{$\bullet$ Reflexive Banach spaces:} Let $\X$ be a Banach space and let $\pi_\X:\X\to \X^{**}$ be the canonical embedding defined by
\begin{equation}\label{Cano. embed. X**}
    \pi_\X (x) = \hat{x}, \qquad \hat{x}(x^*) = x^*(x), \quad x^* \in \X^*.
\end{equation}
The map $\pi_\X$ is a linear isometry. The space $\X$ is called \textit{reflexive} if $\pi_\X$ is surjective. Using James characterization (Theorem \ref{James Characterization}), orthogonality of linear functionals on a reflexive Banach space can be characterized in terms of norm attainment sets and kernels. Given any member $f$ of $\X^*$, the symbol $M_f$ stands for the norm attainment set of $f$, i.e.,
\begin{equation} \label{eqn:new-022}
 M_f=\{y\in S_\X:~|f(y)|=\|f\|\}.
\end{equation}
The following lemma is analogous to the characterization of orthogonality of weak$^*$ continuous functionals proved in Lemma 3.7 of \cite{Roy-Sain-new}. However, we give a brief proof to it, because, in \cite{Roy-Sain-new} the analogous result was proved in the setting of the double dual $X^{**}$.
\begin{lem}\label{Orthogonality of functionals}
Let $\X$ be a reflexive Banach space and let $f,~g\in \X^*.$ Then $f\perp_B g$ if and only if $M_f\cap \ker~ g \neq\emptyset.$
\end{lem}

\begin{proof}
If $u\in M_f\cap \ker~ g,$ then 
\[
   \|f+\lambda g\| \geq | (f+\lambda g)u| = |f(u)|= \|f\|, \qquad \lambda\in \mathbb{C}.
\]
Thus, $f\perp_B g.$ On the other hand, if $f\perp_B g$, then by Theorem \ref{James Characterization} there exists $l\in J(f)$ such that $l(g)=0.$ Since $\X$ is reflexive, $\pi_\X(x)=l$ for some $x\in \X.$ Therefore,
\[
 l(f)= \pi_\X(x)(f)=f(x)= \|f\|^2=\|f\| \|x\|, \quad \text{ and } \quad \pi_\X(x)(g)=g(x)=0.
\]
Consequently, $\dfrac{x}{\|x\|}\in M_f\cap \ker~ g.$ This completes the proof.
\end{proof}

\noindent \textbf{$\bullet$ Smooth and strictly convex Banach spaces:} Let $\X$ be a Banach space. A nonzero point $x\in \X$ is said to be \textit{smooth}, if $J(x)$ is a singleton set. Also, a point $x \in \X$ is said to be an \textit{exposed point} of $B(\mathbf{0},\|x\|)$, if there exists $f_x\in J(x)$ such that $f_x(y)=\|x\|\|y\|$ for any $y\in \X$ with $\|x\|=\|y\|$  implies that $x=y$. The Banach space $\X$ is said to be \textit{smooth} if each nonzero point $x \in \X$ is smooth. The Banach space $\X$ is said to be \textit{strictly convex} or \textit{rotund} or \textit{strictly normed} if $\|tx_1+(1-t)x_2\|<1$ whenever $x_1,x_2$ are distinct points of $S_{\X}$ and $0<t<1$. Also, we learn from the literature (see \cite{Megginson}) that a Banach space $\X$ is strictly convex if and only if any nonzero $y\in \X$ is an exposed point of $B(\mathbf{0},\|y\|)$. In general, if $\X^*$ is smooth then $\X$ is strictly convex and if $\X^*$ is strictly convex then $\X$ is smooth but the converse does not hold, e.g. see Propositions 5.4.5 \& 5.4.6 in \cite{Megginson} and the discussion thereafter. Therefore, a reflexive Banach space $\X$ is smooth or strictly convex if and only if $\X^*$ is strictly convex or smooth respectively. Smoothness of a point can also be characterized in terms of the right-additivity of Birkhoff-James orthogonality in the following way.

\begin{thm}\cite{James 1}, \cite[Theorem 1.1]{Saikat Roy} \label{Orthogonality = Right-additivity}
    Let $\X$ be a Banach space and $x\in\X$ be a nonzero element. Then $x$ is smooth if and only if $x\perp_B y$ and $x\perp_B z$ imply that $x\perp_B (y+z)$ for all $y,z\in\X.$
\end{thm}

\smallskip

We now introduce the notion of right and left-complemented subspaces of a Banach space. These concepts will be useful in the coming sections.

\smallskip

\noindent \textbf{$\bullet$ Right-complemented subspace:} Let $\X$ be a Banach space and let $\Y$ be a closed linear subspace of $\X$. Then $\Y$ is said to be \textit{right-complemented} in $\X$, if there exists a subspace $\widetilde{\Y}$ of $\X$ such that $\widetilde{\Y}$ is closed, vector space complement of $\Y$, and $\Y\perp_B \widetilde{\Y}$. In that case, $\X$ is written as $\X=\Y\bigoplus_\SR \widetilde{\Y}$, and $\widetilde{\Y}$ is said to be the right complement of $\Y$ in $\X$.

\medskip

\noindent \textbf{$\bullet$ Left-complemented subspace:} Let $\X$ be a Banach space and let $\Y$ be a closed linear subspace of $\X$. Then $\Y$ is said to be \textit{left-complemented} in $\X$, if there exists a subspace $\widetilde{\Y}$ of $\X$ such that $\widetilde{\Y}$ is closed, vector space complement of $\Y$, and $\widetilde{\Y}\perp_B \Y$. In that case, $\X$ is written as $\X=\widetilde{\Y}\bigoplus_\SL \Y$, and $\widetilde{\Y}$ is said to be the left complement of $\Y$ in $\X$.

\medskip 

\noindent \textbf{$\bullet$ Complemented subspace:}  We define complemented subspaces of a Banach space as it is defined in \cite{Megginson}. A closed linear subspace $\Y$ of a Banach space $\X$ is said to be a \textit{complemented subspace} if there is a closed linear subspace $\widetilde{\Y}$ of $\X$ such that $\Y \cap \widetilde{\Y} = \{ \mathbf 0 \}$ and $\X=\Y \oplus \widetilde{\Y}$, where $\X=\Y \oplus \widetilde{\Y}$ means that every element $x\in \X$ can be (uniquely) written as $x=y + y_1$ with $y \in \Y$ and $y_1 \in \widetilde{\Y}$. In this case, $\widetilde{\Y}$ is called a \textit{complement} of $\Y$ in $\X$.

\begin{prop}
A left-complemented or right-complemented subspace of a Banach space is a complemented subspace.
\end{prop}

\begin{proof}

Follows from the definitions of left-complemented and right-complemented subspaces.
\end{proof}
However, a complemented subspace may not always be right or left-complemented as the following example shows.
\begin{eg}
Let $\X$ be a Banach space which is not a Hilbert space. By Theorem A in \cite{Bohenblust}, there exists a two dimensional subspace $\Y$ which is not right-complemented in $\X$. But the subspace $\Y$ is complemented by Theorem 3.2.18 in \cite{Megginson}. Again, by Theorem A in \cite{Bohenblust} and Theorem 3.2.18 in \cite{Megginson}, there is a complemented subspace $\widehat{\Y}$ of codimension 2 which is not left-complemented in $\X$. 
\end{eg} \qed

Interestingly, a closed subspace in a Banach space is complemented if and only if it is the range of a bounded projection operator. A reader is referred to Corollary 3.2.15 in \cite{Megginson} for a proof to this. Also, a right-complemented subspace of a Banach space is the range of a norm-one projection and a left-complemented subspace is the kernel of a norm-one projection and vice-versa. This result is known in a different form for real Banach spaces, e.g. see \cite{Kinnunen, R 1}. However, we could not locate a proof in the literature for the same for a complex Banach space and thus we write a proof below. 

\begin{prop}\label{Range and kernel of norm one Projection}
Let $\X$ be a Banach space and $\Y$ be a closed proper subspace of $\X$. Then the following hold.
\begin{itemize}
\item[(i)] $\Y$ is the range of a norm-one projection if and only if there is a closed subspace $\widetilde{\Y}$ of $\X$ such that $\widetilde{\Y}$ is vector space complement of $\Y$ and $\Y\perp_B \widetilde{\Y}$.

\smallskip

\item[(ii)] $\Y$ is the kernel of a norm-one projection if and only if there is a closed subspace $\widetilde{\Y}$ of $\X$ such that $\widetilde{\Y}$ is vector space complement of $\Y$ and $\widetilde{\Y}\perp_B \Y$.
\end{itemize}
\end{prop}

\begin{proof}
(i) Suppose that $\Y$ is range of norm-one projection $P$. Set $(I-P)\X=\widetilde{\Y}$, where $I$ is the identity operator on $\X$. Then $\widetilde{\Y}$ is closed and is a vector space complement of $\Y$. Also, for any $y\in \Y$ and $z\in \widetilde{\Y}$, we have
\[
 \|y+\lambda z\|\geq \|P(y+\lambda z)\| = \|P(y)\| = \|y\|, \qquad \lambda \in \mathbb{C}.
\]
Therefore, $\Y\perp_B \widetilde{\Y}$.

\smallskip

Conversely, let $\widetilde{\Y}$ be a subspace of $\X$ satisfying the stated conditions. Let $P_\Y$ be the projection on $\Y$. Since $\Y\perp_B \widetilde{\Y}$, we have that
\[
 \|y+z\|\geq \|y\| = \|P_\Y(y)\| = \|P_\Y(y+z)\|, \qquad y\in \Y,~z\in \widetilde{\Y}.
\]
Thus, $\|P_\Y\|=1.$

\smallskip

(ii) Suppose that $\Y$ is the kernel of a norm-one projection $P$. Set $\widetilde{\Y} = \text{Ran}~P$. Then $\widetilde{\Y}$ is a closed subspace and is also vector space complement of $\Y$. Also, for any $z\in \widetilde{\Y}$ and $y\in \Y$, we have
\[
 \|z+\lambda y\| \geq \|P(z+\lambda y)\| = \|z\|, \qquad \lambda \in \mathbb{C}.
\]
Thus, $\widetilde{\Y}\perp_B \Y$.

\smallskip

Conversely, let $\widetilde{\Y}$ be a subspace of $\X$ satisfying the stated conditions. Let $P_{\widetilde{\Y}}$ be the projection on $\widetilde{\Y}$. Then $\ker~P_{\widetilde{\Y}}=\Y$. Moreover, since $\widetilde{\Y}\perp_B \Y$, we have $\|P_{\widetilde{\Y}}\|=1$. This completes the proof.

\end{proof}
We shall shortly see that right-complemented (left-complemented) subspaces are also connected to the concept of duality maps and Hahn-Banach extension operator.

\smallskip

\noindent \textbf{$\bullet$ Orthogonally complemented or 1-complemented subspaces:} An \textit{orthogonally complemented} or $1$-\textit{complemented} subspace of a Banach space is precisely a right-complemented subspace, i.e. the range of a norm-one projection (as per Proposition \ref{Range and kernel of norm one Projection}). See \cite{Kinnunen, Moslehian, R 1, R 2} for more details.

\smallskip

\noindent \textbf{$\bullet$ Duality map and Hahn-Banach extension operator:} A duality map of $\X$ is a selector $\J_\X:\X\to \X^*$ such that $\J_\X(x)\in J(x)$, where $J(x)$ is as in (\ref{eqn:new-021}), for each nonzero $x\in \X$ and $\J_\X(\mathbf{0})=\mathbf{0}$. Note that if $\X$ is smooth then the duality map of $\X$ is uniquely determined. For a closed subspace $\mathbb{Y}$ of $\X$, we denote by $\mathbb{J}_{\X}(\Y)$ the following set:
\begin{equation} \label{eqn:new-001}
\J_{\X}(\Y) = \{\J_{\X}(x): x\in \Y\}.
\end{equation}
Let $\Y$ be a closed subspace of $\X$. Given any $f\in \Y^*$, we consider the (nonempty) collection
\begin{equation} \label{eqn:neq-0002}
 H_\Y(f):=\left\{\widetilde{f}\in \X^*:~\widetilde{f}~\text{is a Hahn-Banach extension of}~f\right\}.
\end{equation}
A Hahn-Banach extension operator for $\mathbb{Y}$ is a selector $\Psi:\mathbb{Y}^*\to \X^*$ such that $\Psi(f)\in H_\Y(f)$ for each $f\in \Y^*$. Note that if $\X^*$ is strictly convex then the Hahn-Banach extension operator is uniquely determined.  Also, denote the annihilator of $\Y$ by $\Ann(\Y)$, which is defined as
\begin{equation} \label{eqn:new-0003}
\Ann(\mathbb{Y}): = \{f\in \X^*:~f|_\mathbb{Y}= \mathbf 0\}.
\end{equation}
A Hahn-Banach extension operator for a closed subspace is not necessarily linear even in the strictly convex case. In fact, linearity of a Hahn-Banach extension operator is an extensive area of research in Banach space theory, e.g. see \cite{AGS, Martin, SY}.

\smallskip

\noindent \textbf{$\bullet$ Wandering subspace, unilateral shift and bilateral shift on a Hilbert space:} Let $V$ be an isometry on a Hilbert space $\mathcal{H}$. A subspace $\mathcal{L}$ of $\mathcal{H}$ is called a \textit{wandering subspace} for $V$ if $V^n \mathcal{L} \perp \mathcal{L}$ for all $n\in \mathbb{N}$. The isometry $V$ is called a \textit{unilateral shift} if there exists a subspace $\mathcal{L}$, which is wandering for $V$ and such that $\displaystyle\oplus_{n=0}^\infty V^n \mathcal{L}= \mathcal{H}$. This subspace $\mathcal{L}$ is called generating subspace for $V$. It is uniquely determined by $ \mathcal{H} \ominus V\mathcal{H}$, the orthogonal complement of $V\mathcal{H}$ in $\mathcal{H}$. The dimension of $\mathcal{L}$ is called the \textit{multiplicity} of the unilateral shift $V$. 

\smallskip

An operator $U$ on $\mathcal{H}$ is called a bilateral shift if $U$ is unitary and if there exists a subspace $\mathcal{L}$ of $\mathcal H$, such that $U^n \mathcal{L} \perp \mathcal{L}$ for all $n\in \mathbb{Z} \setminus \{0\}$ and $ \displaystyle \oplus_{n=-\infty}^\infty V^n \mathcal{L} = \mathcal{H}$. Every such subspace $\mathcal{L}$ is called a generating subspace of $U$. Unlike a unilateral shift, a wandering subspace $\mathcal{L}$ is not uniquely determined for a bilateral shift. However, the dimension of $\mathcal{L}$ is uniquely determined and $\ dim \mathcal{L}$ is called the \textit{multiplicity} of the bilateral shift $U$. Indeed, if $\mathcal{L}$ is a generating subspace for $U$, then $U\mathcal{L}$ is also a generating subspace for $U$. 

\medskip

\noindent \textbf{$\bullet$ Wandering subspace of a Banach space isometry:} Let $V$ be an isometry on a Banach space $\X$. Then a subspace $\mathcal L$ of $\X$ is said to be a \textit{wandering subspace} for $V$ if $V^n \mathcal L \perp_B \mathcal L$ for all $n\geq 1$. Analogous to Hilbert space isometries, the wandering subspace of a Banach space isometry has the following property.

\begin{lem}\label{Wandering Space}
Suppose that $V$ is an isometry on a Banach space $\mathbb{X}$ and $\mathcal{L}$ is a closed subspace of $\mathbb{X}.$ Then the following are equivalent:
\begin{itemize}
\item[(i)] $\mathcal L$ is a wandering subspace for $V$;

\item[(ii)] $V^n\mathcal{L}\perp_B V^m\mathcal{L},~0\leq m < n.$
\end{itemize}
\end{lem}
\begin{proof}
It suffices to prove (i)$\implies$(ii) as the other implication follows trivially. Let $\mathcal L$ be a wandering subspace for $V$. Then $V^n \mathcal L \perp_B \mathcal L$ for all $n \geq 1$. Let $\lambda\in\mathbb{C}$ and let $l_1,l_2\in\mathcal{L}$ be arbitrary. Then for $n > m\geq 0,$ we have
\[
    \|V^nl_1+\lambda V^ml_2\|  = \|V^m\left(V^{n-m}l_1+\lambda l_2\right)\| = \|\left(V^{n-m}l_1+\lambda l_2\right)\| \geq \|V^{n-m}l_1\|=\|l_1\|=\|V^nl_1\|
\]
and the proof is complete. 
\end{proof}

\noindent \textbf{$\bullet$ Unilateral shift on a Banach space:} The following definition of unilateral shift on a Banach space was first introduced by Faulkner and Huneycutt in \cite{Faulkner Huneycutt}.

\begin{defn}
An isometry $V: \mathbb{X}\to \mathbb{X}$ is said to be a \textit{unilateral shift} if there exists a subspace $\mathcal{L}\subseteq\mathbb{X}$ such that
\begin{itemize}
\item[(i)] for $n > m$, $V^n \mathcal{L}\perp_B V^m \mathcal{L} ;$
\item [(ii)] $\mathbb{X}=\displaystyle\bigoplus_{n=0}^{\infty}V^n\mathcal{L}$,
\end{itemize}
where the direct sum $\displaystyle\bigoplus_{n=0}^{\infty}V^n\mathcal{L}$ stands for the following set:
\begin{equation} \label{eqn:new-0222}
   \bigoplus_{n=0}^{\infty}V^{n}\mathcal{L} =\left\{x\in\X: x\mbox{ can be uniquely expressed as }x=\sum_{n=0}^{\infty}V^nl_n,\;l_n\in \mathcal{L} \text{ for } n \geq 0\right\}.
\end{equation}
The space $\mathcal{L}$ is called the \textit{generating subspace} and the $\dim \mathcal{L}$ is called the \textit{multiplicity} of $V$.
\end{defn}
It is easy to see that the multiplicity of a unilateral shift is well-defined. Let $\X$ be a Banach space and $V$ be a unilateral shift on $\X$ with a generating subspace $\mathcal{L}$. Then $\dim \mathcal{L}$ is uniquely determined by $\dim(V / V\X)$. This is because, $\X = V\X \oplus \mathcal{L}$ and $V\X$ is a closed subspace of $\X$. Therefore, the quotient space $\X / V\X$ is isomorphic with $\mathcal{L}$ as a vector space. Hence $\dim(\X / V\X) = \dim \mathcal{L}$.

\smallskip

Evidently, the generating subspace $\mathcal L$ is a wandering subspace for $V$. The forward shift operator $M_z$ on $\ell_2(\X)$, defined by $M_z(x_1, x_2,x_3, \dots)=(\mathbf 0,x_1,x_2, \dots)$, is an example of a unilateral shift. Note that the direct sum as in (\ref{eqn:new-0222}) can be expressed in different forms under the assumption of smoothness, as the Equation-(\ref{Direct sum equality}) in the next lemma shows.

\begin{lem}\label{span closure equals to direct sum} Let $V$ be an isometry on a smooth Banach space $\mathbb{X}$ and 
$
    \mathbb{X}=V\mathbb{X}\bigoplus_\SR \mathcal{L}.
$
Suppose 
$
    L_0 = \mathcal{L}, \ \ L_n=\bigoplus_{k=0}^nV^k\mathcal{L},~n\geq 1.
$
Then
\begin{equation}\label{Direct sum equality}
    \left(~\overline{\bigcup_{n=0}^\infty L_n}~\right)=\overline{span}\left\{V^n\mathcal{L}:~n\geq 0\right\}=\bigoplus_{n=0}^{\infty}V^{n}\mathcal{L}.
\end{equation}
\end{lem}

\begin{proof}
We only prove the second equality in (\ref{Direct sum equality}) since 
the first equality is straightforward. Let 
$\mathbb{Y}=\bigoplus_{n=0}^{\infty}V^{n}\mathcal{L}\subseteq\mathbb{X}$ 
and $\mathbb{W}= \overline{span} \left\{V^n\mathcal{L}:~n\geq 0 \right\}$. 
Evidently, $\mathbb{Y}\subseteq {\mathbb{W}}$. We prove the 
other side of the inclusion, i.e.,
$\mathbb{W}\subseteq\mathbb{Y}.$ Let $x\in {\mathbb{W}}$ be 
arbitrary. Then there is a sequence $(x_n)$ in $ span \left\{V^n\mathcal{L}:~n\geq 0 \right\}$ such that 
${\displaystyle \lim_{n \rightarrow \infty}x_n=x }$. Therefore, there is a subsequence $(x_{n_{k}})$ of 
$(x_n)$ satisfying
$
    \|x_{n_{k}}-x_{n_{k-1}}\|\leq\frac{1}{2^{k}}$ for $k\geq 1$. Since $(x_{n_{k}}-x_{n_{k-1}})\in {\mathbb{W}}$ for every $k,$ 
the series $\displaystyle\sum_{k=1}^{\infty}(x_{n_{k}}-x_{n_{k-1}})$ 
converges in ${\mathbb{W}}.$ Thus,

\[
   x= \sum_{k=1}^{\infty}y_k, \quad y_1=x_{n_{1}},\quad 
y_k=(x_{n_{k}}-x_{n_{k-1}}),\qquad k\geq 2.
\]

To prove that the above expression is unique, it is enough to show that

\[
\sum_{k=1}^{\infty}V^{k}l_k=\mathbf{0} \quad \text{implies}\quad 
l_k=\mathbf{0},\qquad k\geq 1.
\]

Let $x\in\mathbb{X}$ be such that 
$x=\sum_{k=1}^{\infty}V^{k}l_k=\mathbf{0}.$ Evidently, $Vx\perp_B Vy$ 
when $x\perp_B y$. Thus, $V^n\mathcal{L}\perp_B V^m\mathcal{L}$, for 
$n>m\geq 1.$ Now, by the right-additivity of Birkhoff-James 
orthogonality, we have

\begin{equation}\label{Vn_perpendicular_to_all_lowerpowers}
V^n(\mathcal{L})\perp_B\left( \mathcal{L} \oplus V\mathcal{L} \oplus 
V^2\mathcal{L} \oplus \cdots \oplus V^{n-1}\mathcal{L} \right), \qquad 
n\geq 1.
\end{equation}

\medskip

Let $y_m= \sum_{k=m+1}^{\infty}V^{k}l_k.$ We show that $y_m = 
\mathbf{0}$ for all $m\geq 1.$ Let $(t_n)$ be the sequence of partial 
sum of the series $\sum_{i=m+1}^{\infty}V^{i-m-1}l_i.$ For $q>p$ we have

\begin{align*}
     \|t_q-t_p\| & 
=\left\|V^{p+1}l_{p+m+2}+V^{p+2}l_{p+m+3}+\cdots+V^ql_{q+m+1}\right\| \\
     & = \left\|V^{p+m+2}l_{p+m+2} + V^{p+m+3}l_{p+m+3} + 
\cdots+V^{q+m+1}l_{q+m+1}\right\| \longrightarrow 0, \quad \text{ as } 
p,~ q\to\infty.
\end{align*}

Thus, $(t_n)$ is a Cauchy sequence and the claim is proved. For 
$z=\sum_{k=m+1}^{\infty}V^{k-m-1}l_k,$ we have

\[
   V^{m+1}z=y_m=\sum_{k=m+1}^{\infty}V^kl_k,
\]

and by (\ref{Vn_perpendicular_to_all_lowerpowers}) we have

\[
   0 = \|x\| = \left\| \sum_{k=0}^{m}V^k l_k+y_m \right\| = \left\| 
\sum_{k=0}^mV^kl_k + V^{m+1}z \right\|\geq\| V^{m+1}z \| = \|y_m\|, 
\qquad m\geq 1.
\]

Consequently, $y_m=\mathbf{0}$ for $m\geq 1.$ On the other hand by 
(\ref{Vn_perpendicular_to_all_lowerpowers}) we also have

\[
\|l_k\| = \| V^k l_k \| \leq \left\| \sum_{i=0}^{k-1}V^il_i + V^kl_k 
\right\|=0,\quad 1\leq k\leq m.
\] Thus,
$
l_i = \mathbf{0}, \quad 0\leq i\leq m $ for $m\geq 1$. This completes the proof.
\end{proof}

\noindent \textbf{$\bullet$ Bilateral shift between Banach spaces:} Bilateral shift between Banach spaces was defined in \cite{CFS}, see Definition 2 in \cite{CFS}. A bijective linear (not necessarily bounded) map $U: \mathbb X \rightarrow \mathbb X_1$, where $\X, \, \X_1$ are Banach spaces with $\X_1 \subseteq \X$, is said to be a \textit{bilateral shift} if there is a linear subspace $\mathcal{L}\subseteq \X$ such that $\X_1 = \displaystyle \bigoplus_{n=-\infty}^\infty U^n \mathcal{L}$, that is every element $x\in \X_1$ has unique representation as a convergent series $x= \displaystyle \sum_{n=-\infty}^\infty U^n l_n$ with $l_n\in \mathcal{L}$ for all $n\in \mathbb{Z}$. The subspace $\mathcal{L}$ is a \textit{wandering subspace} for $U$ and $\dim \mathcal L$ is called the \textit{multiplicity} of $U$. In \cite{Gellar}, Gellar and Silber constructed a Banach space $\X$ with a Schauder basis $(y_n)_{n\in \mathbb{Z}}$ and proved that the bilateral shift $S: \X \to \X$ defined by 
\[
S\left(\sum_{n=-\infty}^\infty a_n y_n \right) = \sum_{n=-\infty}^\infty a_n y_{n-1}, \quad a_n \in \mathbb{C} \text{ for all } n\in \mathbb{Z},
\]
is not bounded. He also proved that the inverse map $S^{-1}: \X \to \X$ defined by 
\[
S^{-1}\left(\sum_{n=-\infty}^\infty a_n y_n \right) = \sum_{n=-\infty}^\infty a_n y_{n-1}, \quad a_n \in \mathbb{C} \text{ for all } n\in \mathbb{Z},
\]
is a bounded linear map. Also, a reader is referred to Example \ref{Example:new-051}, where we provide an example of a bilateral shift which is also a Banach space unitary.

\medskip

\noindent \textbf{$\bullet$ Wold isometry on a Banach space:} Every Hilbert space isometry admits a Wold decomposition, which is to say that if $V$ is an isometry on a Hilbert space $\HS$, then $\HS$ admits an orthogonal decomposition $\HS=\HS_0 \oplus \HS_1$ into reducing subspaces of $V$ such that $V|_{\HS_0}$ is a unitary and $V|_{\HS_1}$ is a pure isometry. A \textit{pure isometry} $V_1$ on a Hilbert space $\widetilde{\HS}$ is an isometry that has no unitary part, i.e. there is no nonzero closed subspace of $\widetilde{\HS}$ which is reduced by $V_1$ and $V_1$ restricted to that subspace is a unitary. A pure isometry on a Hilbert space is unitary equivalent to a unilateral shift. An analogue of this result does not always hold in Banach spaces. In fact, the range of a Banach space isometry may not be even a complemented subspace, a Wold decomposition is too far for all Banach space isometries. For example, it was shown in \cite{Ditor} that the range of an isometry on $C[0,1]$ may not be complemented. Indeed, the question when the range of an isometry is complemented remains an open problem till date. We shall answer this question in a later section of this article.

\smallskip

In \cite{Faulkner Huneycutt}, Faulkner and Huneycutt studied the Wold decomposition of a Banach space isometry. By a Wold isometry on a Banach space one naturally means an isometry that admits a Wold-type decomposition and consequently Wold isometry in Banach space setting was defined in the following way, see \cite{CFS, CFG, Faulkner Huneycutt, R 1} for details.

\begin{defn}
Let $V$ be an isometry on a Banach space $\mathbb{X}$. Then $V$ is said to be a \textit{Wold isometry}, if $\X$ admits a decomposition $\X=\X_1\oplus \X_2$ into invariant subspaces $\X_1, \X_2$ of $V$ such that $V|_{\X_1}$ is a unitary (i.e. a surjective isometry) and $V|_{\X_2}$ is a pure isometry. Here $\X= \X_1 \oplus \X_2$ means that every element $x\in \X$ can be uniquely expressed as $x= x_1 + x_2 $ for some $x_1 \in \X_1$ and $x_2 \in \X_2$. A Wold isometry $V$ is written as $V=V|_{\X_1} \oplus V|_{\X_2}$, which means $Vx=V|_{\X_1}x_1 \oplus V|_{\X_2}x_2 \ $, and this expression is called the \textit{Wold decomposition} of $V$.
\end{defn}

Note that for a Wold isometry $V$ on $\X$, the spaces $\X_1, X_2$ are uniquely determined as
\begin{equation}\label{Wold decomposition Banach spaces}
\X_1= \bigcap_{n= 0}^{\infty} V^n\mathbb{X}, \quad \X_2= \bigoplus_{n=0}^{\infty}V^n\mathcal{L},
\end{equation}
where $\mathcal L$ is vector space complement of $V\X$.
Unlike pure isometries on a Hilbert space, a pure Banach space isometry may not always be a unilateral shift as was shown in \cite{PAF}. Interestingly, if the underlying Banach space is reflexive, then every isometry with right-complemented range admits a Wold type decomposition and hence is a Wold isometry. This was proved in \cite[Corollary 1]{CFS}, which we state below.
\begin{thm} \label{thm:Wold Isometry-001}
Every isometry on a reflexive Banach space whose range is right-complemented is a Wold isometry.
\end{thm}
The range of any isometry on $\mathcal{L}_p$ space with $1\leq p \leq \infty$ is right-complemented and thus, any isometry on $\mathcal{L}_p$  is a Wold isometry, e.g. see \cite{Ando, Moslehian}.

\vspace{0.2cm}

\section{More on complemented subspaces in a Banach space} \label{sec:03}

\vspace{0.2cm}

\noindent The results of this Section are for reflexive, strictly convex and smooth Banach spaces. In Section \ref{sec:02}, we had a general discussion on complemented subspaces (both right and left-complemented subspaces) of a Banach space. Note that right or left complement of a complemented subspace in Banach space may not always be unique unlike Hilbert spaces. However, under the assumptions of reflexivity, strict convexity and smoothness, a (left or right) complemented subspace of a Banach space becomes unique. In this Section, we study in detail the left and right-complemented subspaces of a Banach space. We find several characterizations for a right-complemented (or left-complemented) subspace of a Banach space $\X$ in terms of linearity of Hahn Banach extension operator and the duality map (as in \ref{eqn:new-001}). Then we find explicit description of the right-complement (or the left-complement) of a right-complemented (or left-complemented) subspace. We begin with a duality relation between right-complemented and left-complemented subspaces in a Banach space. 

\begin{prop}\label{Characterization complemented: Corollary}
Let $\X$ be a Banach space and $\Y$ be a closed proper subspace of $\X$. Then the following hold.
\begin{itemize}
    \item[(i)] If $\Y$ is right-complemented in $\X$, then $\Ann(\Y)$ as in $(\ref{eqn:new-0003})$ is left-complemented in $\X^*$.\\
    \item[(ii)] If $\Y$ is left-complemented in $\X$, then $\Ann(\Y)$ is right-complemented in $\X^*$.
\end{itemize}
\end{prop}
A proof to this result follows from the fact that Banach adjoint of a norm-one projection on a Banach space $\X$ is again a norm-one projection on $\X^*$.

\smallskip

A fundamental result due to Calvert \cite[Theorem 1]{Calvert} states that a closed proper subspace $\Y$ of a reflexive, smooth and strictly convex Banach space $\X$ is right-complemented if and only if $\J_\X(\Y)$ is a subspace of $\X^*$. In our next result Theorem \ref{Characterization of Right-complemented}, we present an equivalent formulation of right-complemented subspaces in $\X$ which shows that $\J_\X(\Y)$ is isometrically isomorphic to $\Y^*$ via the Hahn-Banach extension operator. Thus, if $\Y$ is right-complemented in $\X$, then $\Y^*$ is isometrically isomorphic to a subspace of $\X^*$. Note that dual of a closed subspace of a Banach space need not be isometrically isomorphic to a subspace of its dual in general. We need a couple of lemmas for this purpose. We begin with the following lemma which can also be derived alternatively from a more general characterization of smoothness of multilinear maps proved in Theorem 3.1 of \cite{Saikat Roy}.

\begin{lem}\label{Lemma: Smoothness}
Let $\X$ be a reflexive Banach space. For any nonzero $f\in \X^*$, $f$ is smooth if and only if $M_f=\{\mu x_0:~|\mu|=1\}$ for some fixed $x_0\in S_\X$.
\end{lem}

\begin{proof}
Suppose that $M_f$ is of the stated form. Let $g_1,g_2$ be distinct nonzero members of $\X^*$ with $f\perp_B g_1$ and $f\perp_B g_2$. Then by James characterization (Theorem \ref{James Characterization}), there exist $\phi_1, \phi_2$ in $\X^{**}$ such that 
\[
    \phi_1(f) = \|\phi_1\|\|f\|, \quad \phi_2(f) = \|\phi_2\|\|f\|,\quad \phi_1(g_1)=\phi_2(g_2)=0.
\]
Since $\X$ is reflexive, $\phi_1=\pi_{\X}(x_1)$ and $\phi_2=\pi_{\X}(x_2)$, for some $x_1$ and $x_2$ in $\X$. Since $M_f=\{\mu x_0:~|\mu|=1\}$, this in particular shows that $x_1=x_2= x_0$. Thus, $\phi_1=\phi_2=\phi_0$, where $\phi_0=\pi_{\X}(x_0)$. Also, since $\phi_0(g_1+g_2)=0$, by James characterization, we have $f\perp_B (g_1+g_2)$. Therefore, $f$ is smooth by Theorem \ref{Orthogonality = Right-additivity}.

Conversely, suppose $f\in \X^*$ is smooth. Let if possible $\{x_0,y_0\}\subseteq M_f$, and $x_0\neq \mu y_0$, for any unimodular scalar $\mu$. By Hahn-Banach Theorem, we can find $p\in \X^*$ such that $p(x_0)=\|f\|\|x_0\|$ and $p(y_0)=0$. Clearly,
\[
   \|f+\lambda p\|\geq |f(y_0)+ \lambda p(y_0)|=|f(y_0)|=\|f\|,
\]
for all scalars $\lambda$. Thus, $f\perp_B p$. Again, since $f(x_0)-p(x_0)=0$, we have $f\perp_B (f-p)$ by similar arguments. Then, by Theorem \ref{Orthogonality = Right-additivity}, we have $f\perp_B f$, which is a contradiction. This completes the proof.
\end{proof}


\begin{lem}\label{Lemma: I}
Let $\X$ be a reflexive Banach space and $\Y$ be a closed proper subspace of $\X$. Then there is $x\in \X$ such that $x\perp_B \Y$.
\end{lem}

\begin{proof}

Let $x$ be a nonzero vector in $\X\setminus \Y$. Then, the one-point set $\{x\}$ is compact and is disjoint from $\Y$. Therefore, there is a convex neighbourhood $V$ of $\{\bf 0\}$ such that $(\Y+V)\cap (x+V)=\emptyset$. Now, an easy application of Geometric Hahn-Banach Theorem ensures the existence of a continuous linear functional $p$ such that $p(y)=0$ for all $y\in \Y$. Therefore, $p\in \Ann(\Y)$. Also, since $\X$ is reflexive, $M_p\neq \emptyset$. Let $x_0\in M_p$. Then for any $y\in \Y$, we have
\[
   \|x_0+\lambda y\|\geq \frac{1}{\|p\|}|p(x_0+\lambda y)|=\frac{1}{\|p\|}|p(x_0)|=\|x_0\|,
\]
for all scalars $\lambda$. Consequently, $x_0\perp_B \Y$.
\end{proof}

Now we are in a position to present the main result of this Section that gives various characterizations of a right-complemented subspace in a Banach space. Note that the equivalence of parts $(i)$ \& $(iv)$ of the following theorem was mentioned in \cite{SY} as `Observation (1)' and `Observation (2)' in case of a reflexive real Banach space $\X$. However, we deal with complex Banach spaces here.

\begin{thm}\label{Characterization of Right-complemented}
Let $\X$ be a reflexive, smooth and strictly convex Banach space and let $\Y$ be a closed proper subspace of $\X$. Then the following are equivalent.

\smallskip

\begin{itemize}
    \item[(i)] $\Y$ is right-complemented in $\X$.
    
    \smallskip
    
    \item[(ii)] The closed linear subspace 
   ${\displaystyle \bigcap_{f\in \X^*}\{\ker~f:~M_f\subseteq S_\Y\} } 
    $
    is a vector space complement of $\Y$.
    
    \item[(iii)] $\J_\X(\Y)$ is isometrically isomorphic to $\Y^*$.
    
    \smallskip

    \item[(iv)] The Hahn-Banach extension operator $\Psi:\Y^*\to \X^*$ is linear. 
\end{itemize}
\end{thm}

\begin{proof}
(i)$\implies$(ii). Since $\Y$ is right-complemented in $\X$, there is a subspace $\widetilde{\Y}$ of $\X$ such that $\X=\Y \bigoplus \widetilde{\Y}$ and $\Y\perp_B \widetilde{\Y}$. We show that $\widetilde{\Y}=\bigcap_{f\in \X^*}\{\ker~f:~M_f\subseteq S_\Y\}$. Let $z\in \widetilde{\Y}$ be arbitrary. Let $l\in \X^*$ with $M_l\subseteq S_\Y$. Note that
\begin{equation}\label{Support functional and norm attainment}
    \J_\X(\Y) = \{f\in \X^*:~M_f\subseteq S_\Y\}\cup \{\mathbf{0}\}.
\end{equation}
Thus, $l=\J_\X(y_0)$ for some $y_0\in \Y$. Since $\X$ is smooth, we have $J(y_0)=\{l\}$. Also, since $y_0\perp_B \widetilde{\Y}$, by Theorem \ref{James Characterization}, we have $ l(z)=0$ for all $z\in \widetilde{\Y}$. Consequently, ${ \displaystyle \widetilde{\Y}\subseteq \bigcap_{f\in \X^*}\{\ker~f:~M_f\subseteq S_\Y\} }$.

\smallskip

Conversely, let ${\displaystyle u_0\in \bigcap_{f\in \X^*}\{\ker~f:~M_f\subseteq S_\Y\} }$ be arbitrary. We first show that $\Y\perp_B u_0$. Let $y\in \Y$ be any nonzero vector. Consider the unique support functional $f_{y}$ at $y$. Then $f_y= \J_\X(y)$. It follows from (\ref{Support functional and norm attainment}) and the hypothesis that $u_0\in \ker~f_y$. Thus, by Theorem \ref{James Characterization}, we have $y\perp_B u_0$. Since $y\in \Y$ was nonzero and chosen arbitrarily, we have $\Y\perp_B u_0$. Evidently, $u_0=y_0+z_0$ for unique $y_0\in \Y$ and $z_0\in \widetilde{\Y}$. Since $y_0\perp_B z_0$ and $y_0\perp_B (y_0+z_0)$, applying homogeneity and right-additivity of Birkhoff-James orthogonality, we have $y_0\perp_B y_0$. Therefore, we have $y_0={\bf 0}$ and $u_0=z_0\in \widetilde{\Y}$. Altogether, we get $\widetilde{\Y}=\bigcap_{f\in \X^*}\{\ker~f:~M_f\subseteq S_\Y\}$. Also, it follows from (i) that $\widetilde{\Y}$ is a vector space complement of $\Y.$

\smallskip

\noindent (ii)$\implies$(iii).  Let 
$ {\displaystyle 
\widetilde{\Y} = \bigcap_{f\in \X^*}\{\ker~f:~M_f\subseteq S_\Y\}}$.
We first show that 
\begin{equation}\label{Equality of norm attainment with annihilator}
 \Ann(\widetilde{\Y}) = \J_\X(\Y). 
\end{equation} 
It follows from (\ref{Support functional and norm attainment}) that ${\displaystyle \widetilde{\Y} = \bigcap_{f\in \J_\X(\Y)}\ker~f}$. Therefore, $\Ann(\widetilde{\Y})\supseteq \J_\X(\Y)$. To prove the other side of the inclusion, consider any $g\in \Ann(\widetilde{\Y})$. Since $\X$ is reflexive and strictly convex, $g$ attains its norm at a unique (up to unimodular scalar multiple) vector $x_0$ of $S_\X$. Now, $x_0$ can be uniquely expressed as $x_0=y_0+z_0$ for some $y_0\in \mathbb{Y}$, $z_0\in \widetilde{\Y}$. Therefore,
$
   |g(y_0+z_0)| = |g(y_0)| = \|g\|.
$
Evidently, $\|y_0\|\geq 1$. Consider the unique support functional $f_{y_0}$ at $y_0$. Then by Lemma \ref{Lemma: Smoothness}, we have  
\[
 M_{f_{y_0}} = \left\{\mu\frac{y_0}{\|y_0\|}:~|\mu|=1\right\}\subseteq S_\Y.
\]
It now follows from the hypothesis that $f_{y_0}(z_0)=0$. Thus, $y_0\perp_B z_0$ and we have
$
  1=\|y_0+z_0\| \geq \|y_0\|.
$
This shows that $\|y_0\|=1$. By the strict convexity of $\X$, we have $x_0=y_0$. Consequently, by (\ref{Support functional and norm attainment}), we have $g\in \J_\X(\Y)$ and thus, $\J_\X(\Y)\subseteq \Ann(\widetilde{\Y})$.
Now, consider the map $\phi: \Ann(\widetilde{\Y})\to \X^*/\Ann(\Y)$ defined by
\[
   \phi(l)=l+\Ann(\Y), \qquad l\in \Ann(\widetilde{\Y}).
\]
It follows from  (\ref{Support functional and norm attainment}) and (\ref{Equality of norm attainment with annihilator}) that every $l\in \Ann(\widetilde{\Y})$ attains its norm at some $\widetilde{y}\in S_\mathbb{Y}$. Thus, for any $l'\in \Ann(\mathbb{Y})$, we have
$
   \ker l'\cap M_{l}\neq \emptyset.
$
Since $l'\in \Ann(\Y)$ is arbitrary, it follows from Lemma \ref{Orthogonality of functionals} that $l\perp_B \Ann(\Y).$ Therefore,
$
   \|l+g\|\geq \|l\|$ for all $g\in \Ann(\mathbb{Y})$. This shows that
\[
  \|l\|\geq \inf\{\|l+f\|:~f\in \Ann(\mathbb{Y})\} =\|l+\Ann(\mathbb{Y})\|\geq \|l\|.
\] 
Consequently, $\phi$ is an isometry.

\smallskip

Next, we show that $\phi$ is surjective. Let $p+\Ann(\Y)\in \X^*/\Ann(\Y)$ be arbitrary. Then $p=f_0+g_0$ for some unique $f_0\in \Ann(\widetilde{\Y})$ and $g_0\in \Ann(\mathbb{Y})$ and this happens because,
$
   \X^*=\Ann(\widetilde{\Y})\bigoplus \Ann(\mathbb{Y}),
$
by virtue of $\X = \mathbb{Y} \bigoplus \widetilde{\Y}$. Thus,
\[
  \phi(f_0)=f_0+\Ann(\mathbb{Y}) = p+\Ann(\mathbb{Y}).
\]
Consequently, $\phi$ is a surjective isometry as desired. Let $\eta:\X^*/\Ann(\mathbb{Y})\to \mathbb{Y}^*$ be the canonical isometric isomorphism. Therefore, $\eta\circ \phi$ induces an isometric isomorphism from $\J_\X(\Y)$ to $\mathbb{Y}^*$, which is to say that the following diagram commutes.
\begin{displaymath}
  \xymatrix
  {
    {\Ann(\widetilde{\Y})=\J_\X(\Y)}\ar[rr]^{\phi} \ar@{.>}[dd]_{\eta\circ \phi} & & {\X^*/\Ann\mathbb{Y} \ar[ddll]^{\eta}} \\ \\
    {\mathbb{Y}^* } & &  \\
  }
\end{displaymath}

\bigskip

\noindent (iii)$\implies$(iv). Evidently, $\J_\X(\Y)$ is a linear subspace of $\X^*$. Let $f,g\in\mathbb{Y}^*$ and let $\widetilde{f},\widetilde{g}$ be the Hahn-Banach extensions of $f, g$ respectively. Note that Hahn-Banach extension is unique because, $\X^*$ is strictly convex as $\X$ is reflexive and smooth. Let $h$ be the Hahn-Banach extension of $f+g.$ We show that $h=\widetilde{f}+\widetilde{g}.$ By the uniqueness of Hahn-Banach extension, it is enough to show that $\|h\|=\|\widetilde{f}+\widetilde{g}\|.$ Clearly, $h$ attains its norm in $\Y$ as $(f+g)$ does so. Moreover, the fact that
\[
   h(y)=f(y)+g(y)=\widetilde{f}(y)+\widetilde{g}(y), \qquad y\in \mathbb{Y},
\]
shows that $M_{h}\cap\ker \left(h-\left(\widetilde{f}+\widetilde{g}\right)\right)\neq\emptyset.$
Therefore, $h\perp_B (h-(\widetilde{f}+\widetilde{g}))$, and hence

\begin{equation}\label{Equation 1}
       \|h\|\leq \left\|\widetilde{f}+\widetilde{g}\right\|.
\end{equation}
On the other hand, since $\J_\X(\Y)$ is a linear subspace and $\widetilde{f},\widetilde{g}\in 
\J_\X(\Y),$ we have $\widetilde{f}+\widetilde{g} \in \J_\X(\Y).$ Let 
$\widetilde{f}+\widetilde{g}$ attain its norm at $y_0\in\mathbb{Y}.$ Then,

\begin{equation}\label{Equation 2}
      \left\|\widetilde{f}+\widetilde{g}\right\|=\left|(\widetilde{f}+\widetilde{g})(y_0)\right|=|h(y_0)|\leq\|h\|.
\end{equation}
Therefore, by (\ref{Equation 1}) and (\ref{Equation 2}), we have 
$\|h\|=\|\widetilde{f}+\widetilde{g}\|$. Consequently, $\Psi(f+g)=h=\widetilde{f}+\widetilde{g}.$ Also, we have $\Psi(\alpha f)=\alpha \widetilde{f}$, for every scalar $\alpha$. Thus, $\Psi$ is a linear map.

\medskip

\noindent (iv)$\implies$(i). We only show that $\J_\X(\Y)$ is a linear subspace of $\X^*$ and the rest of the proof follows from \cite[Theorem 1]{Calvert}. To this end, consider $\widetilde{f},\widetilde{g}\in \J_\X(\Y)$ and $\widetilde{y_{0}},\widetilde{y_{1}}\in S_{\mathbb{Y}}$ such that 
$\widetilde{f}(\widetilde{y_{0}})=\|\widetilde{f}\|,~ \widetilde{g}(\widetilde{y_{1}})=\|\widetilde{g}\|.$ Let $f=\widetilde{f}|_{\mathbb{Y}}$ and $g=\widetilde{g}|_{\mathbb{Y}}.$ Then $f,g\in\mathbb{Y}^*$ and
\[
\|f\|\leq\left\|\widetilde{f}\right\|=\widetilde{f}(\widetilde{y}_0)=f(\widetilde{y}_0)\leq\|f\|.
\]
Thus, $\widetilde{f}$ is a Hahn-Banach extension of $f$. Moreover, $\widetilde{f}$ is unique because $\X^*$ is strictly convex. Consequently, $\Psi(f)=\widetilde{f}$. Similarly, we have $\Psi(g)=\widetilde{g}$ and thus the linearity of $\Psi$ shows that $\widetilde{f}+\widetilde{g}$ is the unique Hahn-Banach extension of $f+g.$ Therefore, $\widetilde{f}+\widetilde{g}$ attains its norm in $\Y$ and it follows from (\ref{Support functional and norm attainment}) that $\widetilde{f}+\widetilde{g} \in \J_\X(\Y).$ Also, $\alpha f\in\J_\X(\Y)$ for any scalar $\alpha$ and $f\in\J_\X(\Y).$ Consequently, $\J_\X(\Y)$ is a linear subspace of $\X^*$.
This completes the proof.
\end{proof}

Naturally, we are led to find an analogue of Theorem \ref{Characterization of Right-complemented} for a left-complemented subspace in a Banach space what we present below.
\begin{thm}\label{Characterization of left-complemented}
Let $\X$ be a reflexive, smooth and strictly convex Banach space and let $\Y$ be a closed proper subspace of $\X$. Then the following are equivalent.

\smallskip

\begin{itemize}
    \item[(i)] $\Y$ is left-complemented in $\X$.
    
    \smallskip
    
    \item[(ii)] There is a non-trivial subspace $\widetilde{\Y}$ in $\X$ such that
    \[
      S_{\widetilde{\Y}} = \bigcup_{f\in \Ann(\Y)\setminus \{\mathbf{0}\}} M_f.
    \]
    \item[(iii)] There is a closed subspace $\widetilde{\Y}_1$ of $\X$ such that the restriction map $f\mapsto f|_{\widetilde{\Y}_1}$ from $\Ann(\Y)$ to $\widetilde{\Y}_1^*$ is a surjective isometry.
    
    \smallskip
    
    \item[(iv)] There is a subspace $\widehat{\Y}_1$ in $\X$ such that every $f\in \Ann(\Y)$ attains its norm in $\widehat{\Y}_1$ and the linear map $\Theta:\X^*\to \widehat{\Y}_1^*$ defined by
    $
     \Theta(f) = f|_{\widehat{\Y}_1}
    $
    has norm one, and the norm attainment set $M_\Theta$ of $\Theta$ is given by
    $
     M_\Theta =\{f\in S_{\X^*}:~f\in \Ann(\Y)\}.
    $
\end{itemize}
\end{thm}

\begin{proof}
(i)$\implies$(ii). Since $\Y$ is left-complemented in $\X$, there exists a closed linear subspace $\widetilde{\Y}$ of $\X$ such that $\X=\widetilde{\Y}\bigoplus \Y$ and $\widetilde{\Y}\perp_B \Y$. We show that $S_{\widetilde{\Y}} = \bigcup_{f\in \Ann(\Y)\setminus\{\bf 0\}} M_f.$ Since $\X=\widetilde{\Y}\bigoplus_\SL \Y$, by Proposition \ref{Characterization complemented: Corollary}, we have $\X^*=\Ann(\Y)\bigoplus_\SR \Ann(\widetilde{\Y})$. Therefore, it follows from (\ref{Equality of norm attainment with annihilator}) that
\begin{equation}\label{Equality of norm attainment with annihilator: Dual}
\J_{\X^*}(\Ann(\Y)) = \Ann(\Ann(\widetilde{\Y})).   
\end{equation}
Since $\widetilde{\Y}$ is closed, by an easy application of Geometric Hahn-Banach Theorem, we have
\[
 \Ann(\Ann(\widetilde{\Y})) = \{\pi_\X(x)\in \X^{**}:~f(x)=0,~f\in \Ann(\widetilde{\Y})\} = \pi_\X(\widetilde{\Y}).
\]
Let $f\in \Ann(\Y)$ be arbitrary and let $J(f)=\pi_\X(y)$ for some $y\in \X$. Then by (\ref{Equality of norm attainment with annihilator: Dual}), we have $\pi_\X(y)\in \pi_\X(\widetilde{\Y})$ and so $y\in \widetilde{\Y}$. Also,
\[
 f\left(\frac{y}{\|y\|}\right) = \frac{1}{\|y\|}\pi_\X(y)(f) = \frac{1}{\|y\|}\|f\|^2 = \|f\|.
\]
So, we have $M_f\subseteq S_{\widetilde{\Y}}$. On the other hand, for $z\in S_{\widetilde{\Y}}$, we have $\pi_\X(z)\in \J_{\X^*}(\Ann(\Y))$ by (\ref{Equality of norm attainment with annihilator: Dual}). Let $\pi_\X(z)=\J_{\X^*}(g)$ for some $g\in \Ann(\Y).$ Since $\J_{\X^*}$ is norm preserving, we have
\[
 1 = \|z\|^2 = \|g\|^2 = \pi_\X(z)(g) = g(z).
\]
Therefore, $z\in M_g$ and the desired equality follows.

\medskip

\noindent (ii)$\implies$(iii). Let $\widetilde{\Y}=\widetilde{\Y}_1$. We first show that $\widetilde{\Y}$ is closed. We claim that
$
  \widetilde{\Y}=\{x\in \X:~x\perp_B \Y\}.
$
Let $z\in \widetilde{\Y}$ be arbitrary. It follows from the hypotheses that $\dfrac{z}{\|z\|}\in M_f$ for some $f\in \Ann(\Y)$. Then for any $y\in \Y$, we have
\[
  \|f\|\|z+\lambda y\|\geq |f(z+\lambda y)|=|f(z)|=\|f\|\|z\|,
\]
for all scalars $\lambda$. Thus, $z\perp_B y$ for any $y\in \Y$. Since $z\in \widetilde{\Y}$ is arbitrary, we have $\widetilde{\Y}\perp_B \Y$. By forgoing arguments $\widetilde{\Y}\subseteq \{x\in \X:~x\perp_B \Y\}$. To prove the other side of the inclusion, consider any $x\in \X$ such that $x\perp_B \Y$. Define $g:\mathrm{span}\{x, \Y\}\to \mathbb{C}$ by
$
   g(\alpha x+y)=\alpha\|x\|.
$
It is not difficult to see that the Hahn-Banach extension $\widetilde{g}$ of $g$ annihilates $\Y$ and attains norm at ${x}/{\|x\|}$. Therefore, it follows from the hypotheses that ${x}/{\|x\|}\in S_{\widetilde{\Y}}$. Consequently, $x\in \widetilde{\Y}$ as desired. Now, we consider any sequence $(z_n)$ in $\widetilde{\Y}$ that converges to some $z_0\in \X$. Then for any $y\in \Y$ and $\lambda \in \C$, we have
$
  \|z_n+\lambda y\|\geq \|z_n\|$.
By continuity of the norm, we have $\|z_0+\lambda y\|\geq \|z_0\|$, for all scalars $\lambda$. Since $y\in \Y$ is arbitrary, we have $z_0\perp_B \Y$. Therefore, $z_0\in \widetilde{\Y}$, i.e., $\widetilde{\Y}$ is closed, and the claim is proved.

\smallskip

Next, consider the map $\phi: \Ann(\Y)\to \X^*/ \Ann(\widetilde{\Y})$, defined by
$
 \phi(f) = f+\Ann(\widetilde{\Y})$. By the hypotheses, we have $M_f\subseteq S_{\widetilde{\Y}}$ for every $f\in\Ann(\Y).$ Thus, it follows from Lemma \ref{Orthogonality of functionals} that $f\perp_B \Ann(\widetilde{\Y})$. Now, by an argument similar to that in the proof of (ii)$\implies$(iii) of Theorem \ref{Characterization of Right-complemented}, we can show that $\phi$ is an isometry. Let $\widetilde{\eta}: \X^*/\Ann(\widetilde{\Y})\to \widetilde{\Y}^*$ be the canonical isometry. Thus, $\widetilde{\eta}\circ \phi:\Ann(\Y)\to \widetilde{\Y}^*$ given by
$
  \widetilde{\eta}\circ \phi(f) =f|_{\widetilde{\Y}}$
induces an isometry. So, the following diagram commutes.
\begin{displaymath}
    \xymatrix
    {
     {\Ann(\Y)}\ar[rr]^\phi \ar@{.>}[dd]_{\widetilde{\eta}\circ\phi} & & {\X^*/\Ann(\widetilde{\Y})} \ar[ddll]^{\widetilde{\eta}} \\\\
     \widetilde{\Y}^* & &
    }
\end{displaymath}

\bigskip

\noindent (iii)$\implies$(iv). Let $\widehat{\Y}_1=\widetilde{\Y}_1$. Consider any $f\in \Ann(\Y)$ and let $\hat{f} = f|_{\widehat{\Y}_1}$. Evidently, $\widehat{\Y}_1$ is reflexive and $\hat{f}$ attains its norm at some unit vector $z_0\in S_{\widehat{\Y}_1}$. However, by the hypothesis
\[
 \|f\| = \|f|_{\widehat{\Y}_1}\| = \|\hat{f}\| = |\hat{f}(z_0)| = |f(z_0)|.
\]
Therefore, $z_0\in M_f$. Consequently, any member $f\in \Ann(\Y)$ attains its norm in $\widehat{\Y}_1$. To prove the next part, first observe that
$
 \|\Theta(g)\| = \|g|_{\widehat{\Y}_1}\| \leq \|g\|$ for every $g\in \X^*$ and for any unit vector $f\in \Ann(\Y)$
we have $
 \|\Theta(f)\| = \|f|_{\widehat{\Y}_1}\| = \|f\| = 1.
$
Thus, $\|\Theta\|=1$. Moreover, $\{f\in S_{\X^*}:~f\in \Ann(\Y)\}\subseteq M_\Theta$. Also, for any $l\in M_\Theta$ we have
$
 \|l\|= 1 = \|\Theta(l)\| = \|l|_{\widehat{\Y}_1}\|.
$
Therefore, $l$ is a Hahn-Banach extension of $l|_{\widehat{\Y}_1}$.  Let $\hat{l} = l|_{\widehat{\Y}_1}.$ Then $\hat{l}\in \widehat{\Y}_1^*$ and by the hypotheses there exists $l'\in \Ann(\Y)$ such that $l'|_{\widehat{\Y}_1}=\hat{l}$ with $\|l'\| = \|\hat{l}\|$. Therefore, $l'$ is also a Hahn-Banach extension of $l$. By the uniqueness of the Hahn-Banach extension, we have $l=l'$. Consequently, $l\in \Ann(\Y)$ and $M_\Theta \subseteq \{f\in S_{\X^*}:~f\in \Ann(\Y)\}$, and the desired equality is proved.

\smallskip

\noindent (iv)$\implies$(i). We first show that $\widehat{\Y}_1=\{x\in \X:~x\perp_B \Y\}.$ Let $x_0 \in \X$ be such that $x_0\perp_B \Y$. Then the unique support functional $f_{x_0}$ at $x_0$ vanishes on $\Y$. Thus, by the hypotheses we have $x_0\in \widehat{\Y}_1$. On the other hand, for any $z_0\in \widehat{\Y}_1$, we have
\[
 \frac{1}{\|z_0\|}\|f_{z_0}\|\geq \left\|\Theta\left(\frac{1}{\|z_0\|}f_{z_0}\right)\right\| = \frac{1}{\|z_0\|} \left\|f_{z_0}|_{\widehat{\Y}_1}\right\| \geq \frac{1}{\|z_0\|^2}\left|f_{z_0}|_{\widehat{\Y}_1} (z_0)\right| = \frac{1}{\|z_0\|^2}|f_{z_0} (z_0)| = \frac{1}{\|z_0\|}\|f_{z_0}\|.
\]
Thus, $\frac{1}{\|z_0\|}f_{z_0}\in M_\Theta$ and consequently $f_{z_0}\in \Ann(\Y)$. So, $z_0\perp_B \Y$. It now follows from the proof of (i)$\implies$(ii) that $\widehat{\Y}_1$ is a closed subspace of $\X$. Observe that $\widehat{\Y}_1\bigoplus \Y=\{z+y:~z\in \widehat{\Y}_1,~y\in \Y\}$ is a closed subspace of $\X$. Indeed, for any sequence $(z_n+y_n)$ in $\widehat{\Y}_1\bigoplus \Y$ that converges to some $x_0\in \X$, the sequence $(z_n)$ is Cauchy, because 
\[
   \|z_n+y_n-z_m-y_m\|\geq \|z_n-z_m\|, \qquad n\in \mathbb{N}.
\]
Since $\widehat{\Y}_1$ is closed, $(z_n)$ converges to $z_0$, for some $z_0\in \widehat{\Y}_1$. Consequently, $(y_n)$ converges to $(x_0-z_0)$ and $(x_0-z_0)\in \Y$. Therefore, $x_0\in \widehat{\Y}_1\bigoplus \Y$. Now, if $\widehat{\Y}_1\bigoplus \Y$ is a proper subspace of $\X$, then by Lemma \ref{Lemma: I}, there is a nonzero $u_0\in \X$ such that $u_0\perp_B (\widehat{\Y}_1\bigoplus \Y)$. In particular, $u_0\perp_B \Y$. Then $u_0\in \widehat{\Y}_1$, which is a contradiction. Thus, $\widehat{\Y}_1$ is a vector space complement of $\Y$ and $\X=\widehat{\Y}_1\bigoplus_\SL \Y$. The proof is now complete.
\end{proof}

It is well-known (e.g. see \cite{Bohenblust, Kakutani}) that if $\X$ is a Banach space with $dim~\X \geq 3$, then every closed subspace of $\X$ is right-complemented if and only if $\X$ is a Hilbert space. Also, every one-dimensional subspace in a Banach space is left-complemented if and only if the space is a Hilbert space, see \cite[Theorem 5]{James 2}. However, by James characterization \cite{James 1}, every one-dimensional subspace is right-complemented. Thus, it follows from Proposition \ref{Characterization complemented: Corollary} that right-complemented (or left-complemented) subspaces of a Banach space are not necessarily left-complemented (or right-complemented).
The following example shows that there are closed linear subspaces which are both right-complemented and left-complemented in a Banach space. 

\begin{eg}
Let $(\Omega,\Sigma, \mu)$ be a measure space and $\X$ be a Banach space over $\mathbb{C}$. For $1\leq p < \infty$, consider the Banach space
\[\mathcal{L}_{p}(\Omega,\Sigma,\mu):=\left\{f:\Omega\to\X\;|\; f\mbox{ measurable and }\int_{\Omega}\|f(x)\|^pd\mu<\infty\right\},\]
equipped with the norm
\[
  \|f\|_p=(\int_{\Omega}\|f(x)\|^pd\mu)^{\frac{1}{p}}, \qquad f\in \mathcal{L}_{p}(\Omega,\Sigma,\mu).
\]
Let $A\in\Sigma$ with $\mu(A)\neq 0.$ Then the set $B=\Omega\setminus A$ belongs to $\Sigma$ too. Consider the subspace
\[
  \mathbb{Y}:=\{\chi_{A}f ~ :~ f\in\mathcal{L}_{p}(\Omega,\Sigma,\mu)\},
\]
consisting of elements $f\in\mathcal{L}_{p}(\Omega,\Sigma,\mu)$ with support inside $A$ and the subspace
\[
  \widetilde{\Y}:=\{\chi_B g ~ : ~g\in\mathcal{L}_{p}(\Omega,\Sigma,\mu)\},
\] 
consisting of elements $g\in\mathcal{L}_{p}(\Omega,\Sigma,\mu)$ with support inside $B=\Omega\setminus A.$
Evidently, both of $\mathbb{Y}$ and $\widetilde{\Y}$ are closed subspaces of $\mathcal{L}_{p}(\Omega,\Sigma,\mu).$
Note that every element $f\in \mathcal{L}_{p}(\Omega,\Sigma,\mu)$ can be uniquely written as $f = f_1 + f_2$ with $f_1\in\mathbb{Y}$ and $f_2\in\widetilde{\Y}$, where $f_1 =\chi_A f$ and $f_2 = \chi_B f.$ 
In particular, we get $\mathbb{Y}\oplus\widetilde{\Y}=\mathcal{L}_{p}(\Omega,\Sigma,\mu).$ Now, for $f\in\mathbb{Y},g\in\widetilde{\Y}$ and $\lambda\in\mathbb{C}
$ we have,
\begin{align*}
\|f+\lambda g\|_{p}^{p} & =\int_{\Omega}\|f(x)+\lambda g(x)\|^pd\mu(x)\\
                &=\int_{A}\|f(x)+\lambda g(x)\|^pd\mu(x)+\int_{B}\|f(x)+\lambda g(x)\|^pd\mu(x)\\
                &=\int_{A}\|f(x)\|^pd\mu(x) + \int_{B}\|\lambda g(x)\|^pd\mu(x)\\
                &=\|f\|_{p}^{p}+\vert\lambda\vert^p\|g\|_{p}^{p}\\
                &\geq \|f\|_{p}^{p}.
\end{align*}
Similarly, we have $\|g+\lambda f\|_{p}^{p}\geq \|g\|_{p}^{p}$ for all scalars $\lambda\in \mathbb{C}$.
Thus, the vector space decomposition is also consistent with the Birkhoff-James orthogonality, i.e.
\[
  \mathcal{L}_{p}(\Omega,\Sigma,\mu)=\mathbb{Y}\bigoplus_\SL \widetilde{\Y}=\widetilde{\Y}\bigoplus_\SR\mathbb{Y}.
\]  \qed
\end{eg}

Right-complement (or left-complement) of a closed subspace may not be unique as the following example shows.

\begin{eg} \label{exm:new-001}

Let $\X$ be the $3$-dimensional Banach space $(\mathbb{C}^3,\|\cdot\|_\infty)$, where
\[
 \|(x,y,z)\|_\infty = \max\{|x|,|y|,|z|\}.
\]
Let $\Y = span\{(1,1,1)\}.$ Then $\widetilde{\Y}:=span\{(1,0,0),(0,1,0)\}$ is a vector space complement of $\Y$. Moreover,
\[\|(1+\lambda, 1+\mu, 1)\|_\infty \geq 1 = \|(1,1,1)\|_\infty, \quad \lambda,\mu\in \mathbb{C}.\]
Consequently, $\widetilde{\Y}$ is a right-complement of $\Y$ in $\X$. Similarly, it can be shown that the subspace $\Z':=span\{(0,1,0),(0,0,1)\}$ is a right-complement of $\Y$. Also, by Proposition \ref{Characterization complemented: Corollary} both $\Ann(\widetilde{\Y})$ and $\Ann(\Z')$ are left-complements of $\Ann(\Y)$ in $\X^*$.
\qed 
\end{eg} 

\begin{rem}
Example \ref{exm:new-001} shows that the right-complemented (or left-complemented) subspaces of a Banach space may not be the range (or kernel) of a unique norm one projection. However, it is unique if $\X$ is reflexive, smooth and strictly convex as was shown in \cite[Lemma 2]{Calvert}.
\end{rem}

\vspace{0.05cm}

\section{Range of an isometry and Wold isometry} \label{sec:04}

\vspace{0.2cm}

\noindent
The range of a Banach space isometry may not be a right-complemented (or $1$-complemented) subspace as was shown in \cite{Ditor} by an example. In \cite{Faulkner Huneycutt}, Faulkner and Huneycutt asked when the range of a Banach space isometry is right-complemented, see Question 2 in \cite{Faulkner Huneycutt}. Here we obtain a characterization for an isometry on a reflexive Banach space with right-complemented range and thus answer the question posed in \cite{Faulkner Huneycutt} for reflexive Banach spaces. Again, Theorem \ref{thm:Wold Isometry-001} tells us that an isometry on a reflexive Banach space with right-complemented range is a Wold isometry. So, in this way we characterize a Wold isometry too on a reflexive Banach space. Also, we apply the results of Section \ref{sec:03} to find several characterizations for an isometry with right-complemented range that acts on a reflexive, smooth and strictly convex Banach space. These characterizations give rise to a few sufficient conditions such that an isometry on a reflexive, smooth and strictly convex Banach space becomes a Wold isometry. We begin with an elementary lemma.

\begin{lem} \label{lem:new-041}
Let $V$ be a Wold isometry on a Banach space $\X$ with Wold decomposition $V=V|_{\X_1} \oplus V|_{\X_2}$, where $\X_1$ is the maximal invariant subspace of $V$ such that $V|_{\X_1}$ is a unitary and $V|_{\X_2}$ is a pure isometry. If the range of $V$ is right-complemented, then $V|_{\X_2}$ is a unilateral shift.
\end{lem}

\begin{proof}
Let $\mathcal{L}$ be the right-complement of $V(\X)$ in $\X$. Then for all $n\in \mathbb{N}$, $V^n \mathcal{L} \perp_B \mathcal{L}$  as $V^n \mathcal{L} \subseteq V(\X)$. So, by Lemma \ref{Wandering Space} the conclusion follows.
\end{proof}

We have already discussed in Section \ref{sec:02} that the range of an isometry is not necessarily right-complemented. Even the range of a Wold isometry may not always be right complemented. Indeed, every unilateral shift is a pure isometry and hence is a Wold isometry, but one can have an example of a unilateral shift whose range is not right-complemented.
\begin{eg} \label{exm:new-041}
Let us consider the disk algebra $\mathcal A$ consisting of the continuous functions $f:\overline{\mathbb{D}}\to \mathbb{C}$ that are analytic on $\mathbb D$ and are equipped with the supremum norm. For the function $\phi(z)=z$, consider the multiplication operator $M_\phi$ on $\mathcal A$ defined by 
$
M_\phi f(z) = \phi(z)f(z)=zf(z).
$
Let $\mathcal{L}$ be the linear subspace of $\mathcal A$ consisting of constant functions. We show the following:
\begin{itemize}
    \item [(i)] $M_{\phi}$ is an isometry ;
    \item[(ii)] ${\displaystyle \mathcal A=\bigoplus_{n=0}^\infty \, M_\phi^n \mathcal{L}}$ ;
    \item[(iii)] $M_\phi ^n \mathcal{L}\perp_B M_\phi^m \mathcal{L}$ for every $n>m\geq 0$ ;
    \item[(iv)] the range of $M_\phi$ is not right-complemented. 
\end{itemize}
This is same as showing that $M_\phi$ is a unilateral shift whose range is not right-complemented.

\medskip

\noindent (i) For any $f\in \mathcal A$, we have
\[
\|M_\phi f\|=\sup_{|z|\leq 1}|z.f(z)|= \sup_{|z|= 1}|z||f(z)|=\|f\|_{\infty, \overline{\mathbb D}}=\|f\|
\]
and hence $M_{\phi}$ is an isometry.

\smallskip

\noindent (ii) Given any $f\in \mathcal A$, we consider its Taylor series which is given by
\[
f(z)=\sum_{n=0}^\infty \frac{f^n(0)}{n!}z^n=\sum_{n=0}^\infty M_\phi^n \left(\frac{f^n(0)}{n!}\right)=\sum_{n=0}^\infty M_\phi^nl_n,
\]
where $l_n$ are the constant functions $l_n(z)=\frac{f^n(0)}{n!}$ for all $z$. Therefore, we have
$
{ \displaystyle \mathcal A=\bigoplus_{n=0}^\infty M_\phi^n \mathcal{L}.}
$

\smallskip

\noindent (iii) Since Birkhoff-James orthogonality is homogeneous, it is enough to show that $z^n\perp_B z^m$ for every $n> m \geq 0$. We use a characterization of Birkhoff-James orthogonality of continuous functions proved in \cite[Theorem 2.2]{RB22}. Let $p_n(z)=z^n$ and $p_m(z)=z^m$. Then by \cite[Theorem 2.2]{RB22}, we have
\[
p_n(z)\perp_B p_m(z) \iff 0\in \text{Convex hull of the set} \left\{p_n(x)\overline{P_m(x)}:~|x|\leq 1,~|p_n(x)|=\|p_n\|=1\right\}.
\]
However, since 
\[
\left\{p_n(x)\overline{P_m(x)}:~|x|\leq 1,~|p_n(x)|=\|p_n\|=1\right\}=\left\{p_n(x)\overline{P_m(x)}:~|x|= 1,~|p_n(x)|=\|p_n\|=1\right\},
\]
we have
\[
\left\{p_n(x)\overline{P_m(x)}:~|x|= 1,~|p_n(x)|=\|p_n\|=1\right\} = \{z^{n-m}:~|z|=1\} = \mathbb T,
\]
where $\mathbb T$ is the unit circle. Note that convex hull of $\mathbb T$ contains $0$. Consequently, $p_n(z)\perp_B p_m(z)$ for every $n>m\geq 0$.

\medskip

\noindent (iv) Again, we apply Theorem 2.2 in \cite{RB22}. Consider $f(z)=z+z^2 =M_\phi(1+z)\in M_\phi(\mathcal A)$. We show that $f\not\perp_B {\bf 1}$, where ${\bf 1}$ denotes the constant function $1$. By \cite[Theorem 2.2]{RB22}, we have
\begin{align*}
f\perp_B {\bf 1} & \iff 0\in \text{Convex hull of the set} \left\{f(z)\overline{{\bf 1}(z)}:~|z|\leq 1,~|f(z)|=\|f\|\right\}\\
 i.e., \, \,  f\perp_B {\bf 1} & \iff 0\in \text{Convex hull of the set}\left\{f(z):~|z|\leq 1,~|f(z)|=\|f\|\right\}.
\end{align*}
It is not difficult to see that $|f(z)|\leq 2$ and equality holds if and only if $z=1$. Thus, we have
\[
\text{Convex hull of the set}\left\{f(z):~|z|\leq 1,~|f(z)|=\|f\|\right\} = \{2\}.
\]
Consequently, $f\not \perp_B {\bf 1}$. If the range of $M_\phi$ were right-complemented, then there would exist a norm-one projection $P$ onto the range of $M_{\phi}$ such that $\text{range}(P)\perp_B \text{range}(I-P)$. However, $\text{range}(I-P)$ only consists of the constant functions and we already have $f\not \perp_B {\bf 1}$. So, the range of $M_{\phi}$ is not right-complemented. \qed

\end{eg} 

We learn from Theorem \ref{thm:Wold Isometry-001} that an isometry on a reflexive Banach space whose range is right-complemented range is a Wold isometry. This gives rise to an immediate question: is the range of an isometry on a reflexive Banach space always right-complemented ? It is well-known that in general the range of an isometry is not right-complemented. However, it makes sense to ask the question in the reflexive-Banach space setting. In \cite{P-B}, Pelczar-Barwacz constructed a reflexive Banach space $\X$ that has a proper closed linear subspace $\X_1$ such that $\X_1$ is isometrically isomorphic with $\X$ and $\X_1$ is not right-complement. In fact $\X_1$ is not complemented in $\X$. Thus, the inclusion map from $\X_1$ into $\X$ is an isometry without a right-complemented range in reflexive Banach space setting. Below we characterize an isometry on a reflexive Banach space whose range is right-complemented.

\begin{thm} \label{Left Inverse of Wold Isometry}
Let $\mathbb{X}$ be a reflexive Banach space and let $V$ be an isometry on $\X$. Then the following are equivalent:
\item[(i)] $V$ has right-complemented range ;

\item[(ii)] There is an operator $T: \mathbb{X} \to \mathbb{X}$ with $\|T\| =1$ such that $TV = I$.
\end{thm}

\begin{proof}
$(i) \implies (ii)$. Let $\mathcal{L}$ be a right-complement of $V(\mathbb{X})$ in $\mathbb{X}$. Since the range of $V$ is right-complemented, then by Theorem \ref{thm:Wold Isometry-001} we have that $V$ is a Wold isometry. Suppose
$ \X=\mathbb{X}_1 \oplus \X_2$,
where $\X_1$ and $\X_2$ are given by (\ref{Wold decomposition Banach spaces}).
Now, both $\mathbb{X}_1$ and $\mathbb{X}_2$ are invariant under $V$ and $V|_{\mathbb{X}_1}$ is a unitary and $V|_{\mathbb{X}_2}$ is a unilateral shift by Lemma \ref{lem:new-041}. Let $T_1$ be the linear operator $V|_{\mathbb{X}_1}:\mathbb{X}_1\to \mathbb{X}_1$ and define $T_2:\mathbb{X}_2\to \mathbb{X}_2$ by
\[
T_2\left(\sum_{n=0}^\infty V^nl_n\right)=\sum_{n=0}^\infty V^nl_{n+1}.
\] 
Define $T:\mathbb{X}\to \mathbb{X}$ by
$
T(x_1+x_2)=T_1^{-1}x_1+T_2x_2$ for $x_1\in \mathbb{X}_1,~x_2\in \mathbb{X}_2$. Evidently, $x_1=T_1y_1$ for some $y_1\in \mathbb{X}_1$ and $x_2=\sum_{n=0}^\infty V^nl_n$ for some sequence $(l_n) \subseteq \mathcal{L}$. Therefore,
\begin{align*}
\|T(x_1+x_2)\|  = \|T_1^{-1}x_1+T_2x_2\|
& = \left\|y_1+\sum_{n=0}^\infty V^nl_{n+1}\right\| \\
& = \left\|V(y_1+\sum_{n=0}^\infty V^nl_{n+1})\right\| \\
& \leq \left\|l_0+V(y_1+\sum_{n=0}^\infty V^nl_{n+1})\right\| \qquad [~V\mathbb{X}\perp_B \mathcal{L}~]; \\
& =\left\|T_1y_1+\sum_{n=0}^\infty V^nl_n\right\| \\
& = \left\|x_1+x_2\right\|.
\end{align*}
Therefore, $\|T\|\leq 1$. Also, since any $x\in \X$ can be uniquely expressed as $x=x_1+x_2$, for $x_1\in \X_1$ and $x_2\in \X_2$, we have that
\[
TV(x)=TV(x_1+x_2) = TVx_1 + TVx_2 = T_1^{-1}(Vx_1) + T_2\left(\sum_{n=0}^\infty V^{n+1}l_n\right) = x_1+x_2 =x.
\]
Consequently, $TV=I$ and $\|T\|=1.$

\medskip

\noindent $(ii) \implies (i)$. Consider the operator $P= VT$. Then $P$ and $V$ have the same range. Indeed, $P(\mathbb{X}) \subseteq V(\mathbb{X})$, and  for any $y\in V(\mathbb{X})$ with inverse image $x$, we have $Py = VTVx =Vx =y$. In addition, $P^2 = (VT) (VT) = V(TV) T = P$ and $\|P\| \leq \|V\| \|T\| =1$. Thus, $P$ is a norm-one projection on the range of $V$. Consequently, the range of $V$ is right-complemented by Proposition \ref{Range and kernel of norm one Projection}.
\end{proof}

Now we move to consider isometries acting mostly on reflexive, smooth and strictly convex Banach spaces. We start with two basic results which will be used in the sequel and are also of independent interests.

\begin{prop}\label{V cross f_Vx}
Let $\X$ be a reflexive, smooth and strictly convex Banach space and $V$ be an isometry on $\X.$ Then the map $\Phi:\J_\X\left(V\left(\X\right)\right)\to\X^*$ defined by the restriction of the Banach adjoint $V^\times$ on $\J_\X(V(\X))$, is a bijective map.
\end{prop}

\begin{proof}

 We show that
\begin{equation}\label{V cross f_Vx equals f_x}
  V^{\times}\left(f_{Vx}\right)=f_{x},\qquad x\in\X.
\end{equation}
For every $y\in S_{\X}$ we have
\[
  \vert (V^{\times}(f_{Vx}))(y)\vert=\vert f_{Vx}(Vy)\vert\leq\|f_{Vx}\|\|Vy\|=\|Vx\|=\|x\|.
\]
On the other hand 
\[
V^{\times}(f_{Vx})\left(\frac{x}{\|x\|}\right)=(f_{Vx})\left(\frac{Vx}{\|x\|}\right)=\frac{\|Vx\|^2}{\|x\|}=\|x\|.
\]
Therefore,
$
   \|V^{\times}(f_{Vx})\|=\|x\|$ and $V^{\times}(f_{Vx})(x)=\|Vx\|^2=\|x\|^2$. Thus, $V^\times(f_{Vx})$ is a support functional at $x$. Again, since $\X$ is smooth, we have that
$
   V^{\times}(f_{Vx})=f_x$. So, $\Phi$ can be described as 
\begin{equation}\label{Definition of Phi}
  \Phi(f_{Vx})=V^\times (f_{Vx})=f_x,\quad x\in\X.
\end{equation}
The map $\Phi$ is injective. Indeed, for $f_{Vx},f_{Vy}\in \J_\X\left(V\left(\X\right)\right)$ and $\Phi(f_{Vx})=\Phi(f_{Vy}),$ we have $f_x=f_y.$ Therefore,
$
    \|x\|=\|f_x\|=\|f_y\|=\|y\|.
$ 
Note that,
\[
  f_x\left(\frac{y}{\|y\|}\right)=f_y\left(\frac{y}{\|y\|}\right)=\|f_y\|=\|f_x\|=f_x\left(\frac{x}{\|x\|}\right).
\]
Since $\X$ is strictly convex, by Lemma \ref{Lemma: Smoothness}, $f_x$ attains its norm at a unique point (up to unimodular scalar multiple). So, it follows that $x=y$. Consequently, $Vx=Vy$ and $f_{Vx}=f_{Vy}$.
The map $\Phi$ is surjective, because by the reflexivity of $\X$, any nonzero element $f\in \X^*$ is of the form $f_x$. Thus, $\Phi(f_{Vx})=V^\times (f_{Vx})=f_x.$ This completes the proof.

\end{proof}

\begin{lem}\label{Norm attainment of V cross}
Let $\X$ be a reflexive Banach space and let $V$ be an isometry on $\X$. Then $\|V^{\times}(f)\|=\|f\|$ if and only if $f\in\J_\X\left(V\left(\X\right)\right)$.
\end{lem}

\begin{proof}
We first prove the necessity. Let $f\in\X^*$ be arbitrary. Since $\X$ is reflexive, $V^\times f$ attains its norm at a unit vector, say $x_0$. Then
\[
   \left\|V^\times f\right\|=\sup_{\|x\|=1} | f(V(x)) |=|f(V(x_0))|\leq \|f\| =\left\|V^\times f\right\|.
\]
Therefore, $f$ attains its norm at $V(x_0)$ and hence $f\in\J_\X(V(\X))$, see (\ref{Support functional and norm attainment}). We now prove the sufficiency. Since $f\in\J_\X(V(\X))$, $f$ attains its norm at $V(y_0)$ for some unit vector $y_0.$ Therefore,
$
   \|f\|=\left\|f|_{V(\X)}\right\|=\| V^\times f\|
$ and the proof is complete.
\end{proof}

Now, we characterize the isometries acting on reflexive, smooth and strictly convex Banach spaces whose ranges are right-complemented.

\begin{thm}\label{Isometry Range}
Let $\X$ be a reflexive, smooth and strictly convex Banach space and $V$ be a non-surjective isometry on $\X.$ Then the  following are equivalent:

\begin{itemize}
\item[(i)] $V(\X)$ is right-complemented ;

\smallskip

\item[(ii)] $M_{V^\times}=S_{\mathbb{W}}$ for some subspace $\mathbb{W}$ of $\X^*$ ;

\smallskip

\item[(iii)] $\|f_x+f_y\|=\|f_{Vx}+f_{Vy}\|\quad \text{for all}~x,y\in\X$ ;

\smallskip

\item[(iv)] The map $\Phi^{-1}:\X^*\to\J_\X\left(V\left(\X\right)\right)$ is linear, where $\Phi$ is as in $(\ref{Definition of Phi})$.

\end{itemize}
\end{thm}

\begin{proof}
(i)$\implies$(ii). Since $V(\X)$ is right-complemented, $\J_\X(V\X)$ is a linear subspace of $\X^*$ by \cite[Theorem 1]{Calvert}. Let $\mathbb{W}=\J_\X\left(V\left(\X\right)\right).$ Then by Lemma \ref{Norm attainment of V cross} we have that $M_{V^\times}=S_{\mathbb{W}}.$

\medskip

\noindent (ii)$\implies$(iii). Let $x$ and $y$ be any two elements of $\X$. Since $f_{Vx}$ and $f_{Vy}$ are members of $\J_\X\left(V\left(\X\right)\right)$, by Lemma \ref{Norm attainment of V cross} we have that
$
   \|V^\times(f_{Vx})\|=\|f_{Vx}\|$ and $\|V^\times(f_{Vy})\|=\|f_{Vy}\|$. Therefore, 
\[
   \left\{\frac{f_{Vx}}{\|f_{Vx}\|}, \frac{f_{Vx}}{\|f_{Vx}\|}\right\}\subseteq M_{V^\times}\subseteq \mathbb{W}.
\]
Since $\mathbb{W}$ is a linear subspace, we have that $(f_{Vx}+f_{Vy})\in\mathbb{W}.$ Therefore, by the hypothesis 
\[
  \|V^\times\left(f_{Vx}+f_{Vy}\right)\|=\|\left(f_{Vx}+f_{Vy}\right)\|.
\] 
It now follows from (\ref{V cross f_Vx equals f_x}) that $\|f_x+f_y\|=\|f_{Vx}+f_{Vy}\|$.

\medskip

\noindent (iii)$\implies$(iv). Let $f_x,f_y$ be any two elements of $\X^*$ and let $\alpha,\beta \in\mathbb{C}$. Since $\Phi$ is surjective, there is $z\in\X$ such that
$
   \Phi(f_{Vz}) = \alpha f_x + \beta f_y.
$
It follows from (\ref{Definition of Phi}) that
$
V^\times (f_{Vz})=V^\times(\alpha f_{Vx}+\beta f_{Vy}).
$
Therefore, there is $g\in\ker(V^\times)$ such that
$
   f_{Vz}=\alpha f_{Vx}+\beta f_{Vy}+g.
$ 
Since $M_{f_{Vz}}=\left\{\mu\frac{Vz}{\|Vz\|}:~|\mu|=1\right\}$ (by Lemma \ref{Lemma: Smoothness}) and $g$ vanishes on $V(\X)$, we have
\[
   (\alpha f_{Vx}+\beta f_{Vy}+g)|_{V(\X)}=(\alpha f_{Vx}+\beta f_{Vy})|_{V(\X)},
\]
and 
\[
   \left\|(\alpha f_{Vx}+\beta f_{Vy}+g)\right\|=\left\|(\alpha f_{Vx}+\beta f_{Vy}+g)|_{V(\X)}\right\|=\left\|(\alpha f_{Vx}+\beta f_{Vy})|_{V(\X)}\right\|.
\]
Therefore, $(\alpha f_{Vx}+\beta f_{Vy}+g)$ is a Hahn-Banach extension of $(\alpha f_{Vx}+\beta f_{Vy})|_{V(\X)}.$ Thus, we have 
\begin{align*}
\left\|\alpha f_{Vx}+\beta f_{Vy}\right\|
=\left\|f_{\overline{\alpha} V(x)}+f_{\overline{\beta} V(y)}\right\| = \left\|f_{V\left(\overline{\alpha} x\right)}+f_{V\left(\overline{\beta} y\right)}\right\|
& =\left\|f_{\overline{\alpha} x}+f_{\overline{\beta} y}\right\|\\
& =\|\alpha f_x+\beta f_y\|\\
& =\left\|\alpha V^\times(f_{Vx})+\beta V^\times(f_{Vy})\right\|\\
& =\|V^\times(\alpha f_{Vx}+\beta f_{Vy})\|,
\end{align*} 
where the third equality follows from the hypothesis.
Thus, by Lemma \ref{Norm attainment of V cross}, we have $(\alpha f_{Vx}+\beta f_{Vy})\in\J_\X\left(V\left(\X\right)\right).$ Suppose $(\alpha f_{Vx}+ \beta f_{Vy})$ attains its norm at $Vz_0$ for some $z_0\in S_{\X}$. Then
\[
   \left\|(\alpha f_{Vx}+\beta f_{Vy})|_{V(\X)}\right\|  \leq \left\|\alpha f_{Vx}+ \beta f_{Vy}\right\| =\left|(\alpha f_{Vx}+ \beta f_{Vy})(Vz_0)\right| \leq \left\|(\alpha f_{Vx}+\beta f_{Vy})|_{V(\X)}\right\|.
\]
Therefore, we have
$
    \left\|(\alpha f_{Vx}+\beta f_{Vy})|_{V(\X)}\right\|=\left\|(\alpha f_{Vx}+\beta f_{Vy})\right\|.
$
This shows that $(\alpha f_{Vx}+\beta f_{Vy})$ is also a Hahn-Banach extension of $(\alpha f_{Vx}+\beta f_{Vy})|_{V(\X)}.$ Therefore, by the uniqueness of Hahn-Banach extension, we have
$
    (\alpha f_{Vx}+\beta f_{Vy}+g)=(\alpha f_{Vx}+\beta f_{Vy}).
$
Consequently,
\[
    \Phi^{-1}(\alpha f_x+\beta f_y)=\alpha f_{Vx}+\beta f_{Vy}+g=\alpha f_{Vx}+\beta f_{Vy}=\alpha\Phi^{-1}(f_x)+\beta\Phi^{-1}(f_y),
\]
and $\Phi^{-1}$ is a linear map.

\medskip

\noindent (iv)$\implies$(i). The hypothesis indicates that $\J_\X(V(\X))$ is a linear subspace of $\X^*.$ Therefore, by \cite[Theorem 1]{Calvert} $V(\X)$ is right-complemented in $\X.$ This completes the proof.
\end{proof}

Combining Theorems \ref{thm:Wold Isometry-001} \& \ref{Isometry Range}, we have the following sufficient conditions such that a Banach space isometry becomes a Wold isometry.

\begin{thm}
Let $\mathbb{X}$ be a reflexive, smooth and strictly convex Banach space. Then each of the conditions below implies that $V$ is a Wold isometry.
\begin{itemize}
\item[(i)] $\|f_{Vx}+f_{Vy}\|=\|f_x+f_y\|\quad\mbox{ for all }x,y\in\mathbb{X}.$

\smallskip

\item[(ii)] The map $\Phi^{-1}:\mathbb{X}^*\to\J_\X\left(V\left(\mathbb{X}\right)\right)$ defined by $f_x\mapsto f_{Vx}$ is linear.

\smallskip

\item[(iii)] $M_{V^\times}=S_{\mathbb{W}}$ for some subspace $\mathbb{W}$ of $\mathbb{X}^*.$ 
\end{itemize}
\end{thm}


\section{The \boldmath\texorpdfstring{$\sigma$-}  ~shift and extension of a Wold isometry} \label{sec:05}

\vspace{0.2cm}

\noindent
Recall from Section \ref{sec:02} that Wold decomposition of an isometry $V$ on a Hilbert space $\HS$ splits $\HS$ into two orthogonal parts $\HS=\HS_0 \oplus_2 \HS_1$ such that $\HS_0, \HS_1$ are reducing subspaces of $V$ and $V|_{\HS_0}$ is a unitary and $V|_{\HS_1}$ is a pure isometry, i.e. a unilateral shift. On the other hand, not all isometries on a Banach space possess a Wold decomposition. The Banach space isometries that admit Wold decomposition are called Wold isometries (see Section \ref{sec:02} for details). In \cite{Faulkner Huneycutt}, Faulkner and Huneycutt found a Wold-type decomposition for a Wold isometry which goes parallel with the Hilbert space Wold decomposition. In Hilbert space setting, the unilateral shift part $V|_{\HS_1}$ of an isometry $V$ can be extended to a bilateral shift say $B$ (which is a unitary) in a natural way. It follows that every Hilbert space isometry $V$ extends to a unitary $V|_{\HS_0} \oplus B$. However, a bilateral shift on a Banach space is not always a unitary, in fact it may not be even a bounded operator, e.g. see \cite{CFS, Gellar}. In this Section, we define $\sigma$-shift between Banach spaces $\X, \, \Y$ which is analogous to Hilbert space bilateral shift in the sense that it is a unitary and it coincides with bilateral shift when $\X$ is a Hilbert space. Then we show that every unilateral shift on a Banach space extends to a $\sigma$-shift. It does not follow from here that every Wold isometry extends to a Banach sapce unitary and the obvious reason is that the pure isometry part of a Wold isometry may not be a unilateral shift. So, in the later half of this Section we also prove that every Wold isometry on a Banach space extends to a Banach space unitary.

\smallskip

Let $V$ be a Wold isometry on $\mathbb{X}$ and $\mathbb{X}=V\left(\mathbb{X}\right)\bigoplus \mathcal{L}.$ Let $\mathbb{X}_1=\bigcap_{n=0}^\infty V^n\mathbb{X}$ and $\mathbb{X}_2=\bigoplus_{n=0}^\infty V^n\mathcal{L}.$ We denote by $\mathbb N_0$ the set of all non-negative integers, i.e. $\mathbb N \cup \{0\}$. The vector space of sequences
\begin{equation}\label{Sequence space}
      \mathcal{S}:=\left\{(l_n)_{n\in \mathbb{N}_0}:~l_n\in \mathcal{L} \, \, \, \text{ for all } \, \, n \in \mathbb N_0 \, \, \, \text{ and } \, \, \,  ~\sum_{n=0}^\infty V^nl_n\in\mathbb{X}_2\right\}
\end{equation}
that is equipped with the norm
\begin{equation}\label{Norm of sequence space}       
\left\|(l_n)_{n\in \mathbb{N}_0}\right\|_\mathcal{L}:=\left\|\sum_{n=0}^\infty V^nl_n\right\|,
\end{equation}
is a Banach space and it is isometrically isomorphic to $\X_2$. Therefore, for a unilateral shift $V$ on $\X$ we have the following identification. 

\begin{prop}\label{Unilateral sequence space}
Let $\mathbb{X}$ be a Banach space and $V$ be a unilateral shift on $\mathbb{X}$. Then $\mathbb{X}$ is isometrically isomorphic to $(\mathcal{S},\|\cdot\|_\mathcal{L})$.
\end{prop}

\begin{proof}
Evidently, the map $\Lambda: \X \to \mathcal{S}$ defined by 
\begin{equation}\label{Definition of lambda}
 \Lambda\left(\sum_{n=0}^\infty V^nl_n\right) =(l_n)_{n \in \mathbb{N}_0} 
\end{equation}
is a linear surjective isometry.
\end{proof}

We shall be consistent with the notations of Theorem \ref{Left Inverse of Wold Isometry} and Proposition \ref{Unilateral sequence space} in the following remark.

\begin{rem}
A unilateral shift $V$ on a Banach space $\X$ with the generating subspace $\mathcal{L}$ can be realized as a forward shift operator on $(\mathcal{S},\|\cdot\|_\mathcal{L})$. The isometry $\widehat{V}: (\mathcal{S},\|\cdot\|_\mathcal{L})\to (\mathcal{S},\|\cdot\|_\mathcal{L})$, defined by
\[
\widehat{V}(l_0,l_1,\dots) = \Lambda V\Lambda^{-1}(l_0,l_1, \dots) = (\mathbf{0}, l_0,l_1, \dots)
\]
is a forward shift operator. Thus, $\widehat{V}$ is a unilateral shift with generating subspace $\widehat{\mathcal{L}}$, where
\[
\widehat{\mathcal{L}} = \{(l,{\mathbf{0}},{\mathbf{0}}, \dots)\in \mathcal{S}:~l\in \mathcal{L}\}.
\]
This is because, for any $n\in \mathbb{N}$ and $l_1, l_2\in \mathcal{L}$ we have
\[
 \left\|\widehat{V}^n(l_1,\mathbf{0},\dots)+(l_2,{\mathbf{0}}, \dots)\right\|_\mathcal{L} = \| V^n l_1 + l_2\| \geq  \| V^n l_1\| = \| \widehat{V}^n (l_1, \mathbf{0}, \dotsc )\|_\mathcal{L}.
 \]
 Moreover, if the range of $V$ is right-complemented, then the left-inverse of $V$ as in Theorem \ref{Left Inverse of Wold Isometry} can be realized as a backward shift operator on $(\mathcal{S}, \|\cdot\|_\mathcal{L})$. The following commutative diagram is a gist of this entire discussion.
\begin{center}
\begin{equation}\label{Isometric copy of unilateral shift} 
\begin{gathered}
\xymatrix
    {
       \X \ar[rr]^V \ar[d]^\Lambda && \X \ar[d]_\Lambda \\
       \mathcal{S} \ar[rr]^{\widehat{V}} && \mathcal{S}
    }\\
\end{gathered}
\end{equation}
{\small{Isometric copy of unilateral shift}}
\end{center}
\end{rem} \qed

\bigskip

Below we present an example of a unilateral shift which will be used in sequel.

\begin{eg}
Recall that for a Banach space $\X$, the space $\ell_2(\X)$ consists of all $\X$-valued square summable sequences, i.e.
\[
\ell_2(\X) = \left\{(x_0,x_1,\dots):~\sum_{n=0}^\infty \|x_n\|^2 < \infty\right\}
\]
and is equipped with the norm
\[
 \|(x_0,x_1,\dots)\| = \left(\sum_{n=0}^\infty \|x_n\|^2\right)^\frac{1}{2}.
\]
The operator $M_z:\ell_2(\X)\to \ell_2(\X)$ defined by 
\begin{equation}\label{Forward shift}
 M_z(x_0,x_1, \dots) = (\mathbf{0},x_0,x_1,\dots),
\end{equation}
is a typical example of a unilateral shift with the generating subspace
$
 \{(x,\mathbf{0},\mathbf{0},\dots):~x\in \X\}.
$
We denote the left-inverse of $M_z$ by $\widehat{M}_{z}$ which is the backward shift on $\ell_2(\X)$, i.e.,
\begin{equation}\label{backward shift}
\widehat{M}_{z}(x_0,x_1, \dots) = (x_1,x_2, \dots).
\end{equation}
In particular when $\X$ is a Hilbert space, the backward shift $\widehat{M}_{z}$ coincides with $M_z^*$. \qed
\end{eg}

We now define $\sigma$-shift between Banach spaces and this is an analogue of bilateral shift on a Hilbert space.

\begin{defn}\label{Sigma Shift}
Let $\mathbb{Y}$ be a vector space and let $\|\cdot\|_\mathbb{Y}$, $\|\cdot\|_\sigma$ be two norms in $\mathbb{Y}$ such that both $\left(\mathbb{Y},\|\cdot\|_\sigma\right)$ and $\left(\mathbb{Y},\|\cdot\|_\mathbb{Y}\right)$ are Banach spaces. A surjective isometry $\widetilde{V}:\left(\mathbb{Y},\|\cdot\|_\sigma\right)\to\left(\mathbb{Y},\|\cdot\|_\mathbb{Y}\right)$ is said to be a \textit{$\sigma$-shift} if it satisfies the following properties.

\smallskip

\begin{itemize}

\item[(i)] There exists a closed linear subspace $\mathcal{I}\subseteq \left(\mathbb{Y},\|\cdot\|_\mathbb{Y}\right) \cap \left(\mathbb{Y},\|\cdot\|_\sigma\right)$ such that $\widetilde{V}^n\mathcal{I}\perp_B\mathcal{I}$ for $n=1, 2, \dots $ and $\mathcal{I}\perp_B \widetilde{V}^n\mathcal{I}$ for $n=-1, -2, \dots$ .

\item[(ii)] $\left(\mathbb{Y},\|\cdot\|_\mathbb{Y}\right)=\left(\mathbb{W},\|\cdot\|_\mathbb{W}\right),$
where $\mathbb{W}=\displaystyle\left\{w_1+w_2~:~w_1\in\bigoplus_{n= 0}^{\infty}\widetilde{V}^n\mathcal{I},~ w_2\in \bigoplus_{n = - \infty}^{-1}\widetilde{V}^n\mathcal{I}\right\}$
and $\|w_1+w_2\|_{\mathbb{W}}=\left(\|w_1\|_\mathbb{Y}^2+\|w_2\|_\mathbb{Y}^2\right)^{\frac{1}{2}}$.

\end{itemize}

The subspace $\mathcal{I}$ is called a \textit{generating subspace} and $dim(\mathcal{I})$ is called the \textit{multiplicity} of the $\sigma$-shift operator $\widetilde{V}.$
\end{defn}
Evidently, it trivially follows that a $\sigma$-shift is a bilateral shift between Banach spaces in the sense of the definition of ``Bilateral shift between Banach spaces" given in Section \ref{sec:02}. Below we give an example of a $\sigma$-shift for the convenience of the readers.

\begin{eg} \label{Example:new-051}
Let $\X$ be a Banach space and $\ell_2(\mathbb{Z}, \X)$ be the Banach space consists of all square summable bi-sequences with entries in $\X$, i.e.,
\[
  \ell_2(\mathbb{Z}, \X) =\{(x_n)_{n\in \mathbb{Z}}: x_n \in \X \text{ for all } n\in \mathbb{Z}, ~ \sum_{n=-\infty}^\infty \|x_n\|^2 < \infty\},
\]
and is equipped with the norm 
\[
  \|(x_n)_{n\in \mathbb{Z}}\| = \left( \sum_{n=-\infty}^\infty \|x_n\|^2 \right)^{\frac{1}{2}}.
\]
Consider the operator $U:\ell_2(\mathbb{Z}, \mathbb{X}) \to \ell_2(\mathbb{Z},\X)$ defined by
\[
 U(\dots, x_{-2}, x_{-1}, \boxed{x_0}, x_1, x_2, \dots)= (\dots, x_{-2}, \boxed{x_{-1}}, x_0, x_1, x_2, \dots), \quad (x_n)_{n\in\mathbb{Z}}\in \ell_2(\mathbb{Z},\mathbb{X}),
\]
where the box denotes the $0$-th position in the sequence. Then $U$ is a surjective isometry with inverse $U^{-1}: \ell_2(\mathbb{Z}, \X) \to \ell_2(\mathbb{Z}, \X)$ given by
\[
 U^{-1}(\dots, x_{-2}, x_{-1}, \boxed{x_0}, x_1, x_2, \dots)= (\dots, x_{-2}, x_{-1}, x_0, \boxed{x_1}, x_2, \dots), \quad (x_n)_{n\in\mathbb{Z}}\in \ell_2(\mathbb{Z},\mathbb{X}).
\]
Let $\mathcal{I} =\{(\dotsc, \mathbf{0}, \boxed{x}, \mathbf{0}, \dotsc): x\in \X\}$. Then $\mathcal{I}$ is a closed linear subspace of $\ell_2(\mathbb{Z},\X)$ and for arbitrary $\mathbf{x}_1$, $\mathbf{x}_2$ in $\mathcal{I}$, we have
\[
 \|U^n \mathbf{x}_1 + \mathbf{x}_2 \| \geq \|U^n \mathbf{x}_1\| \quad \text{and} \quad \| U^n\mathbf{x}_1 + \mathbf{x}_2 \| \geq \| \mathbf{x}_2 \|, \ \
\text{ for all } n\in \mathbb{Z}\setminus \{0\}.
\]
This shows that $U^n \mathcal{I} \perp_B \mathcal{I}$ for all $n\in \mathbb{N}$ and $\mathcal{I} \perp_B U^n \mathcal{I}$ for all negative integers. We also have that $\ell_2(\mathbb{Z}, \X) = \mathbb{W}_1 \oplus_2 \mathbb{W}_2$, where $\mathbb{W}_1= \displaystyle \bigoplus_{n=-\infty}^{-1} U^n \mathcal{I}$ and $\mathbb{W}_2= \displaystyle \bigoplus_{n=0}^\infty U^n \mathcal{I}$. Therefore, $U$ is a $\sigma$-shift operator with generating subspace $\mathcal{I}$. 
\end{eg} \qed

\begin{prop}\label{prop:011}
Let $\widetilde{V}: (\Y,\|\cdot\|_\sigma) \to (\Y,\|\cdot\|_\Y)$ be a $\sigma$-shift operator with generating subspace $\mathcal{I}$. If $(\Y,\|\cdot\|_\sigma)= (\Y,\|\cdot\|_\Y)$ and if it is a Hilbert space, then $\widetilde{V}$ is a bilateral shift with generating subspace $\mathcal{I}$.
\end{prop}

\begin{proof}
Birkhoff-James orthogonality is equivalent to the inner product orthogonality on a Hilbert space. So, Birkhoff-James orthogonality is symmetric and both left-right additive on a Hilbert space. Thus, it follows from property-(i) in Definition \ref{Sigma Shift} that $\widetilde{V}^n \mathcal{I} \perp \mathcal{I}$ for all $n\in \mathbb{Z}\setminus \{0\}$. Consequently, $\mathcal{I}$ is a wandering subspace of $\widetilde{V}$. Property-(ii) in Definition \ref{Sigma Shift} shows that 
\[
(\Y,\|\cdot\|_\Y) = \left(\bigoplus_{n=-\infty}^{-1} \widetilde{V}^n \mathcal{I} \right) \oplus_2 \left(\bigoplus_{n=0}^\infty \widetilde{V}^n \mathcal{I}\right) = \bigoplus_{n=-\infty}^\infty \widetilde{V}^n \mathcal{I}.
\]
Hence $\widetilde{V}$ is a bilateral shift with generating subspace $\mathcal{I}$.
\end{proof}

\begin{thm}\label{Sigma shift extension}
Every unilateral shift on a Banach space can be extended to a \textit{$\sigma$-shift} of same multiplicity. 
\end{thm}

\begin{proof} 
Let $V$ be a unilateral shift on a Banach space $\left(\X, \|\cdot\|_\X\right)$ with the generating subspace $\mathcal{L}.$ It follows from Proposition \ref{Unilateral sequence space} that $\Lambda:(\X, \|\cdot\|_\X)\to(\mathcal{S},\|\cdot\|_\mathcal{L})$ is a surjective isometry (see (\ref{Sequence space}) and (\ref{Norm of sequence space}) for the definition of $(\mathcal{S},\|\cdot\|_\mathcal{L})$). Consider the vector space $\mathbb{Y}$ of all bi-sequences
\[
\mathbb{Y}:=\left\{(l_n)_{n\in \mathbb{Z}}:~l_n\in \mathcal{L} \, \, \, \text{ for all } \, \, \, n \in \mathbb Z \, \, \, \text{ and }\, \, \, ~\sum_{n=0}^\infty V^nl_n\in\mathbb{X},~\sum_{n=0}^\infty V^{|n|}l_n\in \mathbb{X}\right\},
\]
and the function $\|\cdot\|_\mathbb{Y}:\mathbb{Y}\to [0,\infty)$ defined by
\[
\left\|(l_n)_{n\in \mathbb{Z}}\right\|_\mathbb{Y} = \left(\left\|\sum_{n= -\infty}^{-1}V^{|n|}l_n\right\|_\mathbb{X}^2 + \left\|\sum_{n=0}^\infty V^nl_n\right\|_\mathbb{X}^2\right)^\frac{1}{2}.
\]
It is easy to see that $\|\cdot\|_\mathbb{Y}$ is a norm on $\mathbb{Y}$ and $(\Y, \|\cdot\|_\Y)$ is isometrically isomorphic to $(V\X\bigoplus_2 \X)$. Therefore, $\left(\mathbb{Y},\|\cdot\|_{\mathbb{Y}}\right)$ is a Banach space.

\medskip

Let $\sigma:\mathbb{Z}\to \mathbb{Z}$ be such that $\sigma(n)=n-1$. Define another function $\|\cdot\|_\sigma:\mathbb{Y}\to [0,\infty)$ by
\[
\left\|(l_n)_{n\in \mathbb{Z}}\right\|_\sigma = \left\|(l_{\sigma(n)})_{n\in \mathbb{Z}}\right\|_\mathbb{Y}.
\]
It is easy to show that the function $\|\cdot\|_\sigma$ satisfies the positivity and homogeneity. Also, given any $\mathbf{l}=(l_n)_{n\in \mathbb{Z}}$ and $\mathbf{r}=(r_n)_{n\in \mathbb{Z}}$ in $\mathbb{Y}$, we have

\begin{align*}
\|(\mathbf{l}+\mathbf{r})\|_\sigma & = \left(\left\|\sum_{n=-\infty}^{-1} V^{|n|}(l_{n-1}+r_{n-1})\right\|_\mathbb{X}^2 + \left\|\sum_{n=0}^\infty V^{n}(l_{n-1}+r_{n-1})\right\|_\mathbb{X}^2\right)^\frac{1}{2}\\
& = \left(\left\|\sum_{n=-\infty}^{-1} V^{|n|}l_{n-1}+\sum_{n=-\infty}^{-1} V^{|n|}r_{n-1}\right\|_\mathbb{X}^2 + \left\|\sum_{n=0}^\infty V^{n}l_{n-1}+\sum_{n=0}^\infty V^{n}r_{n-1}\right\|_\mathbb{X}^2\right)^\frac{1}{2}\\
& \leq \left[\left(\left\|\sum_{n=-\infty}^{-1} V^{|n|}l_{n-1}\right\|_\mathbb{X}+\left\|\sum_{n=-\infty}^{-1} V^{|n|}r_{n-1}\right\|_\mathbb{X}\right)^2 + \left(\left\|\sum_{n=0}^\infty V^{n}l_{n-1}\right\|_\mathbb{X}+\left\|\sum_{n=0}^\infty V^{n}r_{n-1}\right\|_\mathbb{X}\right)^2 \right]^\frac{1}{2}\\
& \leq \left(\left\|\sum_{n=-\infty}^{-1} V^{|n|}l_{n-1}\right\|_\mathbb{X}^2 + \left\|\sum_{n= 0}^\infty V^{n}l_{n-1}\right\|_\mathbb{X}^2\right)^\frac{1}{2} + \left(\left\|\sum_{n= - \infty}^{-1} V^{|n|}r_{n-1}\right\|_\mathbb{X}^2 + \left\|\sum_{n=0}^\infty V^{n}r_{n-1}\right\|_\mathbb{X}^2\right)^\frac{1}{2}\\
& = \|\mathbf{l}\|_\sigma + \|\mathbf{r}\|_\sigma. 
\end{align*}
Therefore, $\|\cdot\|_\sigma$ is a norm on $\mathbb{Y}$. It is straightforward to see that the map $\widetilde{V}:(\mathbb{Y},\|\cdot\|_\sigma)\to (\mathbb{Y},\|\cdot\|_\mathbb{Y})$, defined by
\[
\widetilde{V}\left((l_n)_{n\in\mathbb{Z}}\right) = (l_{n-1})_{n\in\mathbb{Z}}.
\]
is a surjective isometry. Thus, $(\mathbb{Y},\|\cdot\|_\sigma)$ is a Banach space. Now consider the vector subspace 
\[
\mathcal{I}=\{(\cdots,\mathbf{0},\mathbf{0},\boxed{l},\mathbf{0},\mathbf{0},\cdots):l\in\mathcal{L}\}
\]
of $\mathbb{Y}.$ To prove $\widetilde{V}:(\mathbb{Y},\|\cdot\|_\sigma)\to (\mathbb{Y},\|\cdot\|_\mathbb{Y})$ is a $\sigma$-shift, we show that $\mathcal{I}$ is a generating subspace for $\widetilde{V}$ in the following three steps.

\medskip

Step I: We show that $\mathcal{I}$ is a closed subspace of both $\left(\mathbb{Y},\|\cdot\|_\mathbb{Y}\right)$ and $\left(\mathbb{Y},\|\cdot\|_\sigma\right).$ First observe that $(\mathcal{S},\|\cdot\|_\mathcal{L})$ can be linearly embedded in both of $\left(\mathbb{Y},\|\cdot\|_\mathbb{Y}\right)$ and $\left(\mathbb{Y},\|\cdot\|_\sigma\right)$ by the isometry
\[
 \widetilde{\Lambda}\left(\left(l_n\right)_{n \in \mathbb{N}_0}\right)=(l_n')_{n\in \mathbb{Z}},\quad \text{where} \quad 
 l_n'=\begin{cases}
     l_n, & \mathrm{if}~n\geq 0,\\
     \mathbf{0}, & \mathrm{if}~n\leq -1.
     \end{cases}
\]
Also, note that $(\mathbb{X},\|\cdot\|_\mathbb{X})$ is isometrically embedded in both $(\mathbb{Y},\|\cdot\|_\sigma)$ and $(\mathbb{Y},\|\cdot\|_\mathbb{Y})$ via the isometry $\widetilde{\Lambda}\circ\Lambda$. Consequently, $\mathcal{I}$ is a closed subspace of both $(\mathbb{Y},\|\cdot\|_\sigma)$ and $(\mathbb{Y},\|\cdot\|_\mathbb{Y}),$ since  $\mathcal{L}$ is a closed subspace of $\left(\mathbb{X},\|\cdot\|_{\mathbb{X}}\right).$

\medskip

Step II: Let $\mathbf{l}_i=(\cdots,\mathbf{0},\mathbf{0},\boxed{l_i},\mathbf{0},\mathbf{0},\cdots)$ be any two members of $\mathcal{I}$ for some $l_i\in\mathcal{L},~i=1,2.$ Then for any scalar $\lambda$ we have
\begin{align*}
    \|\widetilde{V}^n\mathbf{l}_1+\lambda \mathbf{l}_2\|_\mathbb{Y} & = \| V^nl_1+\lambda l_2\|_\mathbb{X} \geq \|V^nl_1\|_\mathbb{X} = \|\widetilde{V}^n\mathbf{l}_1\|_\mathbb{Y},\qquad & \text{if}~ n\geq 1,\\
    \|\mathbf{l}_2+\lambda\widetilde{V}^n\mathbf{l}_1\|_\sigma & = \| Vl_2+\lambda l_1\|_\mathbb{X} \geq \|Vl_2\|_\mathbb{X} = \|\mathbf{l}_2\|_\sigma, \qquad & \text{for}~ n=-1, \\
    \|\mathbf{l}_2+\lambda\widetilde{V}^n\mathbf{l}_1\|_\sigma & = \left(\|\lambda V^{-n+1}l_1\|_\mathbb{X}^2+\| Vl_2\|_\mathbb{X}^2\right)^{\frac{1}{2}} \geq \|Vl_2\|_\mathbb{X} = \|\mathbf{l}_2\|_\sigma, \qquad & \text{if}~ n\leq -2,
\end{align*}
which is the property-(ii) in Definition \ref{Sigma Shift}.

\bigskip

Step III: We show that
$\left(\mathbb{Y},\|\cdot\|_\mathbb{Y}\right)=\left(\mathbb{W},\|\cdot\|_\mathbb{W}\right),$
where 
\[
\mathbb{W}=\displaystyle\left\{w_1+w_2~:~w_1\in\bigoplus_{n= 0}^\infty \widetilde{V}^n\mathcal{I},~ w_2\in \bigoplus_{n=-\infty}^{-1} \widetilde{V}^n\mathcal{I}\right\} \quad \mathrm{and} \quad \|w_1+w_2\|_{\mathbb{W}}=\left(\|w_1\|_\mathbb{Y}^2+\|w_2\|_\mathbb{Y}^2\right)^{\frac{1}{2}}.
\]
Let $(w_1+w_2)\in (\mathbb{W},\|\cdot\|_\mathbb{W})$ be arbitrary. Then we have the following:
\begin{align*}
& \text{(a)}~ w_1 =  \lim_n\sum_{k=0}^{n}\widetilde{V}^k \mathbf{l}_k = (\dots, \mathbf{0},\boxed{l_0},l_1,\dots,l_n,l_{n+1},\dots ),\\
& \text{(b)}~ w_2 =  \lim_n\sum_{k=-n}^{-1}\widetilde{V}^k \mathbf{l}_k = (\dots,l_{-n},l_{-n+1},\dots,l_{-1},\boxed{\mathbf{0}},\dots),
\end{align*}
where 
$
\mathbf{l}_{k} = (\dots, \mathbf{0}, \mathbf{0}, \boxed{l_{k}}, \mathbf{0}, \mathbf{0} \dots)$ with $k\in \mathbb{Z}$ and the limits in (a) and (b) are taken with respect to $\|\cdot\|_\Y$ and $ \|\cdot\|_\sigma$ on $\mathbb{Y}$ respectively. For $1\leq m < n,$ we have
\begin{align*}
& \left\|\sum_{k=-n}^{-(m+1)}V^{|k|} l_k\right\|_\mathbb{X}^2 + \left\|\sum_{k=m+1}^n V^kl_k\right\|_\mathbb{X}^2 \\
= & \left\|\sum_{k=-n}^{-(m+1)}V^{|k|-1} l_k\right\|_\mathbb{X}^2 + \left\|\sum_{k=m+1}^n V^kl_k\right\|_\mathbb{X}^2 \\
= & \| (\dots,\mathbf{0},\underbrace{~l_{-n}~}_{(-n+1)-th},{l_{-n+1}},\dots,\underbrace{l_{-m-1}}_{(-m)-th},\mathbf{0},\dots)\|_{_{\mathbb{Y}}}^2 + (\dots,\mathbf{0},l_{m+1},l_{m+2},\dots,l_n,\mathbf{0},\dots)\|_{_{\mathbb{Y}}}^2 \\
= & \|(\dots,\mathbf{0},\underbrace{l_{-n}}_{(-n)-th},{l_{-n+1}},\dots, \underbrace{l_{-m-1}}_{-(m+1)-th},\mathbf{0},\dots)\|_{_{\sigma}} ^2 + \|(\dots,\mathbf{0},l_{m+1},l_{m+2},\dots,l_n,\mathbf{0},\dots)\|_{_{\mathbb{Y}}}^2 \\
= & \left\|\sum_{k=-n}^{-(m+1)}\widetilde{V}^k \mathbf{l}_k\right\|_\sigma ^2 + \left\|\sum_{k=m+1}^{n}\widetilde{V}^k \mathbf{l}_k \right\|_{\mathbb{Y}}^2 \quad \longrightarrow 0 \quad \text{as} \, \,  ~n,m \to\infty.
\end{align*}

Thus, both the series $\sum_{n=-1}^\infty V^{|n|}l_n$ and $\sum_{n=0}^\infty V^n l_n$ converge in $\mathbb{X}.$ Therefore, the bi-sequence $(w_1+w_2)=(\dots, l_{-n},l_{-n+1},\dots,\boxed{l_0},l_1,\dots,l_n,\dots)$ is a member of the vector space $\Y$. Moreover,
\[
  \|(w_1+w_2)\|_\mathbb{W} = \|(\dots, l_{-1},\boxed{l_0},l_1,\dots)\|_\mathbb{Y},
\]
which shows that $\left(\mathbb{W},\|\cdot\|_\mathbb{W}\right)$ is contained in $\left(\mathbb{Y},\|\cdot\|_\mathbb{Y}\right)$ as a  normed subspace.

\smallskip

To prove the other direction, we first observe that
\begin{equation}\label{Relation between sigma and norm}
\|\mathbf{p}\|_{_\mathbb{Y}} \leq \sqrt{2}\|\mathbf{p}\|_\sigma, \quad \text{ for }\quad \mathbf{p}\in \bigoplus_{n=-\infty}^{-1}\widetilde{V}^{n}\mathcal{I}.   
\end{equation}
Let $\mathbf{p}=\sum_{n=-\infty}^{-1}\widetilde{V}^n\mathbf{l}_n$ with $\mathbf{l}_n=(\cdots, \mathbf{0}, \boxed{l_n},\mathbf{0}, \cdots),$ where $l_{-n} \in \mathcal{L}$ for all $n \in \mathbb N$. Then
\begin{align*}
  \|\mathbf{p}\|_{_{\mathbb{Y}}} = \|(\cdots, l_{-2}, l_{-1},\boxed{\mathbf{0}}, \cdots)\|_{_{\mathbb{Y}}} & = \left\|\sum_{n=1}^\infty V^n l_{-n} \right\|_\mathbb{X} \\
    & \leq \|Vl_{-1}\|_{\X} + \left\|\sum_{n=2}^\infty V^n l_{-n} \right\|_\mathbb{X} \\
    &= \|l_{-1}\|_{\X} + \left\|\sum_{n=1}^\infty V^{n}l_{-n-1} \right\|_\mathbb{X} \qquad \left[~\mbox{since $V$ is isometry}~\right] \\
    &= \sqrt{2} \left(\|l_{-1}\|_{\X}^2 + \left\|\sum_{n=1}^\infty V^{n}l_{-n-1} \right\|_\mathbb{X}^2 \right)^{\frac{1}{2}} \\
    & =\sqrt{2} \|\mathbf{p}\|_\sigma .
\end{align*}

Now, consider any $\mathbf{l}=(l_n)_{n\in\mathbb{Z}}\in\left(\mathbb{Y},\|\cdot\|_{\mathbb{Y}}\right)$. Then both the series $\sum_{n=-1}^\infty V^{|n|}l_n$ and $\sum_{n=0}^\infty V^nl_n$ are convergent in $\mathbb{X}.$ We show that $\mathbf{l}\in \mathbb{W},$ by proving $\mathbf{l}= \mathbf{p}+\mathbf{q}$, where
\[
\mathbf{p}=(\dots,l_{-2},l_{-1},\boxed{\mathbf{0}},\mathbf{0},\dots)\in \bigoplus_{n=-\infty}^{-1} \widetilde{V}^n \mathcal{I}, \quad \text{and}\quad \mathbf{q}=(\dots,\mathbf{0},\boxed{l_0},l_1,\dots)\in \bigoplus_{n=0}^\infty \widetilde{V}^n \mathcal{I}.
\]

\medskip

Denote $\mathbf{l}_i=(\cdots, \mathbf{0}, \boxed{l_i}, \mathbf{0}, \cdots ),~ i\geq 1$ and consider the sequences $(z_m)$ and $(w_m)$ defined by
\begin{align*}
& z_m = (\dots,\mathbf{0},l_{-m},l_{-m+1},\dots,l_{-1},\boxed{\mathbf{0}},\mathbf{0},\dots) = \sum_{i=-m}^{-1}\widetilde{V}^i\mathbf{l}_i ,\quad m\geq 1, \\
& w_m =(\dots,\mathbf{0},\boxed{l_{0}},l_{1},\dots,l_m,\mathbf{0},\mathbf{0},\dots)=\sum_{i=0}^{m}\widetilde{V}^i \mathbf{l}_i,\quad m\geq 1.
\end{align*}
Evidently,
\[
z_m\in\left(\bigoplus_{n=- \infty}^{-1} \widetilde{V}^n\mathcal{I},\|\cdot\|_{\sigma}\right), \quad \text{and} \quad w_m\in \left(\bigoplus_{n= 0}^\infty \widetilde{V}^n\mathcal{I},\|\cdot\|_{\mathbb{Y}}\right),\quad m\geq 1.
\]

\medskip

Now,
\begin{align*}
\|z_m - \mathbf{p} \|_\sigma = \|(\dots,l_{-m-2},\underbrace{l_{-(m+1)}}_{-(m+1)-th},\mathbf{0},\dots) \|_{_{\sigma}} = \left\| \sum_{i=-\infty}^{-m} V^{|i|} l_{i-1}\right\|_\mathbb{X} = & \left\| \sum_{i=-\infty}^{-m-1} V^{|i|} l_i\right\|_\mathbb{X} \\
& \longrightarrow 0 \, \, \, \text{as}~ m\to\infty,\\
 \left\|w_m-\mathbf{q}\right\|_{\mathbb{Y}} = \|(\dots,\mathbf{0},\underbrace{l_{m+1}}_{(m+1)-th},l_{m+2},\dots)\|_{_{\mathbb{Y}}}=\left\|\sum_{i=m+1}^\infty V^i l_i\right\|_\mathbb{X}\longrightarrow 0 & \, \, \, \text{ as } \, \, ~ m\to\infty.
\end{align*}
Therefore, using (\ref{Relation between sigma and norm}), we have
\[
\lim_{m \rightarrow \infty}z_m=\mathbf{p},\quad \lim_{m \rightarrow \infty}w_m=\mathbf{q},\quad \|\mathbf{p}\|_\mathbb{Y}\leq \sqrt{2}\|\mathbf{p}\|_\sigma<\infty,\quad \|\mathbf{q}\|_\mathbb{Y}<\infty .
\]
Consequently, $\mathbf{p}\in\left(\bigoplus_{n=-\infty}^{-1}\widetilde{V}^n\mathcal{I},\|\cdot\|_\mathbb{Y}\right)$ and $\mathbf{q}\in\left(\bigoplus_{n= 0}^\infty \widetilde{V}^n\mathcal{I},\|\cdot\|_\mathbb{Y}\right),$ as desired.

\smallskip

 Moreover,
$
  \|(\mathbf{p}+\mathbf{q})\|_\mathbb{W} =\|(l_n)_{n\in\mathbb{Z}}\|_\mathbb{Y},
$
which shows that $\left(\mathbb{Y},\|\cdot\|_\mathbb{Y}\right)$ is contained in $\left(\mathbb{W},\|\cdot\|_\mathbb{W}\right)$ as a  normed subspace.
Altogether, we conclude that $\widetilde{V}$ is a $\sigma$-shift with the generating subspace $\mathcal{I}$ and $\dim (\mathcal{I}) = \dim(\mathcal{L}).$

\smallskip

Finally, we show that $\widetilde{V}$ is an extension of $V.$ Recall that $(\X,\|\cdot\|_\X)$ is isometrically isomorphic to $(\mathcal{S},\|\cdot\|_\mathcal{L})$ via the map $\widetilde{\Lambda}\circ\Lambda$ and the unilateral shift $\widehat{V}$ on $(\mathcal{S},\|\cdot\|_\mathcal{L})$ is an isometric copy of $V$ (see the commutative diagram (\ref{Isometric copy of unilateral shift}) ). Therefore, the following diagram commutes: 

\begin{center}
\begin{equation*}
\xymatrix
{
\left(\mathbb{X},\|\cdot\|_{\mathbb{X}}\right) \ar[rrrrrr]^ V \ar[dd]_\Lambda &&&&&& \left(\mathbb{X},\|\cdot\|_{\mathbb{X}}\right) \ar[dd]^\Lambda \\\\
 {(\mathcal{S}, \|\cdot\|_\mathcal{L})}\ar[rrrrrr]^{\widehat{V}(l_0,l_1, \cdots)\mapsto(\cdots,\mathbf{0}, \boxed{\mathbf{0}},l_0,l_1,,\cdots)} \ar[dd]_{\widetilde{\Lambda}} &&&&&& {(\mathcal{S},\|\cdot\|_\mathcal{L})} \ar[dd]^{\widetilde{\Lambda}} \\\\
{\left(\mathbb{Y},\|\cdot\|_\sigma\right)}\ar[rrrrrr]^{\widetilde{V}\left((l_n)_{n \in \mathbb{Z}}\right)\mapsto (l_{n-1})_{n \in \mathbb{Z}}}  &&&&&& {\left(\mathbb{Y},\|\cdot\|_\mathbb{Y}\right)}
}
\end{equation*}
\end{center}
and consequently we have that $\widetilde{V}$ is an extension of $V$.
\end{proof}
Now the following well-known result of Hilbert space theory follows as a corollary of the previous theorem.

\begin{cor}
Let $V$ be a unilateral shift on a Banach space $\mathbb{\X}$. If $\X$ is a Hilbert space, then $V$ extends to a bilateral shift of same multiplicity .
\end{cor}

\begin{proof}
Since both $\X$ and $V\X$ are Hilbert spaces, $V\X \oplus_2 \X$ is also a Hilbert space. The spaces $(\Y,\|\cdot\|_\Y)$ and $(\Y,\|\cdot\|_\sigma)$ as constructed in Theorem \ref{Sigma shift extension} are identically equal to $\ell_2(\mathbb{Z},\mathcal{L})$, where $\mathcal{L}$ is the generating subspace of $V$. This is because the sequence space $(\mathcal{S},\|\cdot\|_\mathcal{L})$ (see (\ref{Sequence space}) and (\ref{Norm of sequence space})) is identically equal to $\ell_2(\mathcal{L})$. Therefore, it follows from Proposition \ref{prop:011} that the map $\widetilde{V}$ as constructed in Theorem \ref{Sigma shift extension} is a bilateral shift with generating subspace $\mathcal{L}$. This completes the proof.
\end{proof}

\subsection{Generalized extension}
\begin{thm} \label{thm:general-extension-01}
Every Wold isometry on a Banach space extends to a Banach space unitary.
\end{thm}

\begin{proof}
Suppose $V$ is a Wold isometry on a Banach space $\X$ such that $\X = \X_1 \oplus \X_2$ (see (\ref{Wold decomposition Banach spaces})), where
\[
\X_1 = \bigcap_{n=0}^\infty V^n (\X), \qquad \X_2 = \bigoplus_{n=0}^\infty V^n \mathcal{L}, ~~\text{ for some closed subspace } \mathcal{L}\subseteq \X
\]
and $V|_{\X_1}$ is a Banach space unitary and $V|_{\X_2}$ is a pure Banach space isometry.
Consider the vector space 
\[
  \Y = \left\{ \left( x_1, (l_n)_{n\in \Z}\right): x_1\in \X_1, ~l_n\in \mathcal{L}, ~ \forall n\in \Z, \text{ such that }\sum_{n=0}^\infty V^n l_n,~~\sum_{n= -\infty}^{-1} V^{\vert n \vert} l_n \text{ converges in } \X_2 \right\},
\]
and the function $\|\cdot\|_\Y: \Y \to [0, \infty)$ defined by
\[
   \left\|\left( x_1, (l_n)_{n\in \Z} \right) \right\|_\Y = \left( \left\| x_1 + \sum_{n=0}^\infty V^n l_n \right\|_\X^2 + \left\| \sum_{n=-\infty}^{-1} V^{\vert n \vert} l_n \right\|_\X^2 \right)^{\frac{1}{2}}.
\]
We prove that the function $\|\cdot\|_\Y: \Y \to [0, \infty)$ defines a norm on $\Y$. The absolute homogeneity is trivial.

\medskip

$\bullet$ Positivity: For any $\left( x_1, (l_n)_{n\in \Z} \right) \in \Y$, $\left\|\left( x_1, (l_n)_{n\in \Z} \right) \right\|_\Y = 0$ if and only if $x_1 + \sum_{n=0}^\infty V^n l_n =\mathbf{0} $ and $\sum_{n=-\infty}^{-1} V^{\vert n \vert}l_n =\mathbf{0}.$ Since $x_1\in \X_1$ and $\sum_{n=0}^\infty V^n l_n\in \X_2,$ we have $x_1 = \mathbf{0} = \sum_{n=0}^\infty V^n l_n = \sum_{n=-\infty}^{-1} V^{\vert n \vert}l_n.$ Now, it follows from the uniqueness of every element in $\bigoplus_{n=0}^\infty V^n \mathcal{L}$ that $x_1=\mathbf{0} = l_n$ for all $n\in \mathbb{Z}.$ Consequently, $\left( x_1, (l_n)_{n\in \Z} \right) = (\mathbf{0},\mathbf{0})$.

\medskip

$\bullet$ Triangle inequality: For any $\left( x_1, (l_n)_{n\in \Z} \right)$, $\left( y_1, (r_n)_{n\in \Z} \right) \in \Y$, we have
\begin{align*}
  & \left\|\left( x_1, (l_n)_{n\in \Z} \right)  + \left( y_1, (r_n)_{n\in \Z} \right) \right\|_\Y \\
= & \left( \left\| x_1 + y_1 + \sum_{n=0}^\infty V^n (l_n + r_n) \right\|_\X^2 + \left\| \sum_{n=-\infty}^{-1} V^{\vert n \vert} ( l_n + r_n ) \right\|_\X^2 \right)^{\frac{1}{2}} \\
\leq & \left( \left( \left\| x_1 + \sum_{n=0}^\infty V^n l_n \right\|_\X +  \left\| y_1 + \sum_{n=0}^\infty V^n l_n \right\|_\X \right)^2 + \left(\left\| \sum_{n=-\infty}^{-1} V^{\vert n \vert} l_n \right\|_\X + \left\| \sum_{n=-\infty}^{-1} V^{\vert n \vert} r_n \right\|_\X \right)^2 \right)^{\frac{1}{2}} \\
\leq & \left( \left\| x_1 + \sum_{n=0}^\infty V^n l_n \right\|_\X^2 + \left\| \sum_{n=-\infty}^{-1} V^{\vert n \vert} l_n \right\|_\X^2 \right)^{\frac{1}{2}} + \left( \left\| y_1 + \sum_{n=0}^\infty V^n r_n \right\|_\X^2 + \left\| \sum_{n=-\infty}^{-1} V^{\vert n \vert} r_n \right\|_\X^2 \right)^{\frac{1}{2}} \\
= & \left\|\left( x_1, (l_n)_{n\in \Z} \right) \right\|_\Y + \left\|\left( y_1, (r_n)_{n\in \Z} \right) \right\|_\Y.
\end{align*}

Also, $\left(\Y,\|\cdot\|_\Y\right)$ is a Banach space, since $\left(\Y,\|\cdot\|_\Y\right) \cong V\X_2 \bigoplus_2 \X.$ The function $\|\cdot\|_\sigma : \Y \to [0,\infty)$ defined by
\[
  \left\|\left( x_1, (l_n)_{n\in \Z} \right) \right\|_\sigma = \left( \left\| \sum_{n=-\infty}^{-1} V^{\vert n \vert} l_{n-1} \right\|_\X^2 + \left\| Vx_1 + \sum_{n=0}^\infty V^n l_{n-1} \right\|_\X^2 \right)^{\frac{1}{2}} = \left\|\left( Vx_1, (l_{n-1})_{n\in \Z} \right) \right\|_\Y
\]
is again a norm on $\Y$. The map $U: \left(\Y, \|\cdot\|_\sigma \right) \to \left(\Y, \|\cdot\|_\Y \right)$ defined by 
\[
  U\left(\left( x_1 , (l_n)_{n\in \Z} \right) \right) = (Vx_1, (l_{n-1})_{n\in \Z} ), \qquad \left( x_1 , (l_n)_{n\in \Z} \right) \in \left(\Y, \|\cdot\|_\sigma \right),
\]
is a surjective isometry with inverse $U^{-1}: \left(\Y, \|\cdot\|_\Y \right) \to \left(\Y, \|\cdot\|_\sigma \right) $ given by
\[
U^{-1}\left(\left( x_1 , (l_n)_{n\in \Z} \right) \right) = ((V|_{\X_1})^{-1}x_1, (l_{n+1})_{n\in \Z} ), \qquad \left( x_1 , (l_n)_{n\in \Z} \right) \in \left(\Y, \|\cdot\|_\Y \right).
\]
Thus $\left( \Y, \|\cdot\|_\sigma \right)$ is also a Banach space. The space $\X$ can be embedded into both $\left( \Y, \|\cdot\|_\sigma \right)$ and $\left( \Y, \|\cdot\|_\Y \right)$ by the isometry
\[
   \Lambda_v ( (x_1, \sum_{n=0}^\infty V^n l_n ) ) = \left( x_1, (l_n')_{n\in \Z} \right), \qquad l_n'= \begin{cases}
                                     l_n & \text{ if } n \geq 0 \\
                                     \mathbf{0} & \text{ if } n \leq -1,
                                     \end{cases}
\]
because,
\[
   \left\| \left( x_1, (l_n')_{n\in \Z} \right) \right\|_\sigma = \left\| Vx_1 + \sum_{n=0}^\infty V^{n+1}l_n \right\|_\X = \left\| x_1 + \sum_{n=0}^\infty V^n l_n \right\|_\X = \left\| \left( x_1 (l_n')_{n\in \Z} \right) \right\|_\Y.
\]
Moreover, the surjective isometry $U$ is an extension of $V$, since $U\Lambda_v =\Lambda_v V$. This completes the proof.
\end{proof}

\subsection{The spectrum of a \boldmath\texorpdfstring{$\sigma$}~-shift}

We intend to determine the spectrum of a $\sigma$-shift acting on a Banach space $\X$. When $\X$ is a Hilbert space, the spectrums of unilateral shift and bilateral shift are the closed unit disk $\overline{\mathbb D}$ and unit circle $\mathbb T$, respectively. It was proved in \cite{RMC} that the spectrum of a Banach space isometry that is not invertible is equal to $\overline{\mathbb D}$. Thus, $\overline{\mathbb D}$ is the spectrum of a unilateral shift on a Banach space. In this Subsection, we find the spectrum of a $\sigma$-shift acting on a smooth Banach space.

\begin{prop}\label{Spectrum of sigma shift operator} The spectrum of a $\sigma$-shift acting on a smooth Banach space is the whole unit circle $\mathbb T.$
\end{prop}

\begin{proof} Let $(\Y,\|\cdot\|_\Y)$ be a smooth Banach space and let $\widetilde{V}$ on $\Y$ be a $\sigma$-shift. Evidently, $0\notin \sigma(\widetilde{V})$ and $\sigma(\widetilde{V})\subseteq\overline{\mathbb{D}}.$ Again, $\lambda\in\sigma(\widetilde{V})$ implies that $\lambda^{-1}\in\sigma(\widetilde{V}^{-1}),$ which shows $|\lambda|=1,$ as both $\widetilde V$ and $\widetilde V^{-1}$ are isometries. Therefore, $\sigma(\widetilde{V})\subseteq \mathbb T.$ We now show that $(\widetilde{V}-\lambda I)$ is not surjective for $\lambda\in \mathbb T.$ Let $\mathcal{I}$ be the generating subspace of $\widetilde{V}$ and $\mathbf{\xi}\in\mathcal{I}$ be a nonzero vector. We show that $\mathbf{\xi}\in \Y$ has no inverse image under $(\widetilde{V}-\lambda I).$ Suppose on the contrary that $(\widetilde{V}-\lambda I)z = \mathbf{\xi}$ for some $z \in \Y.$ Let $\left( \sum_{n=-\infty}^\infty \widetilde{V}^nx_n \right)\in \Y$ be the inverse image of $z$ under $\widetilde{V}^{-1}$ i.e.,
\[
  z= \widetilde{V}^{-1}\left( \sum_{n=-\infty}^\infty \widetilde{V}^nx_n \right),\quad x_n\in\mathcal{I}, \; n\in\mathbb{Z}.
\]
Now, $(\widetilde{V}-\lambda I)z=\mathbf{\xi}$ gives
$
  x_0 - \lambda x_1 = \mathbf{\xi},~ x_n - \lambda x_{n+1}=\mathbf{0}$ for $n\in\mathbb{Z}\setminus\{0\}$. We now consider the following two cases:
  
  \smallskip

Case 1: Let $x_0=\mathbf{0}$. Then for every natural number $n$, $x_n = -\frac{\mathbf{\xi}}{\lambda^n}$. The series $\sum_{n=1}^\infty \widetilde{V}^n x_n$ is a convergent series in $\Y$ as $\sum_{n=-\infty}^\infty \widetilde{V}^n x_n$ is convergent in $\Y.$
This leads to a contradiction as for $n> m\geq 0$, we have
\begin{align*}
    \left\| \sum_{k=m}^n\widetilde{V}^kx_k\right\|_\mathbb{Y} & =\left\| \widetilde{V}^m\left(\frac{\mathbf{\xi}}{\lambda^m}\right) + \widetilde{V}^{(m+1)}\left(\frac{\mathbf{\xi}}{\lambda^{m+1}}\right) + \cdots + \widetilde{V}^n\left(\frac{\mathbf{\xi}}{\lambda^n}\right) \right\|_\mathbb{Y} \\
    & = \left\| \mathbf{\xi} + \widetilde{V}\left(\frac{\mathbf{\xi}}{\lambda} \right) + \cdots + \widetilde{V}^{(n-m)}\left(\frac{\mathbf{\xi}}{\lambda^{n-m}}\right) \right\|_\mathbb{Y} \\
    & \geq \left \| \widetilde{V}^{(n-m)}\left(\frac{\mathbf{\xi}}{\lambda^{n-m}}\right) \right\|_\mathbb{Y} = \frac{\|\mathbf{\xi}\|_\Y}{|\lambda|^{n-m}} > 0,
\end{align*}
where the last inequality follows from the fact that $\Y$ is smooth, which is to say that the Birkhoff-James orthogonality is right-additive in $\Y$.

\smallskip

Case 2: Let $x_0\neq \mathbf{0}$. Then for every natural number $n$, $
x_{-n}=\lambda^n x_0$. Evidently, $\sum_{n=-\infty}^0 \widetilde{V}^n(\lambda^{-n} x_0)$ is convergent. However, this is again a contradiction as for $n> m\geq 0$, we have
\begin{align*}
  \left\| \sum_{k=-n}^{-m}\widetilde{V}^kx_k\right\|_\mathbb{Y} & = \left\| \widetilde{V}^{-n}\left(\lambda^n x_0\right) + \widetilde{V}^{(-n+1)}\left(\lambda^{n-1} x_0\right) + \cdots + \widetilde{V}^{-m}\left(\lambda^m x_0\right) \right\|_\mathbb{Y} \\
  & = \left\| x_0 + \widetilde{V}^{-1}\left(\lambda x_0 \right) + \cdots + \widetilde{V}^{(m-n)}\left(\lambda^{n-m}x_0\right) \right\|_\mathbb{Y} \\
  & \geq \left \| x_0 \right\|_\mathbb{Y} > 0,
\end{align*}
where the last inequality follows from the smoothness of $\Y$. Consequently, $\mathbb T \subseteq \sigma(\widetilde{V})$ and the proof is complete.
\end{proof}


\section{Dilations in Banach spaces and characterizations of Hilbert spaces} \label{sec:07}

\vspace{0.3cm}

\noindent In this Section, we shall find a necessary and sufficient condition such that a strict Banach space contraction dilates to a Banach space isometry. Then we show by example that there are Banach space contractions that do not dilate to any Banach space isometries. Even more is true; we show that all strict Banach space contractions dilate to Banach space isometries if and only if the Banach space is a Hilbert space. Not only we characterize a Hilbert space in terms of isometric dilation, but also we find a set of conditions each of which is necessary and sufficient for a Banach space to become Hilbert space. In the entire program, a pivotal role is played by the following non-negative real-valued function associated with a contraction $T$ on a Banach space $\X$:
\begin{equation}\label{Equivalent norm function}
A_T(x) = \left(\|x\|^2 -\|Tx\|^2\right)^{\frac{1}{2}}, \qquad x\in \X.  
\end{equation}

Indeed, we shall see that a strict contraction $T$ on a Banach space $\X$ dilates to a Banach space isometry if and only if $A_T$ defines a norm on $\X$. Before proceeding further with this function, let us recall the definition of isometric dilation of a Banach space contraction.
\begin{defn}\label{Definition of Dilation of contractions on Banach space}
Suppose $T$ is a contraction acting on a Banach space $\mathbb{X}$. An isometry $V$ on a Banach space $\widetilde{\mathbb{X}}$ is said to be an {isometric dilation} of $T$ if there is an isometry $W:\mathbb{X}\to \widetilde{\mathbb{X}}$ and a closed linear subspace $\BL$ of $\widetilde{\X}$ such that $\widetilde{\mathbb{X}}$ is isomorphic with $ W(\mathbb{X}) \oplus_2 \BL$ and the operator $\widehat{T}:= \widehat{W} T \widehat{W}^{-1}: W(\X) \to W(\X)$ satisfies
\[
q(\widehat{T})= P_{_{W(\mathbb{X})}} q(V)|_{W(\X)}
\]
for all polynomials $q\in \mathbb{C}[z]$, where $\widehat{W}$ is the unitary (i.e. surjective isometry) $\widehat{W}:= W:\mathbb{X}\to W(\mathbb{X})$ and $P_{_{W(\mathbb{X})}}$ is the norm-one projection of $\widetilde{\X}$ onto $W(\mathbb{X})$. Moreover, such an isometric dilation is called {minimal} if
\[
 \widetilde{\X} = \bigvee_{n=0}^\infty V^n W(\X) = \overline{span}\{V^nWx:~x\in \X,~n\geq 0\}.
\]
\end{defn}

\medskip

The first notable fact about the function $A_T$ is that if it defines a norm on a Banach space $\X$ when $\|T\|<1$, then the space $(\X, A_T)$ is again a Banach space. Surprisingly, if $\|T\|=1$ and if yet $A_T$ defines a norm on a Banach space $\X$, then $(X, A_T)$ is no longer a Banach space. Also, $(\X, A_T)$ is always a Hilbert space when $\X$ is a Hilbert space. However, $(\X, A_T)$ may not be a Hilbert space if we start with a non-Hilbert Banach space $\X$, which we shall show by an example. 

\begin{lem}\label{lem: Equv. of norm}
Let $\X$ be a Banach space and $T \in \mathcal{B}(\X)$ be a strict contraction such that the function $A_T$ as in $(\ref{Equivalent norm function})$ defines a norm on $\X$. Then $\X_0 = (\X, A_T)$ is a Banach space and $A_T$ is equivalent to the original norm on $\X$.
\end{lem}

\begin{proof}
Let $\|.\|$ be the original norm on $\X$. Suppose $(x_n)$ is a Cauchy sequence in $\X_0$. Then for every $\epsilon >0$, there is $n_0 \in \mathbb{N}$ such that
\begin{align*}
    & \left( \|x_n-x_m\|^2 - \|T(x_n - x_m) \|^2 \right)^{ 1\slash 2} < \epsilon \left(1 - \|T\|^2\right)^{1 \slash 2},\quad n, ~m \geq n_0, \\
    & \Rightarrow \|x_n-x_m\|^2 < \epsilon^2 (1 - \|T\|^2) + \|T(x_n-x_m)\|^2, \quad n, ~ m \geq n_0,  \\
    &  \Rightarrow \|x_n-x_m\|^2 < \epsilon^2 (1 - \|T\|^2) + \|T\|^2 \|x_n-x_m\|^2, \quad n,~ m \geq n_0,  \\
    & \Rightarrow \|x_n-x_m\|^2 (1-\|T\|^2) < \epsilon^2 (1 - \|T\|^2), \quad n, ~ m \geq n_0,  \\
    &  \Rightarrow \|x_n-x_m \| < \epsilon, \quad n, ~ m \geq n_0, \qquad \left[ \because(1- \|T\|^2) >0 \right]
\end{align*}
which shows that $(x_n)$ is a Cauchy sequence in $(\X, \|.\|)$. Since $(\X, \|.\|)$ is Banach space, there is $x\in \X$ such that $x_n \to x$ with respect to $\|.\|$ as $n\to \infty$. We now show that the Cauchy sequence $(x_n)$ converges to $x$ in $\X_0$. By the continuity of $T$ and the norm, we have that
\[
  A_T(x_n -x) = \left( \|x_n-x\|^2 - \|T(x_n -x) \|^2 \right)^{ 1\slash 2} \longrightarrow 0, \quad \text{as} \quad n \to \infty.
\]
Thus, $\X_0$ is a Banach space. Moreover, we have $A_T(x) \leq \|x\|$ for all $x\in \X$, which shows that the identity map $I_d: \X \to \X_0$ is a bounded bijective linear map. Therefore, by open mapping theorem, $I_d^{-1}: \X_0 \to \X$ is also a bounded linear map. Therefore, there exists a constant $k >0$ namely $k=\|I_d^{-1}\|$ such that 
\[
    A_T(x) \leq \|x\| \leq k~  A_T(x), \quad \text{ for all }x\in \X.
\]
Thus, the norm $A_T$ on $\X$ is equivalent to the original norm $\|.\|$ on $\X$ and the proof is complete.

\end{proof}

\begin{lem} \label{lem:new-0122}

Let $\X$ be a Banach space and let $T\in \mathcal{B}(\X)$ be such that $A_T$ defines a norm on $\X$. If $\|T\|=1$ then $A_T$ is not equivalent to the original norm on $\X$. Moreover, the normed linear space $\X_0= (X, A_T)$ is not a Banach space.

\end{lem}

\begin{proof}

Let $\|.\|$ be the original norm on $\X$. It follows from the definition of $A_T$ (as in (\ref{Equivalent norm function})) that $A_T(x) \leq \|x\|$ for all $x \in \X$. Let if possible there be $c>0$ be such that
\begin{equation}\label{Eq: Defn. of Equiv. norm}
  c\|x\| \leq A_T(x) \leq \|x\|, \qquad x\in \X.
\end{equation}
Since $\|T\|=1$, there is a sequence $(x_n)$ in $S_{\X}$ such that ${\displaystyle \lim_{n \rightarrow \infty} \|Tx_n\| =1 }$. Consequently,
\[
  A_T(x_n) = \left( 1 - \|Tx_n\|^2 \right)^{1 \slash 2} \longrightarrow 0 \quad  \text{ as } \quad n \to \infty. 
\]
However, $(x_n)$ does not converge to $\mathbf{0}$ in $\X$ with respect to the original norm, which is a contradiction to (\ref{Eq: Defn. of Equiv. norm}). So $A_T$ is not equivalent to the original norm.

\smallskip

From the right inequality in (\ref{Eq: Defn. of Equiv. norm}), it follows that the identity map $I_d: \X \to \X_0$ is a bounded bijective linear map. If $\X_0$ were a Banach space, the open mapping theorem would imply that $I_d^{-1}$ is also a bounded linear map, which would in turn imply  the existence of a scalar $ c > 0 $, e.g. $c=\|I_d^{-1}\|$, satisfying (\ref{Eq: Defn. of Equiv. norm}). However, this contradicts the first part of the proof and we are done.
 
\end{proof}

\begin{lem}
Let $\mathcal{H}$ be a Hilbert space and $T$ be a strict contraction on $\mathcal{H}$. Then the function $A_T$ as in $(\ref{Equivalent norm function})$ defines a norm on $\mathcal{H}$ and $\mathcal{H}_0 = (\mathcal{H}, A_T)$ is also a Hilbert space.
\end{lem}

\begin{proof}
The defect operator $D_T$ is a linear map that satisfies $\|D_T(h)\|^2 = A_T(h)$ for all $h\in \mathcal{H}$. It follows from Lemma \ref{lem: Equv. of norm} that $\mathcal{H}_0$ is a Banach space. Moreover, linearity of $D_T$ implies that the map $\langle \cdot , \cdot \rangle_{_0} : \mathcal{H} \times \mathcal{H} \to \mathbb{C}$ defined by $\langle x, y\rangle_{_0} = \langle D_T (x), D_T (y) \rangle$ is an inner product on $\mathcal{H} \times \mathcal{H}$ and induces the norm $A_T$. This completes the proof.
\end{proof}

However, in general the space $\X_0$ may not be a Hilbert space if the initial space $\X$ is a non-Hilbert Banach space. The following example illustrates this.

\begin{eg}
Let $(\X, \|.\|)$ be a non-Hilbert Banach space and let $0 < c < 1$. Consider the strict contraction $T= cI$. Then $A_T(x) = (1-c^2)^{\frac{1}{2}} \|x\|$ for all $x\in \X$. Then $\X_0$ is not a Hilbert space, because, if $A_T$ is induced by an inner product $\langle \cdot, \cdot \rangle : \mathcal{H} \times \mathcal{H} \to \mathbb{C}$, then $\|\cdot\|$ is induced by the inner product $\langle x, y\rangle_{_0}= \langle rx, ry \rangle$, for all $x,y \in \X$ and $r= \frac{1}{\sqrt{1-c^2}}$, which contradicts the assumption. \qed
\end{eg}

More precisely, if $\X$ is a non-reflexive Banach space and if $T\in \mathcal{B}(\X)$ be such that $A_T$ defines a norm on $\X$, then $\X_0$ is not a Hilbert space. Indeed, if $\X_0$ is a Hilbert space then $\X_0$ is a reflexive Banach space, which implies $\X$ is reflexive as reflexivity is preserved under equivalent norms. Our next result gives a necessary and sufficient condition such that the function $A_T$ defines a norm when $T$ is a strict contraction on a Banach space.

\begin{prop}
Let $\mathbb{X}$ be a Banach space and let $T$ be a strict contraction on $\mathbb{X}$. Then the function $A_T: \mathbb{X}\to [0, \infty)$ as in $(\ref{Equivalent norm function})$ defines a norm on $\mathbb{X}$ if and only if the following inequality holds:
\begin{equation}\label{Triangle inequality}
\|T(x+y)\|^2 + (\|x\|^2 +\|y\|^2) -\|x+y\|^2 \geq \|Tx\|^2 +\|Ty\|^2 - 2 A_T(x)A_T(y),\quad x,y \in \mathbb{X}. 
\end{equation}
\end{prop}

\begin{proof}
Since $T$ is a strict contraction, $A_T(x)\geq 0$ for all $x\in \X$ and the equality holds if and only if $x=\mathbf{0}$. Also, $A_T$ is absolute homogeneous. In addition, for all $x,y\in \X$, the triangle inequality $A_T(x+y)\leq A_T(x)+A_T(y)$ is equivalent to 
\[
  \|x+y\|^2 -\|T(x+y)\|^2 \leq \|x\|^2 +\|y\|^2 -(\|Tx\|^2 + \|Ty\|^2) + 2A_T(x)A_T(y).
\]
Consequently, we have
\[
 \|T(x+y)\|^2 +(\|x\|^2+\|y\|^2) - \|x+y\|^2 \geq \|Tx\|^2 + \|Ty\|^2 - 2 A_T(x)A_T(y).
\]
This completes the proof.
\end{proof}

We are on our way of proving that a Banach space $\X$ is a Hilbert space if and only if $A_T$ defines a norm on $\X$ for every strict contraction $T \in \mathcal B(\X)$. We prove another lemma before that.

\begin{lem}\label{BJ Pythagorean orthogonality}

Let $\mathbb{X}$ be a Banach space and let the function $A_T: \mathbb{X}\to [0,\infty)$ as in $(\ref{Equivalent norm function})$ be a norm on $\X$ for every strict contraction $T$ on $\X$. Then for every $x,~y\in \mathbb{X}$ with $x\perp_B y$ the following holds:
\[
   \|x+y\|^2 = \|x\|^2 +\|y\|^2 = \|x-y\|^2.
\]
\end{lem}

\begin{proof}

The result follows if any of $x,y$ is zero. Let $x,~y\in \mathbb{X}$ be nonzero with $x\perp_B y.$ By Theorem \ref{James Characterization}, there is a support functional $f_x$ at $x$ such that $f_x(y)=0.$ Let $0<r<1$ be arbitrary and consider the operator $T:\mathbb{X}\to \mathbb{X}$ defined by 
\[
  T(z): = \frac{r}{\|x\|~\|y\|}f_x(z)~y \ ,\qquad z\in \mathbb{X}.
\]
Now $T$ is a strict contraction as $\|T\| =r <1$. Also, we have
\[
  T(y)=\mathbf 0,~ \|T(x)\| = r\|x\|,~ A_T(x)= \|x\|(1-r^2)^{\frac{1}{2}}, ~ A_T(y) = \|y\|.
\]
Since $A_T$ is a norm on $\mathbb{X}$, by (\ref{Triangle inequality}), we have
\[
  \|T(x+y)\|^2 + \|x\|^2 +\|y\|^2 -\|x+y\|^2 \geq \|Tx\|^2 + \|Ty\|^2 -2 A_T(x)A_T(y),
\]
which on simplification gives
\[
   \|x+y\|^2 \leq \|x\|^2 +\|y\|^2 + 2\|x\|\|y\|(1-r^2)^{\frac{1}{2}},~\quad 0<r< 1.
\]
Now, taking limit as $r\to 1$ we have
\begin{align}\label{Equation 9}
  \|x+y\|^2 \leq \|x\|^2 +\|y\|^2.
\end{align}
Next, consider another strict contraction $B:\mathbb{X} \to \mathbb{X}$ defined by 
\[
B(z): = \frac{r}{\|x\|~\|y-x\|}f_x(z)(y-x),\qquad  z\in \mathbb{X}.
\]
Since $A_B$ is a norm on $\mathbb{X}$, by (\ref{Triangle inequality}) we have
\[
  \|B(y-x +x)\|^2 + \|y-x\|^2 +\|x\|^2 -\|y-x+x\|^2 \geq \|B(y-x)\|^2 + \|Bx\|^2 -2 A_B(x)A_B(y-x),
\]
which on simplification gives
\[
   \|y-x\|^2 + \|x\|^2 \geq \|y\|^2 + 2r^2\|x\|^2 - 2\|x\|\sqrt{1-r^2}~A_B(y-x),~\quad 0<r< 1.
\]
Again, taking limit as $r\to 1$ we have
\begin{align}\label{Equation 10}
    \|y-x\|^2 \geq \|y\|^2 + \|x\|^2.
\end{align}
Combining (\ref{Equation 9}) and (\ref{Equation 10}) we have our desired identity
$
    \|x-y\|^2 =\|x\|^2 + \|y\|^2 = \|x+y\|^2.
$
The proof is now complete.
\end{proof}

\begin{thm}\label{Norm if and only if Hilbert space}
Let $\mathbb{X}$ be a Banach space. Then, $\mathbb{X}$ is a Hilbert space if and only if for every strict contraction $T\in \mathcal{B}(\mathbb{X}),$ the function $A_T: \mathbb{X}\to [0,\infty)$ as in $(\ref{Equivalent norm function})$ defines a norm on $\mathbb{X}.$
\end{thm}

\begin{proof}
We first prove the necessary part. The positivity and absolute homogeneity of $A_T$ follow from $\|Tx\| <\|x\|$ ($x\in \X$) and linearity of $T$ respectively. The triangle inequality of $A_T$ follows from the linearity of the defect operator $D_T$ and $A_T(x) = \|D_T(x)\|$ for all $x\in \X$. Indeed,
\[
 A_T(x+y) = \|D_T(x+y)\| \leq \|D_T(x)\| + \|D_T(y)\| = A_T(x) + A_T(y), \qquad x,y \in \X.
\] 
Now, we prove the sufficient part by dividing the proof in the following two cases.

\medskip

 Case 1: Let $dim(\mathbb{X}) \geq 3.$ We show that Birkhoff-James orthogonality is symmetric on $\mathbb{X}.$ Let $x,~y\in \mathbb{X}$ be nonzero vectors with $x\perp_B y.$ Without loss of generality, we may assume that $\|x\|=1=\|y\|.$ We claim that $f_y(x)=0$ for every support functional $f_y$ at $y.$ If possible, let $f_y(x)\neq 0$ for some support functional $f_y$ at $y.$ Then there is a $\delta >0$ such that $0< \delta < \vert f_y(x) \vert \leq 1.$ Choose $\frac{1}{(1 + \delta^2)^{\frac{1}{2}}}< r < 1$ and consider the element $(\delta e^{i\theta}x +y)$ in $\mathbb{X}$, where $e^{i\theta}=sgn(f_y(x)) = \frac{\overline{f_y(x)}}{\vert f_y(x)\vert}.$ Since $\delta e^{i\theta}x \perp_B y$ and $r < 1,$ by Lemma \ref{BJ Pythagorean orthogonality} we have
\[
  (1+\delta^2)^{\frac{1}{2}} = \|\delta e^{i\theta}x+y\| > rf_y(\delta e^{i\theta}x+y) = r\left(\delta \vert f_y(x)\vert +1\right) > r\left( \delta^2 +1 \right),
\]
which is a contradiction as $ r> \frac{1}{(1+\delta^2)^{\frac{1}{2}}}.$
Consequently, $f_y(x)=0$ for every support functional $f_y$ at $y$ and $y\perp_B x$, as desired. Therefore, $\X$ is Hilbert space by Theorem \ref{Symmetry of Birkhoff-James orthogonality}.

\medskip

 Case 2: Let $dim(\mathbb{X}) \leq 2.$ If $\dim(X)=1$, then the proof is trivial. Now, assume that $dim(\mathbb{X})=2.$ Let $u,~v$ be two unit vectors in $\X$ with $u\perp_B v$. Then, any $x\in\mathbb{X}$ can be uniquely expressed as $x=au+bv$ for some $a,b\in \mathbb{C}.$ Since Birkhoff-James orthogonality is homogeneous, by Lemma \ref{BJ Pythagorean orthogonality}, we have
\begin{align}\label{Equation 11}
   \|x\|^2= \|au+bv\|^2 = \|au\|^2 +\|bv\|^2 = \vert a\vert ^2 + \vert b\vert ^2.
\end{align}
Now, it is easy to see that the identification $x\mapsto (a,b)$ is an isometric isomorphism from $\X$ to $\mathbb{C}^2$ and hence $\X$ is a Hilbert space. This completes the proof.
\end{proof}
The following corollary is a straight-forward application of Theorem \ref{Norm if and only if Hilbert space}.

\begin{cor}\label{Norm on X and its dual}
Suppose $\mathbb{X}$ is a Banach space. Then the following are equivalent.
\begin{enumerate}
  \item[(i)] For every strict contraction $T$ on $\mathbb{X},$ the function $A_T:\mathbb{X}\to [0,\infty)$ as in $(\ref{Equivalent norm function})$ is a norm on $\mathbb{X}.$
  
\item[(ii)] For every strict contraction $S$ on $\mathbb{X}^*,$ the function $A_{S}:\mathbb{X}^*\to [0,\infty)$ defined by $(\ref{Equivalent norm function})$
is a norm on $\mathbb{X}^*.$
\end{enumerate}
\end{cor}

\begin{proof}
If $A_T$ defines a norm on $\X$ for every strict contraction $T$ on $\X$, then by Theorem \ref{Norm if and only if Hilbert space}, $\X$ is a Hilbert space. Thus, $\X^*$ is a Hilbert space. Then, for every strict contraction $S$ on $\X^*$ and for any $x^* \in \X^*$, we have $A_S(x^*)=\|D_S(x^*)\|$. So, $A_S$ is a norm on $\X^*$. On the other hand, if $A_S$ defines a norm on $\X^*$ for every strict contraction $S$ on $\X^*$, then again by Theorem \ref{Norm if and only if Hilbert space}, $\X^*$ is Hilbert and hence $\X^*$ is reflexive. Consequently, $\X$ is isometrically isomorphic to $(\X^*)^*$ which is a Hilbert space. This completes the proof.
\end{proof}

Recall from (\ref{Forward shift}) \& (\ref{backward shift}) the forward and backward shift operators $M_z, \, \widehat{M}_{z}$ on $\ell_2(\X)$ for any Banach space $\X$. We have already seen that $M_z$ is a unilateral shift. The next theorem establishes a connection between backward shift operator $\widehat{M}_{z}$ on $\ell_2(\X)$ and the function $A_T$.

\begin{thm}\label{Norm and Backward shift}
Let $\mathbb{X}$ be a Banach space. Then the following are equivalent.
   \begin{enumerate}
   \item[(i)] For every strict contraction $T\in \mathcal{B}(\mathbb{X}),$ the function $A_T: \mathbb{X}\to [0,\infty)$ as in $(\ref{Equivalent norm function})$ is a norm on $\X$.
   
   \item[(ii)] For every strict contraction $T$ on $\mathbb{X},$ there exists an isometry $W: \mathbb{X}\to \ell_2(\mathbb{X})$ such that 
\[
 \widehat{M}_{z} W(x)= W(Tx), \qquad \text{for every } \, x\in \mathbb{X}.
\]
   \end{enumerate}
\end{thm}

\begin{proof}
(i)$\implies$(ii). It follows from Theorem \ref{Norm if and only if Hilbert space} that $\mathbb{X}$ is a Hilbert space. For any strict contraction $T$ on $\mathbb{X},$ consider the operator $W: \mathbb{X}\to \ell_2(\mathbb{X})$ defined by 
\[
   W(x)= (D_T(x), D_T(Tx), D_T(T^2x), \cdots), \qquad x\in \mathbb{X}.
\]
Observe that 
\[
   \| W(x) \|^2 = \lim_n \sum_{i=0}^n \| D_T (T^ix) \| = \|x\|^2 - \lim_n \|T^n x\|^2 = \|x\|^2, 
\]
where the last equality follows from the fact that $T$ is a strict contraction. Thus, $W$ is an isometry and we also have
\[
\widehat{M}_{z} W(x) =( D_T(Tx), D_T(T^2x), \cdots)= W(Tx), \qquad x\in \mathbb{X}.
\]

\medskip

\noindent (ii)$\implies$(i). Let $T$ be a strict contraction on $\mathbb{X}$ and let there be an isometry $W: \X \to \ell_2(\mathbb{X})$ satisfying the stated condition. Let 
\[
W(x)= (W_1(x), W_2(x), \dots, ), \qquad x\in \mathbb{X}.
\]
Evidently, $\|W_i\|\leq 1$, for each $i=1,2,\dots$. It follows from the hypothesis that 
$
\widehat M_{z}^ kW = WT^k$ for all $k\geq 1,$
and this shows that
\[
  ( W_{k+1}(x),W_{k+2}(x),\dots ) = (W_1(T^kx),W_2(T^kx),\dots ),\quad x\in\mathbb{X},\quad k\geq 1.
\]
Thus, for each $k\geq 1$ we have
$
W_{k+1}(x)= W_1(T^kx)$ for every $x\in \X$. Consequently, the isometry $W$ is of the following form:
\[
  W(x) = \left( W_1(x), W_1(Tx), W_1(T^2x),\cdots \right), \qquad x\in \mathbb{X}.
\]
So, we have
\[
  \|x\|^2 - \|Tx\|^2 =\| W(x)\|^2 - \|W(Tx)\|^2 = \|W_1(x)\|^2 ,\qquad x\in \mathbb{X},
\]
and the linearity of $W_1$ shows that the function $A_T$ as in (\ref{Equivalent norm function}) satisfies the triangle inequality. The positivity of $A_T$ follows from that fact that $\|Tx\| <\|x\|$ for all $x\in \X$ and the absolute homogeneity of $A_T$ follows from the linearity of $T$. The proof is now complete.
\end{proof}

Recall that $\mathsf{Aut}(\mathbb{D})$ denotes the automorphism group of the unit disk $\mathbb D$ and that any automorphism of $\mathbb D$ is of the form
\[
 \phi_\lambda(z) = e^{i\theta} \frac{z-\lambda}{1-\overline{\lambda}z},  \qquad z\in \mathbb{D}
\]
for some $\lambda\in \mathbb{D}$ and $\theta\in [0,2\pi)$. We learn from the literature (e.g. see \cite[Chapter 1, Section 4]{Nagy Foias}) that $\phi_\lambda(T) = (T-\lambda I) (I - \bar{\lambda}T)^{-1}$ is also a contraction for every strict contraction $T$ on a Hilbert space. Our next result shows that the converse is also true.

\begin{thm}\label{Automorphism of unit disc and contraction}
A Banach space $\mathbb{X}$ is a Hilbert space if and only if $\phi_\lambda(T)$ is a contraction for every strict contraction $T\in \mathcal{B}(\mathbb{X})$ and every automorphism $\phi_\lambda$ of the unit disk $\mathbb D$.
\end{thm}

\begin{proof}
As mentioned above, the `necessary part' of this theorem follows from \cite[Chapter 1, Section 4]{Nagy Foias}. We prove here the `sufficiency part' and shall follow the techniques of the proof of Theorem 1.9 in \cite{Pisier}, which was originally proved by C. Foias. Let $r \in (0,1)$ and let $x,~y\in \mathbb{X}$ be any two unit vectors. Then the operator $T: \mathbb{X}\to \mathbb{X}$ defined by $T(z)= r f_x(z)y, ~ z\in \mathbb{X}$ is a strict contraction. It follows from the hypothesis that $\phi_\lambda(T)$ is a contraction for every $\phi_\lambda\in \mathsf{Aut}(\mathbb{D}).$ Thus, $\|(T-\lambda I) (I - \bar{\lambda}T)^{-1}z\| \leq \|z\|$ for any $z\in \mathbb{X}$. This gives $\|(T-\lambda I)z\| \leq \|(I - \bar{\lambda}T)z\|$ for all $z\in \X$. Choosing $z=x$ we have $T(x)=ry$ and hence $\|ry-\lambda x\| \leq \| x- r\bar{\lambda}y\|$. Taking limit as $r\to 1$, we have
\[
   \|y-\lambda x\| \leq \| x-\bar{\lambda}y \|, \quad \lambda\in \mathbb{D}.
\]
Interchanging $x$ and $y$ in the definition of $T$ and also in the above inequality, we have
\[
    \|x-\bar{\lambda} y\| \leq \| y-\lambda x \|, \quad \lambda\in \mathbb{D}.
\]
So, we have $\|y-\lambda x\| = \| x-\bar{\lambda}y \|$ for all $\lambda \in \mathbb D$ and thus in particular $\|y-\xi x\| = \| x- \xi y \|$ for any $\xi \in (-1,1)$. So, for any $\xi \in (-1,1)$ we have
\[
 \left\|\frac{1}{\xi}y-x\right\| = \left\|\frac{1}{\xi}x-y\right\|.
\]
It follows from here that $\|ax-by\| =\|bx-ay\|,$ for any real numbers $a,b$ and any vectors $x,y \in \X$ with $\|x\|=\|y\|$. We now apply Corollary 2 of \cite{Ficken} which states the following:

\smallskip

``In order that a normed linear space with complex scalars may permit the definition of a scalar product, it is necessary and sufficient that, whenever $\|P\| = \|Q\|$, and $a$ and $b$ are real scalars,
\[
   \| aP + bQ \| = \| bQ + aP \|."
\]
So, $\X$ is a Hilbert space by \cite[Corollary 2]{Ficken}. This completes the proof.
\end{proof}

Being armed with all these results we are now going to present one of the main results of this paper.

\begin{thm}\label{Characterisation of Hilbert spaces in terms of Dilation on Banach space}
Let $\mathbb{X}$ be a complex Banach space. Then the following are equivalent.

\smallskip

\begin{enumerate}

\item[(i)] $\mathbb{X}$ is a Hilbert space.

\smallskip

\item[(ii)] Every strict contraction $T$ on $\X$ dilates to the unilateral shift $M_z$ on $\ell_2(\mathbb{X}).$

\smallskip

\item[(iii)] Every strict contraction $T$ on $\X$ dilates to an isometry.

\smallskip

\item[(iv)]  For every strict contraction $T\in \mathcal{B}(\mathbb{X}),$ the function $A_T: \mathbb{X}\to [0,\infty)$ given by 
\[
  A_T(x) = \left( \| x\|^2 - \|Tx\|^2 \right)^{\frac{1}{2}}, \quad x\in \mathbb{X},
\]
defines a norm on $\mathbb{X}.$

\smallskip

\item[(v)] For every strict contraction $T\in \mathcal{B}(\mathbb{X}),$ there exists an isometry $W: \mathbb{X}\to \ell_2(\mathbb{X})$ such that 
\[
 \widehat{M}_{z} W(x)= W(Tx), \quad x\in \mathbb{X},
\]
where $\widehat{M}_{z}$ is the backward shift operator.

\smallskip

\item[(vi)] For every strict contraction $T \in \mathcal B(\X)$ and for every automorphism $\phi_{\lambda}$ of the unit disk $\mathbb D$, the operator $\phi_\lambda(T)$ is a contraction.

\smallskip

\item[(vii)] For every strict contraction $S$ on the dual space $\mathbb{X}^*,$ the function $A_{S}:\mathbb{X}^*\to [0,\infty)$ given by 
\[
   A_S(x^*) = \left( \|x^*\|^2 - \|S(x^*)\|^2 \right)^{\frac{1}{2}},\quad x^*\in \mathbb{X}.
\]
defines a norm on $\mathbb{X}^*.$

\smallskip

\item[(viii)] $\mathbb{X}$ is reflexive and the Banach adjoint $T^\times$ of $T$ dilates to $M_z$ on $\ell_2(\mathbb{X}^*)$ for every strict contraction $T$ on $\mathbb{X}$.

\smallskip

\item[(ix)] The operator $\phi_\alpha(U)= (U-\alpha I)(I-\bar{\alpha}U)^{-1}$ is a contraction for every automorphism $\phi_\alpha$ of the unit disk $ \mathbb{D}$, $U$ being the bilateral shift operator on $\ell_2(\mathbb{Z},\mathbb{X})$ defined by
\[
U((\dots, x_{-2}, x_{-1}, \boxed{x_0}, x_1, x_2, \dots))= (\dots, x_{-2}, \boxed{x_{-1}}, x_0, x_1, x_2, \dots),
\]
where the box on either side indicates the $0$-th position.
\end{enumerate}

\end{thm}

\begin{proof}
The equivalences (i)$\iff$(iv), (iv)$\iff$(v), (i)$\iff$(vi) and (iv)$\iff$(vii) follow from Theorem \ref{Norm if and only if Hilbert space}, Theorem \ref{Norm and Backward shift}, Theorem \ref{Automorphism of unit disc and contraction} and Corollary \ref{Norm on X and its dual} respectively. We now prove (i)$\implies$(ii)$\implies$(iii) $\implies$ (iv), (vii)$\iff$(viii) and (i) $\iff$ (ix).

\medskip

\noindent (i)$\implies$(ii). A proof to this can be found in the literature, e.g. see \cite{Nagy Foias}. However, for the sake of completeness, we briefly outline a proof here. Let $T$ be a strict contraction on $\mathbb{X}$ and let $T^*$ be its adjoint. Let $W: \mathbb{X}\to \ell_2(\mathbb{X})$ be defined by 
\[
   W(x)= ( D_{T^*}x, D_{T^*}T^*x, D_{T^*}T^{*2}x, \cdots),\qquad x\in \mathbb{X}.
\]
Evidently, $W$ is linear map, and for each $x\in \X$ we have
\[
   \| W(x)\|^2 = \lim _{n \rightarrow \infty} \sum_{k=0}^n \| D_{T^*}T^{*k}x \|^2 = \lim_n \sum_{k=0}^n \left(\|T^{*k}x\|^2 - \|T^{*{k+1}}x\|^2 \right) = \|x\|^2 - \lim_n \|T^{*n}x\|^2 = \|x\|^2,
\]
which shows that $W$ is an isometry. Therefore, $W W^* = P_{_{W(\X)}}$, the orthogonal projection on $\text{Ran}(W)$ and $\ell_2(\mathbb{X})= \text{Ran}(W)\oplus_2 \ker( W W^*)$. It is evident that $M_z^* W(x) = W(T^*x)$ for all $x\in \X$, where $M_z^*=\widehat M_{z}$ is the backward shift operator on $\ell_2(\X)$ defined as in (\ref{backward shift}). Also, since $W^*W= I_\mathbb{X}$, we have $T^*=W^* M_z^* W$, that is, $M_z^*$ is a co-isometric extension of $T^*$. Thus, $M_z$ on $\ell_2(\X)$ is an isometric dilation of $T$.

\medskip

\noindent (ii)$\implies$(iii). Obvious.

\medskip

\noindent (iii)$\implies$(iv). Every strict contraction dilates to isometry. Suppose $T$ is a strict contraction that dilates to an isometry $V$.  Since $T$ is a strict contraction, $A_T$ satisfies positivity and absolute homogeneity. It remains to show that $A_T$ satisfies the triangle inequality. Being consistent with the notations of Definition \ref{Definition of Dilation of contractions on Banach space}, we have
\[
   \|x\|^2 - \|Tx\|^2 = \|V Wx\|^2 - \| P_{_{W(\mathbb{X})}}V W x\|^2 = \|(I-P_{_{W(\mathbb{X})}})V W x\|^2, \qquad x\in \mathbb{X}.
\]
Therefore, 
\[
  A_T(x) = \|(I-P_{_{W(\mathbb{X})}})V W x\|, \qquad x\in \mathbb{X},
\]
which shows $A_T(x+y)\leq A_T(x)+A_T(y)$ for all $x,~y\in \X$, as desired.

\medskip

\noindent (vii)$\implies$(viii). It follows from the proof of (iv)$\implies$(i) that $\X^*$ is a Hilbert space and thus $\X$ is reflexive. For every strict contraction $T$ on $\X$, $T^\times$ on $\X^*$ is also a strict contraction. Therefore, $T^\times$ dilates to $M_z$ on $\ell_2(\X^*)$ by (i)$\implies$(ii).

\medskip

\noindent (viii)$\implies$(vii): We first show that given any strict contraction $S$ on $\X^*$, $S=T^\times$ for some strict contraction $T\in \mathcal{B}(\mathbb{X}).$ For any strict contraction $S$ on $\X^*$, define $T: \mathbb{X} \to \mathbb{X}$ by 
\[
   T(x)= \pi_\X^{-1} S^\times \pi_\X (x), \qquad x\in \mathbb{X},
\]
where $\pi_\X: \X \to \X^{**}$ is the canonical embedding defined by $\pi_\X (x) = \hat{x}$ with $\hat{x}(x^*) = x^*(x)$ for all $x^*\in \X^*$. Evidently, $\|T\| = \|S^\times \| = \|S\| < 1.$ Let $f\in \mathbb{X}^*$ and $y\in \mathbb{X}$ be arbitrary. Thus,
\[
  T^\times (f)(y) = f(Ty) = f( \pi_\X^{-1} S^\times \pi_\X(y)).
\]
Observe that $S^\times\pi_\X(y)$ is a member of $\X^{**}$ which we denote by $\pi_\X(z)$ for some $z\in \X$. Therefore,
\[
  f(z) = \pi_\X(z)f= (S^\times\pi_\X(y) ) (f) = \pi_\X(y)( Sf) = Sf(y),
\]
which shows that $T^\times = S$. By hypothesis the contraction $S$ dilates to $M_z$ on $\ell_2(\X^*)$. The rest of the the proof follows from the proof of the implication (iii)$\implies$(iv) by replacing $V$ and $\X$ by $M_z$ and $X^*$ respectively, where $V$ is the isometric dilation as in the proof of (iii)$\implies$(iv).

\smallskip

\noindent (i)$\iff$(ix): If $\X$ is a Hilbert space, then $\phi_\alpha(U)$ is obviously a contraction as the bilateral shift $U$ is a contraction. 

\smallskip

Now, we prove the sufficiency part. Suppose that $\phi_\alpha(U)= (U-\alpha I)(I-\bar{\alpha}U)^{-1}$ is a contraction for every $\alpha \in \mathbb D$. Let $\underline{x}\in \ell_2(\mathbb Z, \X)$ be arbitrary and set $\underline{y}= (I-\bar{\alpha} U)(\underline{x})$. Then $\|\phi_\alpha(U)\underline{y}\|\leq \|\underline{y}\|$, which further implies that 
\[
\|(U-\alpha I)\underline{x}\| \leq \|(I-\bar{\alpha} U)\underline{x}\| = \|U^{-1}\left(I -\bar{\alpha}U\right)\underline{x}\| = \|\left(U^{-1} - \bar{\alpha}I\right)\underline{x}\|.
\]
Now, let $x,y\in \mathbb{X}$ are arbitrary with $\|x\|=\|y\|$. Consider the vector $\underline{\mathbf{x}}= (\cdots, \mathbf{0},\boxed{x}, y, \mathbf{0}, \cdots)\in \ell_2(\mathbb{Z},\mathbb{X}).$ Then for every $\alpha\in \mathbb{D}$, we have 
\begin{align*}
  \|(U-\alpha I)\underline{\mathbf{x}}\|= \|y\|^2 + \vert \alpha\vert ^2 \|x\|^2 + \|x-\alpha y\|^2 & \leq \| (U^{-1}- \bar{\alpha} I)\underline{\mathbf{x}} \| \\
  & = \|x\|^2 + \vert \alpha\vert ^2 \|y\|^2 + \|y- \bar{\alpha} x\|^2 ,
\end{align*}
which gives $\|x-\alpha y\|\leq \|y-\bar{\alpha}x\|$. Interchanging the roles of $x,y$, we have $\|y-\bar{\alpha} x\| \leq \|x - \alpha y\|$ which is same as saying that $\|y-\alpha x\| = \| x-\bar{\alpha}y \|$ for any $\alpha \in \mathbb D$. In particular $\|y-\xi x\| = \| x- \xi y \|$ for any $\xi \in (-1,1)$.
It follows from here that $\|ax-by\| =\|bx-ay\|,$ for any real numbers $a,b$ and any vectors $x,y \in \X$ with $\|x\|=\|y\|$. So, $\X$ is a Hilbert space by \cite[Corollary 2]{Ficken}.

\smallskip

The proof of the theorem is now complete.
\end{proof}

A natural question arises: for a (non-Hilbert) Banach space $\X$ what are the strict contractions in $\mathcal B(\X)$ that dilate to isometries? We now characterize such strict contractions as well as find explicit isometric dilations for them.

\begin{thm}\label{Necessary and sufficient condition for a contraction to be dilated 1}  Suppose $\X$ is a complex Banach space. Then a strict contraction $T$ on $\X$ dilates to an isometry if and only if the function $A_T: \mathbb{X}\to [0,\infty)$ given by $
  A_T(x) = \left( \| x\|^2 - \|Tx\|^2 \right)^{\frac{1}{2}}$
defines a norm on $\mathbb{X}.$ Moreover, the minimal isometric dilation space of $T$ is isometrically isomorphic to $\mathbb{X}\oplus_2 \ell_2(\mathbb{X}_0)$, where $\mathbb{X}_0$ is the Banach space $(\mathbb{X}, A_T)$.
\end{thm}

\begin{proof}
We prove the sufficiency first. Let $\mathcal{K}= \ell_2(\mathbb{X}\oplus_2\mathbb{X}_0)$, where $\X_0$ is the Banach space $(\X,A_T).$ Let $\mathbf{D},~\mathbf{T}: \mathbb{X}\oplus_2\mathbb{X}_0 \to \mathbb{X}\oplus_2\mathbb{X}_0$ be defined by the operator block matrices
\[
\begin{bmatrix}
0 & 0 \\
I & 0
\end{bmatrix}
\quad \text{ and }\quad
\begin{bmatrix}
T & 0 \\
0 & I
\end{bmatrix}\quad \text{respectively.}
\]
Evidently, we have
$
  \left\|\mathbf{T}\underline{x}\right\|^2 + \left\|\mathbf{D}\underline{x}\right\|^2 = \left\|\underline{x}\right\|^2$ for all $\underline x\in \mathbb{X}\oplus_2\mathbb{X}_0$.
Now, consider the linear map $V:\mathcal{K}\to\mathcal{K}$ defined by
\begin{equation} \label{eqn:new-061}
  V\left(\underline{x_0},\underline{x_1},\dots\right) = \left(\mathbf{T}\underline{x_0}, \mathbf{D}\underline{x_0}, \underline{x_1}, \underline{x_2},\dots\right),\qquad \left(\underline{x_0},\underline{x_1},\dots\right)\in \mathcal{K}.
\end{equation}
Then for each $\left(\underline{x_0},\underline{x_1},\dotsc\right)\in \mathcal{K},$ we have
\begin{align*}
 \left\| V(\underline{x_0},\underline{x_1},\dots) \right\|^2 & = \left\|\mathbf{T}\underline{x_0}\|^2 + \| \mathbf{D}\underline{x_0}\right\|^2 + \sum _{i=1}^\infty \left\| \underline{x_i} \right\|^2 = \sum_{n=0}^\infty \left\|\underline{x_i}\right\|^2.    
\end{align*}

Consequently, $V$ is an isometry. We now show that $V$ is a dilation of $T.$ To this end, we embed the space $\mathbb{X}$ into $\mathcal{K}$ by the natural inclusion map $W: \mathbb{X} \to \mathcal{K}$ defined by 
\[
 W(x)= ((x,\mathbf{0}),\mathbf{0},\mathbf{0}, \dots), \qquad x\in \mathbb{X}.
\]
Then
\[
  \mathcal{K}= W(\mathbb{X}) \oplus_2 \widetilde{\mathbb{Y}},\quad \text{where} \quad \widetilde{\mathbb{Y}}=\left\{(\mathbf{D}~\underline{x_0}, \underline{x_1},\dots ):~ (\underline{x_0},\underline{x_1},\dots)\in \mathcal{K}\right\}.
\]
Consider the surjective isometry $\widehat{W}:=W: \mathbb{X} \to W(\mathbb{X})$ and $\widehat{T}:= \widehat{W} T \widehat{W}^{-1} : W(\X)\to W(\X).$ It is not difficult to see that for $n\geq 1,$ we have
\[
V^n \left(\underline{x_0},\underline{x_1},\dots\right) = \left(\mathbf{T}^n\underline{x_0}, \mathbf{D}\mathbf{T}^{n-1}\underline{x_0}, \mathbf{D}\mathbf{T}^{n-2}\underline{x_0}, \dots, \mathbf{D}\underline{x_0}, \underline{x_1},\underline{x_2},\dots\right),\quad \left(\underline{x_0}, \underline{x_1},\dots\right)\in \mathcal{K}.
\]
Consequently, for each natural number $n$ and for any $x\in X$ we have
\begin{align*}
  P_{_{W(\mathbb{X})}}V^n (W x) 
   & = P_{_{W(\mathbb{X})}}\left(\mathbf{T}^n\underline{x}, \mathbf{D}~\mathbf{T}^{n-1}\underline{x}, \mathbf{D}~\mathbf{T}^{n-2}\underline{x}, \dots, \mathbf{D}\underline{x},\mathbf{0},\mathbf{0},\dots \right), \qquad \left[ ~\underline{x}= (x,\mathbf{0}) ~\right]\\
   & = ((T^nx,\mathbf{0}), \mathbf{0},\mathbf{0},\dots ) \\
   & = W (T^n x) = \widehat{T}^n Wx.
\end{align*}
Since $W(\X)$ is right-complemented in $\mathcal{K},$ we have that $\|P_{_{W(\X)}}\|=1.$ Therefore, $V$ on $\mathcal{K}$ is an isometric dilation of $T$. Moreover, the minimal isometric dilation space is
\[
  \overline{span}\{ V^n Wx : ~ n\geq 0,~ x\in \mathbb{X}\}.
\]
Indeed, for every $n\geq 1$ we have
 \[
 V^n Wx - V^{n-1}W(Tx) = (\mathbf{0}, \mathbf{0}, \cdots, \underbrace{~\mathbf{D}~\underline{x}~}_{n-th}, \mathbf{0}, \mathbf{0}, \cdots),\qquad \underline{x}=(x,\mathbf{0}),
\]
which further implies that
\[
  W(\mathbb{X})\oplus_2 \ell_2\left(\mathbf{D}\left( \mathbb{X}\oplus_2\mathbb{X}_0 \right)\right) \subseteq \overline{span}\{ V^n Wx : ~n\geq 0,~x\in \mathbb{X}\} \subseteq W(\mathbb{X})\oplus_2 \ell_2\left(\mathbf{D}\left( \mathbb{X}\oplus_2 \mathbb{X}_0 \right)\right).
\]
Therefore, the above inclusions are equalities. Since $W( \mathbb{X} )\oplus_2 \ell_2\left(\mathbf{D}\left( \mathbb{X}\oplus_2 \mathbb{X}_0 \right)\right)$ is isometrically isomorphic to $\mathbb{X}\oplus_2 \ell_2\left( \mathbb{X}_0 \right),$ the minimal isometric dilation space is isometrically isomorphic to $\mathbb{X}\oplus_2 \ell_2(\mathbb{X}_0)$.

\medskip

The proof of the `necessary'-part can be imitated from the proof of (iii)$\implies$(iv) of Theorem \ref{Characterisation of Hilbert spaces in terms of Dilation on Banach space}. The proof is now complete.
\end{proof}

\begin{cor} \label{cor:main-02}
Let $T\in \mathcal{B}(\X)$ be a strict contraction such that $A_T$ defines a norm on $\X$. Then $A_{T^n}$ also defines a norm on $\X$ for all $n\geq 2$.
\end{cor}

\begin{proof}

It suffices to show that $A_{T^n}$ satisfies the triangle inequality, since positivity and absolute homogeneity of $A_{T^n}$ follow from the fact that $T^n$ is also a strict contraction if $T$ is so. Let $V$ be the isometric dilation of $T$ as constructed in the sufficiency part of Theorem \ref{Necessary and sufficient condition for a contraction to be dilated 1}. Then for every $n\geq 2$, we have 
\[
   \|x\|^2 - \|T^nx\|^2 = \|V^n Wx\|^2 - \| P_{_{W(\mathbb{X})}} V^n W x\|^2 = \|(I-P_{_{W(\mathbb{X})}})V^n W x\|^2, \qquad x\in \mathbb{X}.
\]
Therefore, 
\[
  A_{T^n}(x) = \|(I-P_{_{W(\mathbb{X})}})V^n W x\|, \qquad x\in \mathbb{X},
\]
which shows $A_{T^n}(x+y)\leq A_{T^n}(x) + A_{T^n}(y)$ for all $x,~y\in \X$, as desired.
\end{proof}

The next result gives another necessary and sufficient condition such that the function $A_T$ (as in (\ref{Equivalent norm function})) gives a norm on a Banach space.

\begin{prop} \label{cor:main-01}
Suppose $(\mathbb{X}, \|.\|)$ is a Banach space and suppose $T$ is a strict contraction on $\mathbb{X}$. Then the function $A_T:\mathbb{X}\to[0,\infty)$ as in $(\ref{Equivalent norm function})$ is a norm on $\mathbb{X}$ if and only if there exists a Banach space $\mathbb{Y}$ and a bounded linear operator $A\in \mathcal{B}\left(\mathbb{X},\mathbb{Y}\right)$ such that $\|A(x)\|=A_T(x)$ for all $x\in\mathbb{X}.$
\end{prop}

\begin{proof}

First we prove the sufficiency part. The positivity and homogeneity of $A_T$ follow from the facts that $\|T\| < 1$ and $T$ is linear. The triangle inequality follows from linearity of the map $A: \mathbb{X} \to \mathbb{Y}$. Indeed, for $x,y\in \X$ we have 
\[
  A_T(x+y) = \|A(x+y)\| = \| Ax + Ay\| \leq \|Ax \| + \| Ay\| = A_T(x) + A_T(y).
\]

For the necessity part, we have that $T$ dilates to an isometry by Theorem \ref{Necessary and sufficient condition for a contraction to be dilated 1}. Consider the isometric dilation $V: \mathcal{K} \to \mathcal{K}$ of $T$ as constructed in Theorem $\ref{Necessary and sufficient condition for a contraction to be dilated 1}$, and the linear map $A: \X \to \mathcal{K}$ defined by $A(x) = (I - P_{_{W(\X)}})VW(x)$ for all $x\in \X$, where $W: \X \to \mathcal{K}$ is the isometric embedding. Then $\|A\| =1$ and it follows from proof of Corollary \ref{cor:main-02} that
\[
   A_T(x) = \left(\|x\|^2 -\|Tx\|^2 \right)^{\frac{1}{2}} = \|(I - P_{_{W(\X)}})VW(x) \| = \| A(x)\|, \qquad x\in \X.
\]
This completes the proof.
\end{proof}

So, we learn from Proposition \ref{cor:main-01} that whenever the function $A_T$ gives a norm on a Banach space $\X$, there is actually an operator $A$ such that $\|A(x)\|=A_T(x)$ for all $x\in \X$. Our next result shows that the existence of a bounded linear map $A: \X \rightarrow \ell_2(X)$ satisfying $\|A(x)\|=\left( \|x\|^2 - \|Tx\|^2 \right)^{\frac{1}{2}}$ is necessary and sufficient for the function $A_T$ to define a norm, which is to say that it is also necessary and sufficient for a strict contraction $T$ to admit an isometric dilation. Moreover, we construct such an isometric dilation which is different from the one provided in Theorem \ref{Necessary and sufficient condition for a contraction to be dilated 1}.

\begin{thm}\label{Classification I}
Let $\X$ be Banach space. Then a strict contraction $T$ on $\X$ dilates to an isometry $V$ on $\ell_2(\mathbb{X})$ if and only if there is a bounded linear map $A: \mathbb{X}\to \ell_2(\mathbb{X})$ such that $\|A(x)\| =\left( \|x\|^2 - \|Tx\|^2 \right)^{\frac{1}{2}}$, for each $x\in \mathbb{X}.$
\end{thm}

\begin{proof}
We prove the sufficiency first. For each $n\geq 1$, consider the linear operator $A_n:\X\to \X$ defined by 
\[
    A_n(x)= \pi_n(Ax),\qquad x\in \mathbb{X},
\]
where $\pi_n$ is the projection of $\ell_2(\mathbb{X})=\X \oplus_2 \X \oplus_2 \X\oplus_2\dots$ onto its $n$-th component.
Thus,
\[
    \|Ax\|^2 = \sum_{n=1}^\infty \|A_nx\|^2, \qquad x\in \mathbb{X}.
\]
Let $W: \mathbb{X}\to \ell_2(\mathbb{X})$ be defined by
\[
W(x)= (x,\mathbf{0},\mathbf{0}, \cdots),\qquad x\in \mathbb{X}.
\]
Then $\text{Ran}(W)$ is a closed subspace of $\ell_2(\mathbb{X})$ and $\ell_2(\X)$ can be expressed as $\ell_2(\mathbb{X})= W(\mathbb{X}) \oplus_2 \widetilde{\mathbb{Y}}$, where $\widetilde{\mathbb{Y}}= \{(\mathbf{0},x_1, x_2, \cdots): ~ x_n\in \mathbb{X},~ n\geq 1\}$. Let $V: \ell_2(\mathbb{X}) \to \ell_2(\mathbb{X})$ be the operator defined by the operator block matrix
\[
V = \begin{bmatrix}
     T & 0 & 0 & 0 & 0 & \cdots \\
     A_0 & 0 & 0 & 0 & 0 & \cdots\\
     0 & I & 0 & 0 & 0 & \cdots\\
     A_1 & 0 & 0 & 0 & 0 & \cdots\\
     0 & 0 & I & 0 & 0 & \cdots\\
     \vdots & \vdots & \vdots & \vdots & \vdots & \ddots.
    \end{bmatrix}.
\]
For any $(x_0,x_1,\dots)\in \ell_2(\X)$, we have that
\begin{align*}
    \|V(x_0,x_1,\dots)\|^2  & = \|(Tx_0, A_0x_0, x_1, A_1x_0, x_2, A_2x_0, \dots)\|^2\\
    & = \|Tx_0\|^2 + \sum_{n=0}^\infty \|A_nx_0\|^2 + \sum_{n=1}^\infty \|x_n\|^2\\
    & = \|Tx_0\|^2 + \|Ax_0\|^2 + \sum_{n=1}^\infty \|x_n\|^2.
\end{align*}
It follows from the hypothesis that $\|x_0\|^2=\|Tx_0\|^2 + \|Ax_0\|^2$. Therefore,
\[
 \|V(x_0,x_1,\dots)\|^2 = \sum_{n=0}^\infty \|x_n\|^2 = \|x\|^2
\]
and consequently $V$ is an isometry. The block matrix representation of $V$ is lower triangular. Now, the product $C_{ij}$ of two lower triangular block matrices, say $A_{ij}$ and $B_{ij}$ is again a lower triangular block matrix with $C_{ii}= A_{ii}B_{ii}.$ Thus, the $(1,1)$ entry of $V^n$ is $T^n$ which is to say that
\[
   P_{_{W(\X)}}V^n W(x)= W(T^nx), \quad x\in \mathbb{X}, \quad n \geq 1,
\]
where $P_{_{W(\mathbb{X})}}$ is projection of $\ell_2(\mathbb{X})$ onto $W(\mathbb{X}).$ Evidently, $W(\X)$ is right-complemented and therefore, $\|P_{_{W(\X)}}\|=1$. Consequently, $V$ is an isometric dilation of $T$.

\medskip

To prove the necessity, suppose there is an isometric embedding $W: \mathbb{X}\to \ell_2(\mathbb{X})$ such that $\ell_2(\mathbb{X})= W(\mathbb{X}) \oplus_2 \widetilde{\mathbb{Y}}$ for some closed subspace $\widetilde{\mathbb{Y}}\subseteq \ell_2(\mathbb{X})$ and that $V$ on $\ell_2(\mathbb{X})$ is an isometric dilation of $T$ with respect to the isometric embedding $W$. Let us consider the linear operator $A: \mathbb{X}\to \ell_2(\mathbb{X})$ defined by 
\[
   A(x) = (I-P_{_{W(\mathbb{X})}}) VW(x),\qquad x\in \mathbb{X}.
\]
Then for each $x\in \X$, we have
$
 VWx = P_{_{W(\X)}} VWx + (I-P_{_{W(\X)}})VWx.
$
Since $(I-P_{_{W(\X)}})VWx\in \widetilde{\mathbb{Y}}$ for every $x\in \X$, we have that
\[
 \|Ax\| = (\|VWx\|^2 - \|P_{_{W(\X)}} VWx\|^2)^\frac{1}{2} = (\|x\|^2 - \|Tx\|^2)^\frac{1}{2}
\]
and the proof is complete.
\end{proof}

For a strict contraction $T$ on a Hilbert space $\mathcal{H}$, the function $A_T$ as in (\ref{Equivalent norm function}) is always a norm, since for each $x\in \mathcal{H}$, $A_T(x) = \|D_T(x)\|$, where $D_T=(I-T^*T)^{\frac{1}{2}}$, the defect operator of $T$. Also, in that case the minimal isometric dilation space is given by 
\[
\ell_2(\mathcal{H})= \mathcal{H}\oplus_2 \ell_2(\mathcal{D}_T), \quad \text{ where } \quad \mathcal{D}_T=\overline{\text{Ran}} \ D_T.
\]
Therefore, Theorem \ref{Necessary and sufficient condition for a contraction to be dilated 1} is a generalization in Banach space setting of the isometric dilation theorem of a Hilbert space contraction due to Sz. Nagy and Foias which is stated below.

\begin{thm}\cite[Chapter I, Theorem 4.2]{Nagy Foias}
For every contraction $T$ on a Hilbert space $\mathcal{H}$ there exists an isometric dilation $V$ on some Hilbert space $\mathcal{K}\supseteq {\mathcal{H}}$, which is moreover minimal in the sense that $\mathcal{K} = \overline{span}\{V^nh: n\geq 0, ~ h\in \mathcal{H}\}$. This minimal isometric dilation of $T$ is unique up to isomorphism.
\end{thm}

We now construct an explicit example of a contraction $T$ on a Banach space $\X$, for which $A_T$ as in (\ref{Equivalent norm function}) defines a norm on $\X$.

\begin{eg}
Let $\X$ be a Banach space. Consider $T_\lambda: \ell_2(\X) \to \ell_2(\X)$ defined by $T_\lambda(x_n) = (\lambda_n x_n)$, where $\lambda_n = \frac{n}{2(n+2)}$ for all $n\in \mathbb{N}$. Then $T$ is a bounded linear map and $\|T\| = \displaystyle \sup_{n\in \mathbb{N}} \vert \lambda_n \vert = 1\slash 2< 1$. So, the function $A_{T_\lambda}$ satisfies the positivity and absolute homogeneity. Let $\mu_n =  (1-\lambda_n^2)^{1\slash 2}$ for all $n\in \mathbb{N}$. Then for every $(x_n)\in \ell_2(\X)$, we have 
\[
A_{T_\lambda}(x_n) = \left( \sum_{n=1}^\infty \|x_n\|^2 - \sum_{n=1}^\infty \|\lambda_n x_n\|^2 \right)^{1\slash 2} = \|T_\mu (x_n)\|,
\]
where $T_\mu: \ell_2(\X) \to \ell_2(\X)$ is defined by $T_\mu(x_n) = (\mu_n x_n)$. 
Consequently, the triangle inequality for $A_{T_\lambda}$ follows from the linearity of $T_\mu$. \qed
\end{eg}

Next we construct a contraction $T$ on $\mathcal{L}_p$ ($1\leq p\leq \infty$, $p\neq 2$) for which $A_T$ (as in (\ref{Equivalent norm function})) is not a norm on $\mathcal{L}_p$. We need the following lemma before that.

\begin{lem}\label{Inequality for p greater than or equal 1}
(i). $\sqrt{\frac{4-4^{\frac{1}{p}}}{4}} < 2^{\frac{1}{p}-1}$ for $p\in (1,2).$ (ii). $ \sqrt{1-4^{-\frac{1}{p}}} < 2^{-\frac{1}{p}}$ for $p\in (2,\infty).$
\end{lem}

\begin{proof}
 (i). When $1< p < 2,$ we have $4- 4^{\frac{1}{p}} < 4^{\frac{1}{p}}$ and therefore, $\frac{4-4^{\frac{1}{p}}}{4} < 4^{\frac{1}{p}-1} = 2^{\frac{2}{p}-2}.$

\medskip
 
 (ii). When $2< p <\infty,$ we have $1-4^{-\frac{1}{p}} < 4^{-\frac{1}{p}}$ and therefore, $\sqrt{1-4^{-\frac{1}{p}}} < 2^{-\frac{1}{p}}.$
\end{proof}

\smallskip

\begin{eg} \label{exmp:main-01}

Let $\X = \mathcal{L}_p(\Omega, M, \mu)$ ($1\leq p\leq \infty$, $p\neq 2$) for some measure space ($\Omega, M, \mu$). Let $A,~ B\in M $ with $A \cap B =\emptyset$ and $\mu(A) \neq 0 , ~ \mu(B) \neq 0.$ Consider the unit vectors $f= \frac{\chi_A}{\mu(A)^\frac{1}{p}}$ and $g= \frac{\chi_B}{\mu(B)^\frac{1}{p}}.$ A direct computation shows that for closed subspaces $\mathbb{Y}= \overline{span} \{ f,~g \}$ and $\widetilde{\mathbb{Y}}=\left\{ f\chi_{A^c \cap B^c} : ~ f\in \mathbb{X}\right\}$ of $\X$, we have $\mathbb{X}=\mathbb{Y} \oplus_\SR \widetilde{\mathbb{Y}}.$

\medskip

\noindent Case I: Let $p\in [1,2)$. Choose a scalar $\lambda$ such that $\lambda\in (0,1)$ if $p=1$, and $\sqrt{\frac{4-4^{\frac{1}{p}}}{4}} < \lambda < 2^{\frac{1}{p}-1}$ if $1<p<2$. This choice of $\lambda$ is guaranteed by Lemma \ref{Inequality for p greater than or equal 1}. Every $\phi\in \X$ can be uniquely expressed as $\phi =\alpha_\phi f + \beta_\phi g + h_\phi$ for some $\alpha_\phi, \beta_\phi \in \mathbb{C}$ and $h_\phi\in \widetilde{\mathbb{Y}}$. Let $T: \mathbb{X} \to \mathbb{X}$ be defined by
\[
  T(\phi) = \lambda \left(\alpha_\phi - \beta_\phi \right)f, \qquad \phi\in\X.
\]
We show that $T$ is a strict contraction. For any $\phi\in S_\X,$ we have
\[
  \| \alpha_\phi f + \beta_\phi g \|_p = \left( \vert \alpha_\phi \vert ^p + \vert \beta_\phi \vert ^p \right)^{\frac{1}{p}} \leq \|\phi\|_p=1,
\]
by virtue of $\Y\perp_B \widetilde{\mathbb{Y}}$. Consequently, we have
\[
 \| T\phi \|_p = \lambda  ~ \vert \alpha_\phi - \beta_\phi \vert \leq \lambda \left( \vert \alpha_\phi \vert + \vert \beta_\phi \vert \right) \leq \lambda \left(\vert \alpha_\phi \vert ^p + \vert \beta_\phi \vert ^p \right)^{\frac{1}{p}} 2^{1-\frac{1}{p}} \leq \lambda . 2^{1-\frac{1}{p}}.\|\phi\|_p.
\]
Therefore, $\|T\| \leq \lambda . 2^{1-\frac{1}{p}} <1$ by an appropriate choice of $\lambda$. Observe that
\[
  \| f+g \|_p = 2^{\frac{1}{p}}, ~ Tf = \lambda f = -Tg, ~ A_T(f) = \left( 1 - \vert \lambda \vert^2 \right)^{\frac{1}{2}} = A_T(g).
\]
A straightforward computation now shows that the inequality (\ref{Triangle inequality}) is violated for this particular choice of $T$, $f$ and $g$. Consequently, the function $A_T$ as in (\ref{Equivalent norm function}) is not a norm on $\X.$

\medskip

\noindent Case II: Let $p\in (2,\infty].$ By $p=\infty$, we mean $\frac{1}{p}$ to be $0$. Choose a scalar $\lambda$ such that $\sqrt{1-2^{-\frac{2}{p}}} < \lambda < 2^{-\frac{1}{p}}$ if $2<p<\infty$, and $0 < \lambda < 1$ if $p=\infty.$  This choice of $\lambda$ is guaranteed by Lemma \ref{Inequality for p greater than or equal 1}. Observe that $\Y=span\{\chi_A, \chi_B\}.$ Thus, every $\phi\in \X$ can be uniquely expressed as $\phi =\alpha_\phi \chi_A + \beta_\phi \chi_B + h_\phi$ for some $\alpha_\phi, \beta_\phi \in \mathbb{C}$ and $h_\phi \in \widetilde{\Y}$. Let $S: \mathbb{X} \to \mathbb{X}$ be defined by
\[
  S(\phi) = \lambda \alpha_\phi \left(\chi_A - \left(\frac{\mu(A)}{\mu(B)}\right)^{\frac{1}{p}}\cdot \chi_B \right),\qquad \phi\in \mathbb{X}.
\]
We show that $S$ is a strict contraction. For any $\phi\in S_\X,$ we have
\[
  \| \alpha_\phi \chi_A + \beta_\phi \chi_B \|_p =
    \begin{cases}
       \left( \vert \alpha_\phi \vert ^p \mu(A) + \vert \beta_\phi \vert ^p\mu(B) \right)^{\frac{1}{p}} & \text{ if } 2 < p < \infty \\
       \max\{ \vert \alpha_\phi \vert, \vert \beta_\phi \vert \} & \text{ if } p =\infty
    \end{cases}
\]
and in either case we have $\vert \alpha_\phi \vert\cdot \mu(A)^{\frac{1}{p}} \leq \|\alpha_\phi \chi _A + \beta_\phi \chi_B \|_p \leq \|\phi\|_p =1$, since $\Y\perp_B \widetilde{\mathbb{Y}}$. Consequently, we have
\begin{align*}
 \| S\phi \|_p & = \lambda\cdot \vert \alpha_\phi \vert \cdot \left\|\chi_A - \left(\frac{\mu(A)}{\mu(B)}\right)^{\frac{1}{p}}\cdot\chi_B \right\|_p\\
               & = \lambda \cdot\vert \alpha_\phi \vert\cdot \left(\mu(A) + \frac{\mu(A)}{\mu(B)}\cdot \mu(B)\right)^{\frac{1}{p}} \\
               & = \lambda \cdot\vert \alpha_\phi \vert\cdot \left(\mu(A) + \mu(A)\right)^{\frac{1}{p}} \\
               & = \lambda \cdot 2^{\frac{1}{p}}\cdot \vert \alpha_\phi \vert\cdot \mu(A)^{\frac{1}{p}} \leq \lambda . 2^{\frac{1}{p}}.\|\phi\|_p. \quad \left[~ \vert \alpha_\phi \vert\cdot \mu(A)^{\frac{1}{p}}\leq \|\phi\|_{p} =1 ~\right]
\end{align*}
Therefore, $\|S\| \leq \lambda . 2^{\frac{1}{p}} < 1$ by an appropriate choice of $\lambda$. Now, consider the unit vectors 
\[
u= \frac{\chi_A}{\left(2\cdot\mu(A)\right)^\frac{1}{p}} + \frac{\chi_B}{\left(2\cdot\mu(B)\right)^\frac{1}{p}} \, ;\quad v= \frac{-\chi_A}{\left(2\cdot\mu(A)\right)^\frac{1}{p}} + \frac{\chi_B}{\left(2\cdot\mu(B)\right)^\frac{1}{p}} \ .
\] 
Observe that
\[
 u+v = \frac{2\cdot\chi_B}{\left(2\cdot \mu(B)\right)^{\frac{1}{p}}}, \quad Su = \frac{\lambda}{\left(2\cdot \mu(A)\right)^{\frac{1}{p}}} \left[\chi_A -\left(\frac{\mu(A)}{\mu(B)}\right)^{\frac{1}{p}}\cdot\chi_B\right] = -S(v),
\]
and
\[
\|u+v\|_p = \frac{2}{2^{\frac{1}{p}}}, \quad S(u+v) = 0, \quad \|Su\|_p =\lambda =\|Sv\|_p, \quad \text{ and }\quad A_S(u) = \left( 1 - \vert \lambda \vert^2 \right)^{\frac{1}{2}} = A_S(v).
\]
A straightforward computation shows that the inequality (\ref{Triangle inequality}) is violated for this particular choice of $S$, $u$ and $v$. Consequently, the function $A_S$ as in (\ref{Equivalent norm function}) is not a norm on $\X.$  \qed

\end{eg}

Example \ref{exmp:main-01} along with Theorem \ref{Norm if and only if Hilbert space} proves the following well-known result.

\begin{thm} \label{thm:new-003A}
For $1\leq p\leq \infty$, $\mathcal{L}_p$ is a Hilbert space if and only if $p=2$.
\end{thm}


\subsection{Induced semi-norm and dilation in Banach spaces} \label{Subsection:new-6.1}
Let $\X$ be a Banach space and $T \in \mathcal{B}(\X)$. Then $\|T\| \leq 1$ is equivalent to the positivity of $A_T$ as in (\ref{Equivalent norm function}), i.e., $A_T(x) \geq 0$ for all $x \in \X$. Also, the strict positivity of $A_T$, i.e., $A_T(x)>0$ for all nonzero $x$ in $\X$ is equivalent to $\|Tx\| < \|x\|$ for all nonzero $x \in \X$. In this Subsection, we study when $A_T$ defines a norm or a semi-norm on $\X$ for operators $T$ with $\|T\|=1$.

\smallskip

 The set of all contractions on $\X$ is a disjoint union of the open unit ball $\mathcal{B}_1(\X)$ and the unit sphere $S_{\mathcal{B}(\X)}$ in $\mathcal{B}(\X)$. All operators \(T \in \mathcal{B}_1(\X)\) for which the function \(A_T\) defines a norm on \(\X\) have been characterized in Theorem \ref{Necessary and sufficient condition for a contraction to be dilated 1}. The operators $T \in S_{\mathcal{B}(\X)}$ can be further split into two disjoint classes based on their norm attainment sets $M_T = \{x \in S_{\X} : \|Tx\| = \|T\|\}$, namely
 \[
 \mathcal{G}_1(\X) = \{T \in S_{\mathcal{B}(\X)} : M_T = \emptyset\} \ \ \text{ and } \ \ \mathcal{G}_2(\X) = \{T \in S_{\mathcal{B}(\X)} : M_T \neq \emptyset\}.
 \]
The homogeneity property of $A_T$ is automatic for all contractions $T$. The strict positivity of $A_T$ also holds for each $T \in \mathcal{G}_1(\X)$, whereas the same is not true for all $T \in \mathcal{G}_2(\X)$. For $T \in \mathcal{G}_2(\X)$ it is naturally asked if the function $A_T$ defines a semi-norm on $\X$. Recall that a semi-norm $N$ on a vector space $\mathbb V$ is a function $N: \mathbb V \to [0,\infty)$ satisfying $N(\alpha v) = |\alpha| N(v)$ and $N(v + w) \leq N(v) + N(w)$ for all $\alpha \in \mathbb{C}$ and $u, w \in \mathbb V$. The following four examples show that nothing is conclusive about the function $A_T$, i.e., if it defines a norm or semi-norm for $T$ in $\mathcal{G}_1(\X)$ or $\mathcal{G}_2(\X)$.

\begin{eg}[$A_T$ does not define a norm on $\X$ for $T\in \mathcal{G}_1(\X)$]
Consider the Banach space $\X= \ell_1= \{(x_n): \sum_{n=1}^\infty \vert x_n\vert < \infty, ~~ x_n\in \mathbb{C}, ~~ n\geq 1 \}$ and the sequence $(\lambda_n)\in \ell_\infty$, where $\lambda_n = \frac{2n-1}{2n}$ for all $n\in \mathbb{N}$. Let $T: \X \to \X$ be defined by 
\[
   T((x_n))= (\lambda_1(x_1-x_2), \lambda_2 x_3, \lambda_3 x_4, \dotsc,), \qquad (x_n)\in \X.
\]
Then it is easy to see that ${ \displaystyle \|T\| = \sup_{n}\lambda_n = 1 }$, and $\|T((x_n))\| < \|(x_n)\|$ for all $(x_n)\in \X$. So $T\in \mathcal{G}_1(\X)$. However, the function $A_T$ does not define a norm on $\X$, since it does not satisfy the triangle inequality for $e_1=(1,0,0,\dotsc)$ and $e_2=(0,1,0,0,\dotsc)$.  \qed

\end{eg}

\begin{eg}[$A_T$ defines a norm on $\X$ for $T\in \mathcal{G}_1(\X)$]
Consider the classical Hilbert space $\X= \ell_2$, where
\[
\ell_2= \{(x_n): \sum_{n=1}^\infty \vert x_n\vert^2 < \infty, ~~ x_n\in \mathbb{C}, ~~ n\geq 1 \},
\]
and the multiplication operator $T: \ell_2 \to \ell_2$ defined by
\[
   T((x_n)) = (\lambda_n x_n), \qquad (x_n)\in \ell_2, ~~ \lambda_n = \frac{n-1}{n},~~ n\geq 1.
\]
Then it is easy to see that $T\in \mathcal{G}_1(\X)$, i.e., $\|T\| =1$ and $\|T(x_n)\| < \|(x_n)\|$ for all nonzero $(x_n)\in \ell_2$. Also, note that $A_T((x_n)) = \|D_T((x_n))\|$ for all $(x_n)\in \ell_2$. Consequently, the linearity of $D_T$ shows that $A_T$ defines a norm on $\X$.  \qed

\end{eg}

\begin{eg}[$A_T$ does not define a semi-norm on $\X$ for $T\in \mathcal{G}_2(\X)$]
Let $\X = (\mathbb{C}^3, \|\cdot\|_1)$, and choose $0< \lambda < 1$. Consider the operator $T: \X \to \X$ defined by $T(x,y,z) = (x, \lambda(y-z), 0)$. Then $\|T\| =1$ and $M_T= span\{(1,0,0)\} \cap S_{\X}$, i.e., $T\in \mathcal{G}_2(\X)$. However, the function $A_T$ does not define a semi-norm on $\X$, since it does not satisfy the triangle inequality for the vectors $(0,1,0)$ and $(0,0,1)$. \qed

\end{eg}

\begin{eg}[$A_T$ defines a semi-norm on $\X$ for $T\in \mathcal{G}_2(\X)$]
Let $\X$ be a Banach space. Consider the backward shift operator $\widehat M_{z}: \ell_2(\X) \to \ell_2(\X)$ (as in (\ref{backward shift})) defined by
\[
   \widehat M_{z}(x_0, x_1, \dotsc) = (x_1, x_2, \dotsc), \qquad (x_0, x_1,\dotsc) \in \ell_2(\X).
\]
Then, $\|\widehat M_{z}\| =1$ and the norm attainment set is given by 
\[
{\displaystyle M_{\widehat M_{z}} =\{(x_0, x_1, \dotsc)\in \ell_2(\X): x_0 =\mathbf{0}, \ \ \sum_{n=1}^{\infty} \ \|x_n\|^2 =1\} }.
\]
Therefore, $\widehat M_{z}\in \mathcal{G}_2(\ell_2(\X))$. Let $P: \ell_2(\X) \to \X$ be the norm-one projection onto the first coordinate of $\ell_2(\X)$. Then we have
\[
  \|x\|^2 - \|A_{\widehat M_{z}} x\|^2 = \|x_0\|^2 = \|P(x)\|^2, \qquad x=(x_0,x_1,x_2 \dots)\in \ell_2(\X).
\]
So the function $A_{\widehat M_{z}}$ defines a semi-norm on $\ell_2(\X)$, as the triangle inequality follows from  the linearity of $P$. \qed

\end{eg}

Now, we characterize all the operators $T\in \mathcal{G}_1(\X)$ ($\mathcal{G}_2(\X)$) in terms of isometric dilation on a Banach (semi-normed linear) space.
\begin{thm} \label{thm:new-6.26}
Let $\X$ be a Banach space and let $T\in \mathcal{G}_1(\X)$  $($or, $T \in \mathcal{G}_2(\X))$ be arbitrary. Then the function $A_T$ $($as in $($\ref{Equivalent norm function}$) )$ defines a norm $($or, a semi-norm$)$ on $\X$ if and only if $T$ dilates to an isometry on a normed space $($or, a semi-normed linear space$)$.
\end{thm}

\begin{proof}

First we prove the necessary part. It follows from Lemma \ref{lem:new-0122} that the space $\X_0 = (\X, A_T)$ is not a Banach space. Now, consider the normed space (or, the semi-normed linear space) $\X_0 = (\X, A_T)$ and let $\mathcal{K} = \ell_2(\X\oplus_2 \X_0)$. Then the operator $V:\mathcal{K} \to \mathcal{K}$ as in the proof of Theorem \ref{Necessary and sufficient condition for a contraction to be dilated 1} remains an isometry and dilates $T$.

\smallskip

The proof of sufficiency follows in a similar manner, utilizing the linearity of the operator $\left(I - P_{_{W(\X)}} \right)VW: \X \to \widetilde{\X}$, where $V: \widetilde{\X} \to \widetilde{\X}$ is the isometric dilation and $W: \X \to \widetilde{\X}$ is the isometric embedding of $\X$ into the normed space (or, semi-normed linear space) $\widetilde{\X}$ (see the proof of Theorem \ref{Necessary and sufficient condition for a contraction to be dilated 1}). The proof is complete.
\end{proof}

\subsection{A new adjoint for Banach space operators and characterizations for Hilbert spaces} \label{Subsection:new-6.2} Unlike Hilbert space operators, it is not possible to define an adjoint for a Banach space operator that acts like Hilbert space adjoint. In this Subsection, we make an attempt of defining an adjoint of a Banach space operator $T$ that is not exactly the Banach adjoint $T^{\times}$. However, we shall see that our notion of adjoint for a Banach space operator generalizes that of a Hilbert space operator and leads to a necessary and sufficient condition such that a Banach space with dimension greater than $2$ becomes a Hilbert space.

\smallskip

For an operator $T$ on a Banach space $\mathbb{X}$, the Banach adjoint $T^\times$ is an operator defined on the dual space $\mathbb{X}^*$ by $f\mapsto f\circ T$ for every $f\in \mathbb{X}^*.$ If $T$ acts on a Hilbert space $\mathcal{H},$ then the Hilbert space adjoint $T^*$ of the operator $T$ is equal to the operator $\mathbb{J}_\X^{-1}T^{\times}\mathbb{J}_\X$, where $\mathbb{J}_\X:\mathcal{H}\to \HS (=\mathcal{H}^*)$ mapping $x$ to $\mathbb{J}_\X(x)$ is defined by $\mathbb{J}_\X(x)(y)=\langle y,x\rangle,~ y\in \mathcal{H}$. See Section \ref{sec:02} for the definition of $\mathbb{J}_\X$ in more general Banach space setting. Needless to mention that $\mathbb{J}_\X$ is a conjugate surjective linear isometry. The fact that the Hilbert space adjoint $T^*$ coincides with $\mathbb{J}_\X^{-1}T^{\times}\mathbb{J}_\X$ can be proved easily. Indeed, if we denote by $\widetilde T$ the operator $\mathbb{J}_\X^{-1}T^{\times}\mathbb{J}_\X$, then for $x,y \in \HS$ it follows that
\[
 T^\times \mathbb{J}_\X(x)(y)=\langle y, z \rangle, \quad z =\widetilde{T}(x).
\]
This further implies that
\[
  \langle Ty, x \rangle = T^\times \mathbb{J}_\X(x)(y)=\langle y, z \rangle = \langle y, \widetilde{T}(x)\rangle, 
\]
and consequently, $T^*= \widetilde{T}.$ Thus, we have a way of defining an adjoint for an operator $T$ on a Banach space $\X$ in terms of $\mathbb{J}_\X^{-1}T^{\times}\mathbb{J}_\X$ which coincides with the Hilbert space adjoint $T^*$ when $\X$ is a Hilbert space. For a Banach space $\X$, the map $\mathbb{J}_\X:\mathbb{X} \to \mathbb{X}^*$ can be canonically defined by $x\mapsto f_x$, where $f_x$ denotes support functional at the point $x\in \X$. However, there are Banach spaces where one can find multiple support functionals at a point. Also, a linear functional on a Banach space may not attain its norm at a point or, even if it attains its norm at a point, the point may not be unique. To get rid of such ambiguities and to define $\mathbb J_{\X}$ appropriately, we assume that the underlying Banach space is reflexive, smooth and strictly convex. Now the map $\mathbb{J}_\X:\mathbb{X} \to \mathbb{X}^*$ given by $x\mapsto f_x$ is well-defined and has the following properties:
\begin{enumerate}
\item[(i)] $\mathbb{J}_\X(\alpha x) = \bar{\alpha}\Phi(x), ~\alpha\in \mathbb{C},~ x\in \mathbb{X}$ ;

\smallskip

\item[(ii)]if $\dim(\X) >2$, then $\mathbb{X}$ is a Hilbert space if and only if $\mathbb{J}_\X$ is a linear map, because, the linearity of $\mathbb{J}_\X$ is equivalent to the left-additivity of Birkhoff-James orthogonality. (See \cite[Theorem 2]{James 2} for details.)
\end{enumerate}
Even without linearity in general, the map $\mathbb{J}_\X$ is a surjective isometry. Thus, for an operator $T$ on a reflexive, smooth and strictly convex Banach space $\X$, we define an adjoint $T_*$ of $T$ in the following way:
\begin{align*}
T_* \,  : & \, \, \mathbb{X}\to\mathbb{X} \\
 & x \mapsto \mathbb{J}_\X^{-1} T^\times \mathbb{J}_\X(x).
\end{align*} 
Note that unlike the Banach adjoint $T^{\times}$, the map $T_*$ acts on $\X$ itself. The adjoint $T_*$ has the following basic properties.
\begin{enumerate}
    \item $T^\times f_x = f_{T_*x}$ for every $x\in \mathbb{X}$.
    \item $\| T_*(x) \| \leq \| T\| ~ \|x\|$ for every $x\in \mathbb{X}$.
    \item $T_*(\alpha x)=\alpha T_*x$ for every $x\in\mathbb{X}$ and $\alpha\in \mathbb{C}.$
    \item $T_*$ is not linear in general.
    \item $f_x(TT_*x)= \| T_*x\|^2$ and $f_{T_*Tx}(x)= \| Tx \|^2$ for all ~$x\in \mathbb{X}.$ Moreover, $\displaystyle \sup_{\|x\|=1} \|T_*T x\| = \|T\|^2.$
    \item $(\alpha T)_* = \overline{\alpha}T_*$ for all $\alpha \in \mathbb{C}.$
    \item $(T_1T_2)_*= T_{2*}T_{1*}$ for every $T_1,T_2\in \mathcal{B}(\mathbb{X}).$
    \item  $(T_1 + T_2)_*$ is not equal to $T_{1*}+ T_{2*}$ in general.
\end{enumerate}

\smallskip

We now investigate the behaviour of the map $T_*$ in particular cases to have some more interesting features. We begin with the unilateral shift and shall see that its adjoint in our sense is the backward shift operator.

\begin{prop}\label{Adjoint of unilateral shift in Banach space setup}
Suppose $\mathbb{X}$ is a reflexive, smooth and strictly convex Banach space and $M_z:\ell_2(\mathbb{X})\to\ell_2(\mathbb{X})$ is the unilateral shift operator as in $(\ref{Forward shift})$. Then $M_{z*}= \widehat M_z ,$ where $\widehat M_z$ is the backward shift operator as in $(\ref{backward shift})$.
\end{prop}

\begin{proof}
We have that $\ell_2(\mathbb{X})^* \cong \ell_2(\mathbb{X}^*)$ by the map $R: \ell_2(\mathbb{X})^* \to \ell_2(\mathbb{X}^*)$ defined by $R(f)=(f_n)$, where $f_n(x)= f(\mathbf{0},\mathbf{0},\cdots, \underbrace{x}_{\text{$n$-th}},\mathbf{0},\mathbf{0},\cdots),~x\in \mathbb{X}.$ Therefore, for every $\underline{x}= (x_n)\in \ell_2(\mathbb{X})$ we have
\[
f_{\underline{x}}= R^{-1}((f_{x_n})) \, \, \text{ and } \,\, \left(M_z^\times f_{\underline{x}}\right)\underline{y} = f_{\underline{x}}(M_z(\underline{y})) = f_{x_2}(y_1)+ f_{x_3}(y_2)+\cdots = f_{\widehat{M}_{z}(\underline{x})}(\underline{y}), \quad \underline{y}\in \ell_2(\mathbb{X})^*.
\]
Consequently, $M_{z*}(\underline{x}) = \widehat{M}_{z}(\underline{x})$ for all $\underline{x}\in \ell_2(\mathbb{X}).$
\end{proof}

The adjoint of a Hilbert space unitary is its inverse. Interestingly, the same happens for the Banach space unitaries.

\begin{prop}\label{Adjoint of bilateral shift in Banach space setup}
Suppose $\mathbb{X}$ is a reflexive, smooth and strictly convex Banach space and $\widetilde{U}$ is a unitary operator on $\X$. Then $\widetilde{U}_*= \widetilde{U}^{-1}$.
\end{prop}

\begin{proof}
It follows from (\ref{Definition of Phi}) that $\widetilde{U}^{\times} f_x = \widetilde{U}^{\times}f_{\widetilde{U}\widetilde{U}^{-1}(x)} = f_{\widetilde{U}^{-1}(x)}$ for all $x\in \X.$ Therefore, $\widetilde{U}_*(x)= \widetilde{U}^{-1}(x)$ for all $x\in \X$.
\end{proof}

Now we present the main result of this Subsection. This provides a pair of characterizations for a Hilbert space of dimension greater than $2$.

\begin{thm}\label{Characterisation of Hilbert space by rank one operators}
Let $\mathbb{X}$ be a reflexive, smooth and strictly convex Banach space with $\dim(\mathbb{X})> 2.$ Then, the following are equivalent.
\begin{enumerate}
\item $\mathbb{X}$ is a Hilbert space.

\item $T_*$ is a linear map for every rank one linear map $T\in \mathcal{B}(\mathbb{X}).$

\item $(I+U)_*= I + U_*$, where $U$ is the bilateral shift operator $U$ on $\ell_2(\mathbb{Z}, \mathbb{X})$ defined by
\[
\quad U(\dots, x_{-2}, x_{-1}, \boxed{x_0}, x_1, x_2, \dots)= (\dots, x_{-2}, \boxed{x_{-1}}, x_0, x_1, x_2, \dots), \quad (x_n)_{n\in\mathbb{Z}}\in \ell_2(\mathbb{Z},\mathbb{X}),
\]
where the box denotes the $0$-th position in the sequence.
\end{enumerate}
 
\end{thm}

\begin{proof}

$(1) \Leftrightarrow (2)$. The forward direction is obvious. We prove the backward direction, i.e. $(2) \Rightarrow (1)$. Let $x,y\in \mathbb{X}$ be arbitrary. Consider the rank one operator $T_{xy}: \mathbb{X}\to\mathbb{X}$ defined by $T_{xy}(z)= f_x(z)y,~z\in \mathbb{X}.$ Then $T_{xy}$ is a bounded linear map and $\|T_{xy}\|= \|x\|\|y\|.$ Now, we find the adjoint operator, i.e. $T_{xy*}.$ Let $z\in \mathbb{X}$ be arbitrary. Then
\[
       (T_{xy}^\times f_z) (w) = f_z(T_{xy}w)= f_z(y)f_x(w)= f_{\overline{f_z(y)}x}(w), \quad w\in \mathbb{X}.
\]
Therefore, by the smoothness of $\mathbb{X}$ we have that $T_{xy*}(z)= \overline{f_z(y})x$ for any $z\in \mathbb{X}.$ Now, the linearity of the family of operators $\{T_{xy*}: ~ x,y\in \mathbb{X}\}$ implies that the Birkhoff-James orthogonality is left additive on $\mathbb{X}.$ Consequently, $\mathbb{X}$ is a Hilbert space by \cite[Theorem 2]{James 2}.

\smallskip

\noindent $(1) \Leftrightarrow (3)$. Again the forward direction is obvious. We prove $(3) \Rightarrow (1)$.
By hypothesis, we have
\begin{equation}\label{Equation 8}
 f_{\underline{x}}+ f_{U^{-1}\underline{x}} = (I+U)^\times f_{\underline{x}} = f_{\left(I+U\right)_*~\underline{x}} = f_{\left(I+U_*\right)\underline{x}} = f_{\left(\underline{x}+ U^{-1}\underline{x}\right)}, \quad \underline{x}\in \ell_2(\mathbb{Z},\mathbb{X}).
\end{equation}
The last equality in (\ref{Equation 8}) follows from Proposition \ref{Adjoint of bilateral shift in Banach space setup}. Let $x,y\in \mathbb{X}$ be arbitrary. Consider, $\underline{x_0}= ( \cdots, \mathbf{0}, \boxed{x},y-x,\mathbf{0},\cdots)\in \ell_2(\mathbb{Z}, \mathbb{X}).$ Then it follows from the identification $\ell_2(\mathbb{Z},\mathbb{X})^* \cong \ell_2(\mathbb{Z},\mathbb{X}^*)$ that the element $f_{\underline{x_0}}\in \ell_2(\mathbb{Z},\mathbb{X})^*$ is uniquely identified with $(\cdots, \mathbf{0}, \boxed{f_{x}}, f_{y-x}, \mathbf{0}, \cdots )\in \ell_2(\mathbb{Z},\mathbb{X}^*).$ Thus, it follows from (\ref{Equation 8}) that the elements $(f_{\underline{x_0}}+ f_{U^{-1}\underline{x_0}})$ and $f_{\left(\underline{x_0}+ U^{-1}\underline{x_0}\right)}$ are equal. Consequently,
\[
  (\cdots , \mathbf{0}, f_x, \boxed{f_x+ f_{y-x}}, f_{y-x}, \mathbf{0}, \cdots )= (\cdots, \mathbf{0}, f_x, \boxed{f_y},f_{y-x}, \mathbf{0}, \cdots ), 
\]
as $\underline{x_0} + U(\underline{x_0}) = (\cdots, \mathbf{0}, x,\boxed{y}, y-x,\mathbf{0},\cdots )$.
By component-wise comparison, we have $f_x + f_{y-x} = f_y.$ Therefore, for $x,y \in \X$ we have
$
  f_x -f_y = f_x + f_{-y} = f_{x-y}$.
This shows that the Birkhoff-James orthogonality is left additive and consequently $\mathbb{X}$ is a Hilbert space by \cite[Theorem 2]{James 2}.

\end{proof}


\section{Dilation and norm attainment set of an operator} \label{Sec:08}

\vspace{0.3cm}
\noindent
Throughout this Section, we shall consider only nonzero operators. The norm attainment set of an operator is a well-studied area in Banach space theory, e.g. see \cite{BP, CN, Lindenstrauss, Sain3,  Sain4, Sain and Paul, Zizler}. In this Section, we study the connection between the norm attainment set and (isometric) dilation of a contraction. The norm attainment set of an operator $T$ on a Banach space $\X$ is defined as 
\[
M_T=\{x\in S_{\X}\,:\, \|Tx\|=\|T\| \}.
\]
Recently, Paul and Sain characterized a finite-dimensional real Hilbert space in terms of the norm attainment sets of operators acting on it. The result is stated below.

\begin{thm}\cite[Theorem 2.2]{Sain and Paul}\label{Thm:IPS & norm sub-sp.}
A finite dimensional real normed linear space $\X$ is an inner product space if and only if for any linear operator $T$ on $\X$, $T$ attains its norm at $e_1,e_2 \in S_\X$ implies $T$ attains its norm at $span\{e_1,e_2\}\cap S_\X$.
\end{thm}

The necessary part of the theorem holds for any inner product space, real or complex and irrespective of dimension. For an operator $T$ acting on a Hilbert space $\mathcal{H}$, $T$ attains its norm at a unit vector $x$ if and only if $T^*Tx = \|T\|^2 x$. In other words, if $T\in \mathcal{B}(\mathcal{H})$ attains its norm, then $M_T = S_{\mathcal{H}_0}$, where $\mathcal{H}_0$ is the eigenspace of $T^*T$ corresponding to the eigenvalue $\|T\|^2$. However, norm attainment set of a Banach space operator is not so straight forward. In \cite{Sain 5} Sain characterized the norm attainment set of a Banach space operator in terms of semi-inner-products \cite{JRG, GL, Sain 5}.

\medskip

Let $\X$ be a Banach space and $T\in \mathcal{B}(\X)$. If $T= \alpha S$ for some operator $S$ on $\X$ and $\alpha\in \mathbb{C}\setminus \{0\}$, then it is obvious that $M_T = M_S$. So, it follows from here that $M_{T} = M_{A}$, where $A= \frac{S}{\|S\|}$ is a norm-one contraction. Thus, for studying the norm attainment sets of scalar multiples of norm-one contractions that admit isometric dilations, it suffices to consider the following class:
\[
 \mathcal{A}(\X) = \left\{ T\in \mathcal{B}(\X): \text{$\frac{T}{\|T\|}$ dilates to  an isometry on some normed (or, semi-normed) space} \right\}.
\]
Evidently, this class coincides with $\mathcal{B}(\X)$ if the underlying space $\X$ is a Hilbert space. In case of Banach space, the class $\mathcal{A}(\X)$ is nonempty as it contains all scalar multiples of isometries, by Theorem \ref{Necessary and sufficient condition for a contraction to be dilated 1}. Nevertheless, every operator in $\mathcal{A}(\X)$ can be characterized in view of Theorem \ref{Necessary and sufficient condition for a contraction to be dilated 1} in the following way.

\begin{thm}
Let $\X$ be a Banach space and $T\in \mathcal{B}(\X)$. Then $T\in \mathcal{A}(\X)$ if and only if the function $\widehat{A}_T: \X \to [0,\infty)$ given by $\widehat{A}_T(x) = \left(\|T\|^2 \|x\|^2 -\|Tx\|^2 \right)^{1\slash 2}$ defines a norm $($or semi-norm$)$ on $\X$.
\end{thm}

\begin{proof}
Consider the operator $S= \frac{T}{\|T\|}$. Then for all $x\in \X$, we have $\widehat{A}_T(x) = \|T\| A_S(x)$, where $A_S$ is as in (\ref{Equivalent norm function}). Therefore, $\widehat{A}_T$ defines a norm (or semi-norm) on $\X$ if and only if $A_S$ does the same. The rest follows from Theorem \ref{Necessary and sufficient condition for a contraction to be dilated 1}.
\end{proof}

The fact that $\widehat A_T$ defines a norm and that $A_{T/\|T\|}$ defines a norm are equivalent. Thus, it follows from Lemma \ref{lem:new-0122} that the function $\widehat{A}_T$ is not equivalent to the original norm on $\X$ for any $T\in \mathcal{A}(\X)$.
Now, we arrive at the main result of this Section.

\begin{thm}\label{thm:623}
Let $\X$ be a Banach space and $T\in \mathcal{A}(\X)$. Let $V$ on $\X'$ be an isometric dilation of $\frac{T}{\|T\|}$ with $W:\X \rightarrow \X'$ being the isometric embedding. Then $x\in M_{T}$ if and only if $x \in \ker \left( ( I-P_{_{W(\X)}}) VW \right) \bigcap S_{\X}$.
\end{thm}

\begin{proof}
Let $x\in \X$ be arbitrary. Then we have
\begin{align*}
  \|T\|^2 \|x\|^2 - \|Tx\|^2 = \|T\|^2\left( \|x\|^2 - \left\|\dfrac{T}{\|T\|}(x)\right\|^2\right) & = \|T\|^2 \left(\|V W(x)\|^2 - \| P_{_{W(\X)}}VW (x)\|^2\right) \nonumber\\
 & = \|T\|^2 \| (I-P_{_{W(\X)}})V W(x)\|^2.
\end{align*}

Consequently, $x\in M_T$ if and only if $x\in \ker\left((I-P_{_{W(\X)}})VW \right)\bigcap S_\X$.
\end{proof}

For a Banach space contraction $T\in \mathcal{B}(\X)$, let us consider the set 
\[
\widehat{M}_T= \{x\in \X: \|T\|\|x\| = \|Tx\|\}.
\]
Then $M_T = S_\X \cap \widehat{M}_T$. For all scalar multiple of isometries $V$ on $\X$, the set $\widehat{M}_V$ is equal to the whole space $\X$ and $M_V=S_\X $. In general the sets $M_T$ and $\widehat{M}_T$ are related in the following way.

\begin{lem}\label{lem:norm sub-sp.}
Let $\X$ be a Banach space and $T\in \mathcal{B}(\X)$. Then the following are equivalent:

\begin{enumerate}
\item[(i)] $span\{x,y\}\cap S_{\X} \subseteq M_T$ whenever $x,y\in M_T$;
\item[(ii)] $\widehat{M}_T$ is a linear subspace of $\X$.
\end{enumerate}
\end{lem}

\begin{proof}
Obvious.
\end{proof}

Theorem \ref{Thm:IPS & norm sub-sp.} and Lemma \ref{lem:norm sub-sp.} show that for every non-Hilbert real finite-dimensional Banach space $\X$, there is an operator $T \in \mathcal{B}(\X)$ such that $\widehat{M}_T$ is not a linear subspace of $\X$. However, it is evident that $\mathbf{0}\in \widehat{M}_T$ for every $T\in \mathcal{B}(\X)$ and $x\in \widehat{M}_T$ implies that $\alpha x\in \widehat{M}_T$ for all scalars $\alpha$. In the following example, we construct an operator $T$ on a complex Banach space for which $\widehat{M}_T$ is not a linear subspace. 

\begin{eg} \label{Example:new-0711}
Let $\mathbb{\X}= \left(\mathbb{C}^2, \|\cdot\|_1\right)$ over the field of complex numbers and choose $0< \lambda < 1$. Consider the operator $T: \X \to \X$ defined by $T(x,y) = (\lambda (x-y), 0)$ for all $(x,y)\in \X$. Then $\|T\| = \lambda$ and $e_1=(1,0)$, $e_2=(0,1)$ belong to $M_T$ but $\frac{e_1+e_2}{\|e_1 +e_2\|} \notin M_T$. Consequently, $\widehat{M}_T$ is not a subspace of $\X$.
\end{eg}

Note that Theorem \ref{thm:623} explicitly describes the subspace $\widehat{M}_T$ for every $T\in \mathcal{A}(\X)$. Despite the fact that $\widehat{M}_T$ is not a subspace in general, Theorem \ref{thm:623} also shows that there is a wide class of operators $T$ for which $\widehat{M}_T$ is a linear subspace of $\X$. Especially, if $T$ dilates to $\|T\|V$ for some isometry $V$ on a normed (or semi-normed) space, then $\widehat{M}_T$ is a subspace of $\X$. It follows from the dilation theorem due to Sz.-Nagy (see \cite[Ch. 1, Sec. 4, Theorem 4.1]{Nagy Foias}) that the class $\mathcal{A}(\X)$ coincides with $\mathcal{B}(\X)$ if $\X$ is a Hilbert space. 

\begin{lem}\label{lem:mod. norm. Hilbert}
Let $\mathcal{H}$ be a Hilbert space, and let $T\in \mathcal{B}(\mathcal{H})$ be arbitrary. Then $\widehat{A}_T$ defines a norm if $M_T = \emptyset$ and a semi-norm if $M_T \neq\emptyset$.
\end{lem}

\begin{proof}
Note that $T \neq \mathbf 0$ and let us consider the operator $S= \frac{T}{\|T\|}$. Then it is easy to see that $\widehat{A}_T(h) = \|T\| \|D_{S}(h)\|$ for all $h\in \mathcal{H}$, where $D_S=(I-T^*T)^{\frac{1}{2}}$. Therefore, the linearity of the operator $D_S$ shows that $\widehat{A}_T$ defines a norm if $M_T = \emptyset$ and a semi-norm if $M_T \neq \emptyset$.
\end{proof}

For an arbitrary Banach space $\X$, the class of operators $T\in \mathcal{B}(\X)$ for which $\widehat{M}_T$ is a linear subspace of $\X$, can be strictly larger than the class $\mathcal{A}(\X)$. In the following example, we find a Banach space operator $T$ for which $\widehat{M}_T$ is a linear subspace but $T\notin \mathcal{A}(\X)$.

\begin{eg}
Let $\X = (\mathbb{C}^3, \|\cdot\|_1)$, and choose $0< \lambda < 1$. Consider the operator $T: \X \to \X$ defined by $T(x,y,z) = (x, \lambda(y-z), 0)$. Then $\|T\| =1$ and $\widehat{M}_T = span\{(1,0,0)\}$. But $T\notin \mathcal{A}(\X)$, as the function $\widehat{A}_T$ does not satisfy the triangle inequality for the vectors $e_2 = (0,1,0)$ and $e_3=(0,0,1)$.
\end{eg}

We conclude this section with an example of an operator on a Banach space $\X$ that belongs to the class $\mathcal{A}(\X)$.

\begin{eg}
Let $\X$ be a complex Banach space and let $(\lambda_n)$ be a bounded sequence of complex numbers. Consider the operator $T_\lambda: \ell_2(\X) \to \ell_2(\X)$ defined by $T_\lambda(x_n) = (\lambda_n x_n)$. Then $\|T_\lambda\| = \displaystyle\sup_{n\in \mathbb{N}}\vert \lambda_n \vert$. So, for any $(x_n) \in \ell_2(\X)$ we have
\begin{align*}
 \left(\widehat{A}_{T_\lambda}(x_n) \right)^2 & = \|T_{\lambda}\|^2 \|(x_n)\|^2 - \|T_{\lambda}(x_n)\|^2 = \sum_{n=1}^\infty \|x_n\|^2 \left( \|T_\lambda\|^2 - \vert\lambda_n\vert ^2\right) = \|T_{\mu}(x_n)\|^2,
\end{align*}
where $\mu_n = \left( \|T_\lambda\|^2 - \vert\lambda_n\vert ^2\right)^{1\slash 2}$ for all $n\in \mathbb{N}$ and $T_\mu: \ell_2(\X) \to \ell_2(\X)$ is defined by $T_\mu(x_n) = (\mu_n x_n)$. Evidently, $\widehat{A}_{T_\lambda}$ defines a norm on $\ell_2(\X)$ if $\vert \lambda_k \vert \neq \displaystyle\sup_{n\in \mathbb{N}}\vert\lambda_n\vert$ for any $k\in \mathbb{N}$ and a semi-norm if $\vert \lambda_k \vert = \displaystyle \sup_{n\in \mathbb{N}} \vert\lambda_n\vert$ for some $k\in \mathbb{N}$. \qed
\end{eg}


\section{Canonical decomposition of a Banach space contraction} \label{sec:06}

\vspace{0.3cm}

\noindent The canonical decomposition of a contraction $T$ on a Hilbert space $\mathcal{H}$ (see Chapter-I of \cite{Nagy Foias}), splits $T$ into two orthogonal parts of which one is a unitary and the other is a completely non-unitary (c.n.u.) contraction, i.e., $\mathcal{H}$ admits an orthogonal decomposition $\mathcal{H} = \mathcal{H}_0 \oplus_2 \mathcal{H}_1$ into reducing subspaces $\mathcal{H}_0$, $\mathcal{H}_1$ of $T$ such that $T|_{\mathcal{H}_0}$ is a unitary and $T|_{\mathcal{H}_1}$ is a c.n.u. contraction. The space $\mathcal{H}_0$ is given by
\[
  \mathcal{H}_0 = \{ h\in \mathcal{H} : \|T^n h\| = \| h\| = \|T^{*n}h\|, ~ n\in \mathbb{N}\}.
\]
Thus, $\mathcal{H}_0$ is the maximal reducing subspace of $T$ such that $T|_{\mathcal{H}_0}$ is a unitary. Note that the maximal invariant subspace on which $T$ is an isometry is the following:
\[
  \mathcal{H}(T) = \{ h\in \mathcal{H} : \|T^n h\| = \|h\|, ~ n\in \mathbb{N}\}.
\]
For a Banach space $\X$ and a contraction $T$ on $\X$, the space
\begin{equation}\label{Max. isometry subspace}
  \X(T) = \{ h\in \mathcal{H} : \|T^n h\| = \|h\|, \ n\in \mathbb N \}
\end{equation}
is always a closed subset of $\X$ but may not be a linear subspace as the following example shows.

\begin{eg}
Consider the Banach space $\X = \left(\mathbb{C}^3, \|\cdot \|_1 \right)$. Then, every $(x,y,z) \in \X$ has the norm $|x|+|y|+|z|$. Consider the operator $T: \X \to \X$ defined by $T(x,y,z) = (x-y, 0, z)$ for all $(x,y,z)\in \X$. Then it is obvious that $T^n = T$ for all $n\geq 1$. Thus, $\X(T) = \{a\in \X : \|Ta\| =\|a\| \}$. This is not a linear subspace because $e_1=(1,0,0)$, $e_2 = (0,1,0)$ belong to $\X(T)$ but $(e_1+e_2) \notin \X(T)$.
\end{eg}

Despite the fact that $\X(T)$ is not a closed (linear) subspace of a Banach space $\X$ in general, there is a wide class of contractions $T\in \mathcal{B}(\X)$ for which $\X(T)$ is a closed subspace. Indeed, if the function $A_T$ (as in (\ref{Equivalent norm function})) gives a semi-norm on $\X$, then $\X(T)$ becomes a closed subspace as the following result explains.

\begin{lem}
Let $\X$ be a Banach space and $T\in \mathcal{B}(\X)$ be a contraction. If $A_T$ $($as in $(\ref{Equivalent norm function}))$ defines a semi-norm on $\X$, then $\X(T)$ (as in $(\ref{Max. isometry subspace})$) is a closed subspace of $\X$.
\end{lem}

\begin{proof}
Consider the semi-normed linear space $\X_0 = (\X, A_T)$ and $\mathcal{K} = \ell_2(\X \oplus_2 \X_0)$. Then the map $W: \X \to \mathcal{K}$ defined by
\[
  W(x) = ((x,\mathbf{0}), \mathbf{0}, \mathbf{0}, \dotsc, ), \qquad x\in \X
\]
is an isometric embedding of $\X$ into $\mathcal{K}$ and the map $V: \mathcal{K} \to \mathcal{K}$ as in (\ref{eqn:new-061}) (see the proof of Theorem \ref{Necessary and sufficient condition for a contraction to be dilated 1}) is an isometric dilation of $T$. Thus, for every $n\geq 1$, we have
\[
  \X_n = \{x\in \X : \|T^n x\| = \|x\|\} = \ker\left(\left( I- P_{_{W(\X)}} \right) V^n W \right),
\]
which shows that $\X_n$ is a closed subspace of $\X$. Consequently, $\X(T) = \cap_{n=1}^\infty \X_n$ is a closed subspace of $\X$ and the proof is complete.
\end{proof}

We denote the class of contractions $T$ on a Banach space $\X$ such that $\X(T)$ is a closed subspace by
\[
  \mathcal{F}(\X) =\{ T\in \mathcal{B}(\X) : \X(T) \text{ is a closed subspace of }\X\}.
\]
Interestingly, the class $\mathcal F(\X)$ is strictly bigger than the class for which $A_T$ defines a semi-norm. In the following example, we find a Banach space contraction $T$ for which $A_T$ does not define a semi-norm yet $\X(T)$ is a closed subspace.

\begin{eg}
Let $\X = (\mathbb{C}^3, \|\cdot\|_1)$, and choose $0< \lambda < 1$. Consider the operator $T: \X \to \X$ defined by $T(x,y,z) = (x, \lambda(y-z), 0)$. Then for all $n\geq 1$, $T^n(x,y,z) = (x, \lambda^n (y-z), 0)$ for all $(x,y,z)\in \X$. Therefore, $\X(T) = span\{(1,0,0)\}$ is a closed subspace but the function $A_T$ does not define a semi-norm on $\X$, since it does not satisfy the triangle inequality for the vectors $(0,1,0)$ and $(0,0,1)$. 
\end{eg}

To go parallel with the Hilbert space theory, we now define completely non-unitary (c.n.u.) contraction on a Banach space.
\begin{defn}
A contraction $T$ on a Banach space $\mathbb{X}$ is said to be \textit{completely non-unitary} or simply \textit{c.n.u.} if there is no nonzero invariant subspace $\mathbb{Y}$ of $\mathbb{X}$ such that $T|_\mathbb{Y}: \mathbb{Y}\to\mathbb{Y}$ is a unitary.
\end{defn}

Note that if $\X$ is a Hilbert space and $T|_\Y:\Y\to \Y$ is a unitary for some closed subspace $\Y$ of $\X$, then $\Y$ is also a reducing subspace for $T$. Thus, our definition of a c.n.u. contraction on a Banach space generalizes that of a c.n.u. Hilbert space contraction. Now, we present a canonical decomposition for a Banach space contraction that belongs to the class $\mathcal{F}(\X)$.

\begin{thm}\label{Canonical Decomposition Banach space}
Let $T$ be a contraction on a Banach space $\mathbb{X}$ and let $\X(T)$ as in $(\ref{Max. isometry subspace})$ be a linear subspace of $\X$. If $T|_{\X(T)}$ is a Wold isometry with right-complemented range, then there is a unique closed subspace $\mathbb{W}$ of $\X$ such that 
\begin{itemize}
    \item[(a)] $\mathbb{W}$ is invariant under $T$ ;
    
    \smallskip
    
    \item[(b)] $T|_{\mathbb{W}}:\mathbb{W}\to \mathbb{W}$ is a unitary ; 
        
    \smallskip

    \item[(c)] if $\mathbb{M}$ is an invariant subspace of $T$ for which $T|_{\mathbb{M}}: \mathbb{M}\to \mathbb{M}$ is a unitary, then $\mathbb{M}\subseteq \mathbb{W}$.
\end{itemize} 

In addition, if $\mathbb{W}$ is smooth and $\mathbb{X} = \mathbb{W}\bigoplus_\SR \mathbb{W}'$, then $T|_{\mathbb{W}'}$ is a completely non-unitary contraction.
\end{thm}

\begin{proof}
Since the operator $\widetilde{T}:= T|_{\X(T)}:\X(T)\to \X(T)$ is a Wold isometry with right-complemented range, there are subspaces $\X_1,\X_2$ of $\X(T)$ that are invariant under $T$ such that $\X(T) =\X_1\oplus \X_2$, and $T|_{\X_1}$ is a unitary whereas $T|_{\X_2}$ is a unilateral shift by Lemma \ref{lem:new-041}. By Theorem \ref{Left Inverse of Wold Isometry}, $\widetilde{T}$ has a left inverse, say $A:\X(T)\to \X(T)$. Let $T_i:=T|_{\X_i}: \X_i \to {\X_i}$ ($i=1,2$). Suppose, $\widetilde{T}_2$ is the left inverse of $T_2$ as in Theorem \ref{Left Inverse of Wold Isometry}. Then with respect to the decomposition $\X(T)=\X_1\oplus \X_2$, $\widetilde{T}$ and its left inverse $A$ can be written as
\[
\widetilde{T}=\begin{bmatrix}
				T_1 & 0 \\
				0 & T_2  \\
			\end{bmatrix} , ~ 
            A=\begin{bmatrix}
				T_1^{-1} & 0 \\
				0 & \widetilde{T_2} \\
			\end{bmatrix}.
   \]
Evidently, 
\begin{equation}\label{Equation 4}
    A\widetilde{T}=\begin{bmatrix}
				I_{\X_1} & 0 \\
				0 & I_{\X_2} \\
			\end{bmatrix}=I_{\X(T)}\quad \mathrm{and} \quad \widetilde{T}A=\begin{bmatrix}
				I_{\X_1} & 0 \\
				0 & T_2\widetilde{T_2} \\
			\end{bmatrix}.
\end{equation}
Set
\[
 \mathbb{W} = \left(\bigcap_{n= 0}^{\infty}T^n(\X(T))\right)\bigcap\left(\bigcap_{n= 0}^{\infty}A^n(\X(T))\right).
\]
Now, we prove the statements (a), (b), (c) of the theorem.

\medskip

\noindent {\bf (a)} Let $x\in \mathbb{W}$ be arbitrary. Then for each $n\in\mathbb{N},$ there exist $x_n, y_n$ in $\X(T)$ such that 
$
T^nx_n=x=A^ny_n.
$
Therefore,
\[
Tx = T^{n+1}x_n = T^n(Tx_n),\quad n\geq 1,
\]
 and
\[
 Tx=A^nT^n(Tx)=A^n(T^{n+1}x),\quad n\geq 1.
\] Thus, $Tx\in \mathbb{W}$ and consequently $\mathbb{W}$ is invariant under $T$.

\medskip

\noindent {\bf (b)} Since $\widetilde T=T|_{\X(T)}$ is an isometry and since $\mathbb{W}\subseteq\X(T)$ is invariant under $T$, we have that $T|_\mathbb{W}:\mathbb{W}\to \mathbb{W}$ is also an isometry. Let $x\in \mathbb{W}$ be arbitrary, and $x_n,~y_n$ be as in part-{\bf (a)}. Then by (\ref{Equation 4}), we have
\[
\|Ax\|  = \|ATx_1\|  = \|x_1\| = \|Tx_1\|    = \|x\|.     
\]
Therefore, $A$ is an isometry. Also, for all $n \in \mathbb N$ we have that
$
  Ax=A^{n+1}y_n$ and 
\[
  Ax=A(T^nx_n)=AT(T^{n-1}x_n)=T^{n-1}x_n.
\]
Consequently, $Ax\in \mathbb{W}$ and $\mathbb{W}$ is invariant under $A$. Moreover,
\[
TAx=TATx_1=T(ATx_1)=Tx_1=x=A\widetilde{T}x = ATx.
\]
Therefore, $A$ and $T$ are unitaries on $\mathbb{W}$ and they are inverse of each other.

\medskip

\noindent {\bf (c)} Let $\mathbb{M}$ be a subspace of $\X$ satisfying the stated conditions. Consider the operator $S:=T|_\mathbb{M}:\mathbb{M}\to \mathbb{M}.$ Since $\mathbb{M}$ is invariant under $T$, it follows that
$
S^nx =T^nx$ for all $n\geq 1$ and $x\in \mathbb{M}$.
Also, for all $n \geq 0$ and for all $x \in \mathbb M$, we have
$
  \|S^n(x)\|=\|T^n(x)\|=\|x\|$.
Thus, $\mathbb{M}\subseteq \X(T)$. Let $x\in \mathbb{M}$ be arbitrary. It follows from the surjectivity of $S^n$ that for every $n\in\mathbb{N},$ there is $x_n\in \mathbb{M}$ such that $ S^nx_n=x$. Therefore, for $n \geq 1$ we have $x=S^n(x_n)=T^n(x_n)$ and $x = A^n T^nx = A^ny_n$, where $y_n=T^nx\in \X(T)$. Consequently, $x\in \mathbb{W}$ and $\mathbb{M}\subseteq \mathbb{W}$. If there is another subspace $\widetilde{\mathbb{W}}$ satisfying (a), (b), (c) of the statement, then it follows from (c) that $\mathbb{W}=\widetilde{\mathbb{W}}$.

\medskip

 In addition, if $\mathbb{W}$ is smooth, then $J(w)=\{T^\times f_{Tw}\}$ for every nonzero $w\in \mathbb{W}$. Indeed, for any $x\in \X$ we have
\[
 \left|T^\times f_{Tw}(x)\right|\leq \|Tx\|\|Tw\| \leq \|w\|\|x\|,
\]
and for $x=w$ it reduces to $
 f_{Tw}(Tw) = \|Tw\|^2 =\|w\|^2$.
Thus, $J(w) =\{T^\times f_{Tw}\}$ by the smoothness of $\mathbb{W}$. Whenever $\X=\mathbb{W}\bigoplus_\SR \mathbb{W}'$, by virtue of $T\mathbb{W} = \mathbb{W}$ and $\mathbb{W}\perp_B \mathbb{W}'$, we have 
$
 T^\times f_{Tw}(w') = 0$ for all $w\in \mathbb{W}'$.
This shows that $\mathbb{W}\perp_B T\mathbb{W}'$. If $Tw'=u+v$ for $u\in \mathbb{W}$ and $w'\in \mathbb{W}'$, then we have $u\perp_B (u+v)$ and $u\perp_B v$. Since by hypothesis $u$ is smooth, it follows from the right-additivity of Birkhoff-James orthogonality that $u=\mathbf{0}$ and hence $T\mathbb{W}'\subseteq \mathbb{W}'$. Consequently, $T|_{\mathbb{W}'}$ is completely non-unitary contraction by virtue of (c). This completes the proof.
\end{proof}

Theorem \ref{Canonical Decomposition Banach space} generalizes the canonical decomposition of a Hilbert space contraction. The subspace $\mathbb{W}$ of $\X$ as in Theorem \ref{Canonical Decomposition Banach space} is the maximal invariant subspace of $T$ such that $T|_{\mathbb W}$ is a unitary. Thus, $T|_{\mathbb W}$ is called the \textit{unitary part} of $T$. Let us have an explicit example of a Banach space contraction and see its canonical decomposition.

\begin{eg}\label{Eg: Canonical Decomposition}
For $1<p< \infty$ and $p\neq 2$, let $\X=\ell_p$. Consider the operator $T:\X\to \X$ defined by
\begin{align*}
& T(x_0,x_1,x_2,x_3,x_4,x_5,x_6,x_7,x_8,x_9,x_{10},x_{11},x_{12},x_{13}, \dots)\\
& = \left(\frac{1}{2}x_0,x_1,\frac{1}{2^2}x_2,0,\frac{1}{2^3}x_4, x_5,\frac{1}{2^4}x_6, x_3, \frac{1}{2^5}x_8,x_9, \frac{1}{2^6}x_{10},x_7,\frac{1}{2^7}x_{12}, x_{13}, \dots  \right).
\end{align*}
Evidently, $T$ is a contraction. For any non-negative integer $k$, let $e_k$ denote the vector having $1$ at $k$-th coordinate and zero elsewhere. It follows from the definition of $T$ that
\[
T(e_{2k}) = \frac{1}{2^{k+1}}e_{2k}, ~ T(e_{4k+1}) = e_{4k+1},~ T(e_{4k+3}) = e_{4k+7}, \qquad k\geq 0.
\]
Also,
\[
 T^n (e_{2k}) = \frac{1}{2^{n(k+1)}}e_{2k},~T^n (e_{4k+1}) = e_{4k+1},~T^n(e_{4k+3}) = e_{4(n+k)+3}, \qquad n\geq 1, \qquad k \geq 0.
\]
Consider the collection
\[
 \X(T)= \left\{(x_0,x_1, \dots)\in \X:~\|T^n (x_0,x_1,\dots)\| = \|(x_0,x_1,\dots)\|,~n\geq 0\right\}.
\]
Then $\X(T)$ is a closed subspace of $\X$ and
$
 \X(T) =\{(x_0,x_1, \dots)\in \X: x_{2n}=0, ~n\geq 0\}.
$
Note that
\[
 T(\X(T))=\{(x_0,x_1, \dots)\in \X: x_{2n}=0,~x_3=0,~n\geq 0\}\subseteq \X(T).
\]
Thus, $\X(T)$ is invariant under $T$ and $T|_{\X(T)}$ is an isometry. The range of $T|_{\X(T)}$ is right-complemented by the subspace $\mathcal{L}=\{(0,0,0,x,0,0,,\dots):~x\in \mathbb{C}\}$. Therefore, $T|_{\X(T)}$ is a Wold isometry with right-complemented range and the hypotheses of Theorem \ref{Canonical Decomposition Banach space} are satisfied. Thus, there is a subspace $\mathbb{W}$ in $\X$ satisfying (a), (b), (c) of Theorem \ref{Canonical Decomposition Banach space}. It is not difficult to see the subspace $\mathbb{W}$ for this example is of the form
\[
 \mathbb{W}= \{(x_0,x_1, \dots)\in \X: x_{2n}=x_{4n+3}=0,~n\geq 0\}.
\]
In addition, $\mathbb{W}$ is smooth and
\[
 \mathbb{W}' = \{(x_0,x_1,\dots)\in \X: x_{4n+1}=0,~ n\geq 0\}
\]
is the right-complement of $\mathbb{W}$ in $\X$. Also, $T|_{\mathbb{W}'}$ is a completely non-unitary contraction. \qed
\end{eg}

In \cite{Levan}, Levan took a next step to decompose the c.n.u. part of a Hilbert space contraction. In this connection, we mention that a contraction $T$ on a Hilbert space $\mathcal{H}$ is said to be \emph{completely non-isometry} or simply \textit{c.n.i.} if there is no nonzero invariant subspace $\mathcal{H}_0 \subseteq \mathcal{H}$ of $T$ such that $T|_{\mathcal{H}_0}:{\mathcal{H}_0}\to{\mathcal{H}_0}$ is an isometry. Levan's decomposition of a c.n.u. Hilbert space contraction $T$ splits $T$ further into two orthogonal parts of which one is a c.n.i. contraction and the other is a unilateral shift. We adopt the same definition for a c.n.i. contraction in the Banach space setting. Below we see that a c.n.u. Banach space contraction can also admit a Levan-type decomposition under certain assumptions. Interestingly, these assumptions are automatic if the underlying Banach space is a Hilbert space.

\medskip

Suppose that $T$ is a contraction on a Banach space $\X$ satisfying the hypotheses of Theorem \ref{Canonical Decomposition Banach space} and that the maximal invariant subspace $\mathbb{W}$ on which $T$ is a unitary is smooth and right-complemented in $\X$. Consider the collection 
\[
 \mathfrak{W}=\left\{ \widetilde{\Y}\subseteq \X(T): T(\widetilde{\Y})\subseteq \widetilde{\Y},~ \X= \widetilde{\Y}\oplus \widetilde{\Y}_1,~T(\widetilde{\Y}_1)\subseteq \widetilde{\Y}_1 \text{ for some closed subspace } \widetilde{\Y}_1 \text{ in } \X\right\},
\]
where $\X(T) \subseteq \X$ is as in (\ref{Max. isometry subspace}) and $\X(T) \supseteq \mathbb{W}$. Then the collection $\mathfrak{W}$ is nonempty, since $\mathbb{W}\in \mathfrak{W}.$ Then by Zorn's Lemma, $\mathfrak{W}$ has a maximal element, say $\widetilde{\mathbb{W}}$. Therefore, $\X$ can be decomposed as $\X=\widetilde{\mathbb{W}}\oplus \mathbb{W}'.$ Evidently, $T|_{\widetilde{\mathbb{W}}}:\widetilde{\mathbb{W}}\to\widetilde{\mathbb{W}}$ is an isometry. If $T|_{\widetilde{\mathbb{W}}}$ is a Wold isometry with right-complemented range, then there exist subspaces $\mathbb{W}_1$ and $\mathbb{W}_2$ of $\widetilde{\mathbb{W}}$ such that $\widetilde{\mathbb{W}}= \mathbb{W}_1\oplus \mathbb{W}_2$, where $T|_{\mathbb{W}_1}:\mathbb{W}_1\to \mathbb{W}_1$ is a unitary and $T|_{\mathbb{W}_2}:\mathbb{W}_2\to \mathbb{W}_2$ is a unilateral shift by Lemma \ref{lem:new-041}. Altogether, we have 
\begin{equation}\label{Levan Decomposition in Banach space}
 \X = \left(\mathbb{W}_1 \oplus \mathbb{W}_2\right) \oplus \mathbb{W}'
\end{equation}
where $\mathbb{W}'$ is invariant under $T$ and $T|_{\mathbb{W}'}:\mathbb{W}'\to \mathbb{W}'$ is a completely non-isometry contraction. We end this Section with the following example that illustrates Levan-type decomposition of a c.n.u. Banach space contraction.
 
\begin{eg}
We consider the c.n.u. part of the contraction $T$ as in Example \ref{Eg: Canonical Decomposition}. Then 
\[
 \X(T) =\{(x_0,x_1, \dots)\in \X:~ x_{2n}=0, ~n\geq 0\}
\]
is a subspace of $\X$. So, the maximal element $\widetilde{\mathbb{W}}$ of $\mathfrak{W}$ is $\X(T)$ itself. It also follows from Example \ref{Eg: Canonical Decomposition} that $T|_{\X(T)}:\X(T)\to \X(T)$ is a Wold isometry. Naturally, $\X$ admits a decomposition as in (\ref{Levan Decomposition in Banach space}). \qed
\end{eg}

\vspace{0.2cm}

\noindent \textbf{Concluding remark.} We conclude this article here. The fact that a Banach space contraction may not always dilate to an isometry gives rise to several questions. For example, if we can characterize the Banach space contractions that dilate to isometries in terms of spectral sets for those contractions. A sequel of this article will appear shortly where we shall answer such questions. Our main focus will be on exploring the relation between spectral set and isometric dilation of a Banach space contraction.

\vspace{0.3cm}

\noindent \textbf{Acknowledgement.}  The first named author is supported by the ``Prime Minister's Research Fellowship (PMRF)" with Award No. PMRF-1302045. The second named author is supported by ``Core Research Grant" of Science and Engineering Research Board (SERB), Govt. of India, with Grant No. CRG/2023/005223 and the ``Early Research Achiever Award Grant" of IIT Bombay with Grant No. RI/0220-10001427-001. The third named author thanks IIT Bombay for providing the ``Institute Postdoctoral Fellowship (IPDF)" during the course of the paper.

\end{document}